%% file: main.tex
\documentclass[preprint]{elsarticle}       
\usepackage{amsmath, amsfonts, amssymb, amsthm}
\usepackage[headings]{fullpage}
\usepackage{graphicx}
\usepackage{multirow}
\usepackage{enumerate}
\usepackage{times}
\usepackage{tikz}
\usetikzlibrary{decorations.pathreplacing, arrows.meta, calligraphy}
\usetikzlibrary{shapes.geometric}
\usepackage{color}
\usepackage{caption}
\usepackage{subcaption}
\usepackage{soul}
\usepackage{float}

 \definecolor{green2}{rgb}{0, 0.5, 0}

 \usepackage[hidelinks,colorlinks=false]{hyperref}
 \newtheorem{theorem}{Theorem}
 \newtheorem{lemma}{Lemma}
 
 \newtheorem{remark}{Remark}

\input{latex-defs}

\journal{Computer Methods in Applied Mechanics and Engineering}

\begin{document}
\begin{frontmatter}

\title{A fast and accurate domain decomposition nonlinear manifold reduced order model}
\author{Alejandro N. Diaz\corref{cor1}\fnref{label1}}
\cortext[cor1]{Corresponding author}
\ead{and5@rice.edu}

\author{Youngsoo Choi\fnref{label2}}
\ead{choi15@llnl.gov}

\author{Matthias Heinkenschloss\fnref{label1}}
\ead{heinken@rice.edu}

\affiliation[label1]{
			organization={Department of Compuational Applied Mathematics and Operations Research, 
						  Rice University},
            city={Houston},
            postcode={77005}, 
            state={TX},
            country={United States of America}}
			
\affiliation[label2]{
			organization={Center for Applied Scientific Computing,
			              Lawrence Livermore National Laboratory},
			city={Livermore},
			postcode={94550}, 
			state={CA},
			country={United States of America}}





\begin{abstract}
This paper integrates nonlinear-manifold reduced order models (NM-ROMs) with domain decomposition (DD). 
NM-ROMs approximate the full order model (FOM) state in a nonlinear-manifold by training a shallow, sparse autoencoder using FOM snapshot data. 
These NM-ROMs can be advantageous over linear-subspace ROMs (LS-ROMs) for problems with slowly decaying Kolmogorov $n$-width. 
However, the number of NM-ROM parameters that need to be trained scales with the size of the FOM. Moreover, for ``extreme-scale" problems, the storage of high-dimensional FOM snapshots alone can make ROM training expensive. 
To alleviate the training cost, this paper applies DD to the FOM, computes NM-ROMs on each subdomain, and couples them to obtain a global NM-ROM. 
This approach has several advantages: 
Subdomain NM-ROMs can be trained in parallel, involve fewer parameters to be trained than global NM-ROMs, require smaller subdomain FOM dimensional training data, and can be tailored to subdomain-specific features of the FOM. 
The shallow, sparse architecture of the autoencoder used in each subdomain NM-ROM allows application of hyper-reduction (HR), reducing the complexity caused by nonlinearity and yielding computational speedup of the NM-ROM.
This paper provides the first application of NM-ROM (with HR) to a DD problem. 
In particular, this paper details an algebraic DD reformulation of the FOM, training a NM-ROM with HR for each subdomain, 
and a sequential quadratic programming (SQP) solver to evaluate the coupled global NM-ROM. 
Theoretical convergence results for the SQP method and {\it a priori} and {\it a posteriori} error estimates for the DD NM-ROM with HR
are provided.
The proposed DD NM-ROM with HR approach is numerically compared to a DD LS-ROM with HR on the 2D steady-state Burgers' equation, showing an order of magnitude improvement in accuracy of the proposed DD NM-ROM over the DD LS-ROM.
\end{abstract}

\end{frontmatter}

\input{intro}
\input{dd_formulation}

\input{dd_lspg}

\input{rom_port_compatibility}

\input{lsrom}
\input{nmrom}
\input{sqp_solver}
\input{autoencoder_architecture}
\input{hyper_reduction}
\input{error_analysis}

\input{numerics}

\input{conclusion}

\section*{Acknowledgements}
This work was performed at Lawrence Livermore National Laboratory. 
A.\ N.\ Diaz was supported for this work by a Defense Science and Technology Internship (DSTI) at Lawrence Livermore National Laboratory and a 2021 National Defense Science and Engineering Graduate Fellowship.
Y.\ Choi was supported for this work by the US Department of Energy under the Mathematical Multifaceted Integrated Capability Centers – DoE Grant DE –  SC0023164; The Center for Hierarchical and Robust Modeling of Non-Equilibrium Transport (CHaRMNET) and partially by LDRD (21-SI-006). 
M.\ Heinkenschloss was supported by AFOSR Grant FA9550-22-1-0004 at Rice University.
Lawrence Livermore National Laboratory is operated by Lawrence Livermore National Security, LLC, for the U.S. Department of Energy, National Nuclear Security Administration under Contract DE-AC52-07NA27344; IM release number: LLNL-JRNL-849457.

\bibliography{references, yc_references}
\bibliographystyle{elsarticle-num}

\end{document}

%% file: latex-defs.tex
\newcommand{\eref}[1]{\mbox{\rm(\ref{#1})}}
\newcommand{\seref}[2]{\mbox{\rm(\ref{#1}{#2})}}

\newcommand{\set}[1]{\left\{ #1 \right\} }
\newcommand{\norm}[1]{\left\| #1 \right\|}

\newcommand{\real}{\mathbb{R}}

\newcommand {\cD} {{\cal D}}

\newcommand {\cI} {{\cal I}}
\newcommand {\cL} {{\cal L}}

\newcommand {\cS} {{\cal S}}

\newcommand{\bzero}{\boldsymbol{0}}
\newcommand{\ba}{\boldsymbol{a}}
\newcommand{\bb}{\boldsymbol{b}}
\newcommand{\bc}{\boldsymbol{c}}

\newcommand{\bg}{\boldsymbol{g}}
\newcommand{\bh}{\boldsymbol{h}}

\newcommand{\br}{\boldsymbol{r}}
\newcommand{\bs}{\boldsymbol{s}}

\newcommand{\bu}{\boldsymbol{u}}
\newcommand{\bv}{\boldsymbol{v}}

\newcommand{\bx}{\boldsymbol{x}}
\newcommand{\by}{\boldsymbol{y}}
\newcommand{\bz}{\boldsymbol{z}}
\newcommand{\bA}{\boldsymbol{A}}
\newcommand{\bB}{\boldsymbol{B}}
\newcommand{\bC}{\boldsymbol{C}}

\newcommand{\bE}{\boldsymbol{E}}

\newcommand{\bH}{\boldsymbol{H}}
\newcommand{\bI}{\boldsymbol{I}}

\newcommand{\bP}{\boldsymbol{P}}

\newcommand{\bR}{\boldsymbol{R}}

\newcommand{\bU}{\boldsymbol{U}}
\newcommand{\bV}{\boldsymbol{V}}
\newcommand{\bW}{\boldsymbol{W}}
\newcommand{\bX}{\boldsymbol{X}}

\newcommand{\bZ}{\boldsymbol{Z}}

\newcommand{\tbb}{\widetilde{\boldsymbol{b}}}
\newcommand{\tbg}{\widetilde{\boldsymbol{g}}}

\newcommand{\tbA}{\widetilde{\boldsymbol{A}}}
\newcommand{\tbB}{\widetilde{\boldsymbol{B}}}
\newcommand{\tbC}{\widetilde{\boldsymbol{C}}}

\newcommand{\tbW}{\widetilde{\boldsymbol{W}}}

\newcommand{\hbw}{\widehat{\boldsymbol{w}}}
\newcommand{\hbx}{\widehat{\boldsymbol{x}}}

\newcommand{\hblambda}{\widehat{\boldsymbol{\lambda}}}
\newcommand{\hbF}{\widehat{\boldsymbol{F}}}

\newcommand{\obx}{\overline{\boldsymbol{x}}}

\newcommand{\oblambda}{\overline{\boldsymbol{\lambda}}}

\newcommand{\hbA}{\widehat{\boldsymbol{A}}}

\newcommand{\hbP}{\widehat{\boldsymbol{P}}}

\newcommand{\bmu}{{\boldsymbol{\mu}}}
\newcommand{\bsigma}{{\boldsymbol{\sigma}}}
\newcommand{\blambda}{{\boldsymbol{\lambda}}}

\newcommand{\brho}{{\boldsymbol{\rho}}}

\newcommand{\bPhi}{{\boldsymbol{\Phi}}}

\newcommand{\bSigma}{{\boldsymbol{\Sigma}}}

%% file: intro.tex

\section{Introduction}   \label{sec:intro}
Many applications in science and engineering require the high-fidelity numerical simulation of a parameterized, large-scale, nonlinear system, referred to as the full order model (FOM). 
For example, in the design of the airfoil of an aircraft, one repeatedly simulates the airflow around the wing to compute the lift and drag for a number of shapes to determine the optimal shape. 
Alternatively, in the case of digital twins, one simulates the high-fidelity FOM in real-time for given system inputs. 
To guarantee a high-fidelity simulation, a high-dimensional numerical model is required, resulting in high computational expense when simulating the FOM.
Consequently, both many-query and real-time applications are infeasible for large-scale problems. 
Model reduction alleviates the computational burden of simulating the high-dimensional FOM by replacing it with a low-dimensional, computationally inexpensive model, referred to as a reduced order model (ROM), that approximates the dynamics of the FOM within a tunable accuracy. 
This ROM can then be used in place of the FOM in real-time and many-query applications.
In this work, we integrate model reduction, specifically the nonlinear-manifold ROM (NM-ROM) approach, with an algebraic domain-decomposition (DD) framework. 

There are a large number of works that consider the integration of DD with model reduction. 
One family of approaches is based on the reduced basis element (RBE) method 
\cite{YMaday_EMRonquist_2002a,YMaday_EMRonquist_2004a}, 
in which reduced bases are computed locally for each subdomain.
In the RBE method, continuity of the reduced basis solution across subdomains can be enforced via Lagrange multipliers as in 
\cite{LIapichino_AQuarteroni_GRozza_2012a},
while others consider a discontinuous Galerkin approach 
\cite{PFAntonietti_PPacciarini_AQuarteroni_2016a}.
Several modifications to the RBE method have been proposed, including the so-called static condensation RBE method 
\cite{JLEftang_DBPHuynh_DJKnezevic_EMRonquist_ATPatera_2012a,DBPHuynh_DJKnezevic_ATPatera_2013a,JLEftang_ATPatera_2013a}, 
which computes a reduced basis (RB) approximation of the Schur complement and provides rigorous {\it a posteriori} error estimators.
The reduced basis hybrid method (RBHM) 
\cite{LIapichino_AQuarteroni_GRozza_2012a}
is another modification of the RBE method, 
in which a global coarse-grid solution is included in the reduced basis to ensure continuity of normal fluxes at subdomain interfaces. For RBHM, this continuity is enforced using Lagrange multipliers. 
Another well-studied approach uses the alternating Schwarz method, which decomposes the physical domain into two or more subdomains with or without overlap, and produces a global solution by iteratively solving the PDE on separate subdomains with boundary conditions coming from the state of neighboring subdomains at the previous iteration. 
The Schwarz method has been developed for both FOM-ROM and ROM-ROM couplings in, e.g., 
\cite{MBuffoni_HTelib_AIollo_2009a, JBarnett_ITezaur_AMota_2022a},
where the ROM is projection-based using Proper Orthogonal Decomposition (POD). 
The approach in \cite{AdeCastro_PBochev_PKuberry_ITezaur_2023a} also considers FOM-ROM and ROM-ROM couplings, but couple subdomain solutions using Lagrange multipliers, and compute bases such that the Schur complement system required for recovering interface solutions is nonsingular. 
The authors in \cite{AIollo_GSambataro_TTaddei_2022a} also consider a Schwarz approach, but use an optimization-based coupling that minimizes the jump between PDE state solutions on the interface of neighboring subdomains.
The authors in \cite{KSmetana_TTaddei_2023a} compute component-based ROMs based on a partition-of-unity to couple local solutions.
Others have considered using DD to compute ROMs for problems with spatially localized nonlinearities \cite{KSun_RGlowinski_MHeinkenschloss_DCSorensen_2008a}, and for use in design optimization
\cite{SMcBane_YChoi_2021a,SMcBane_YChoi_KWillcox_2022a}. 
While these approaches have been successful, they are often problem-specific. 
That is, both RBE- and Schwarz-based methods typically formulate the DD problem at the PDE level and decompose the physical domain into separate subdomains.
In contrast, the authors in \cite{CHoang_YChoi_KCarlberg_2021a} integrate DD and ROM for a general nonlinear FOM at the fully discrete level rather than the PDE level, and {\it algebraically} decompose the FOM rather than considering a decomposition of the physical domain. 
The authors then use POD to compute ROMs for each subdomain, and use an optimization-based coupling to minimize the discrete PDE residual while enforcing compatibility constraints at the interfaces.
In this paper, we extend the DD ROM framework of \cite{CHoang_YChoi_KCarlberg_2021a} to incorporate the NM-ROM approach.

We integrate NM-ROM with DD to reduce the {\it offline} computational cost required for training an NM-ROM, 
and to allow NM-ROMs to scale with increasingly large FOMs. 
Indeed, in the monolithic single-domain case, training NM-ROMs is expensive due to the high-dimensionality of the FOM training data, which results in a large number of neural network (NN) parameters requiring training. 
By coupling NM-ROM with DD, one can compute FOM training data on subdomains, thus reducing the dimensionality of subdomain NM-ROM training data, resulting in fewer parameters that need to be trained per subdomain NM-ROM. 
Furthermore, the subdomain NM-ROMs can be trained in parallel and adapted to subdomain-specific features of the FOM.
We also note that couplings of NNs and DD for solutions of partial differential equations (PDEs) have been considered in previous work 
(e.g., 
\cite{KLi_KTang_TWu_QLiao_2020a, WLi_XXiang_YXu_2020a, QSun_XXu_HYi_2023a, SLi_YXia_YLiu_QLiao_2023a}).
However, these approaches use deep learning to solve a PDE by representing its solution as a NN and minimizing a corresponding physics-informed loss function.
In contrast, our work uses autoencoders to reduce the dimensionality of an existing numerical model. 
The autoencoders are pretrained in an {\it offline} stage to find low-dimensional representations of FOM snapshot data, and used in an {\it online} stage to significantly reduce the computational cost and runtime of numerical simulations. 
Our work is the first to couple autoencoders with DD in the reduced-order modeling context.
%

A number of current model reduction approaches approximate the FOM solution in a low-dimensional linear subspace. 
In this paper, we collectively refer to this class of approaches as linear subspace ROM (LS-ROM). 
The LS-ROM approach supposes that the state solutions of the FOM are contained in a low-dimensional linear subspace. 
A basis for the linear subspace is then computed,  resulting in a ROM whose state consists of the generalized coordinates of the approximate state solution in the reduced subspace.
ROM approaches that follow LS-ROM include the 
reduced basis (RB) method 
\cite{BHaasdonk_2017a,AQuarteroni_AManzoni_FNegri_2016a},   
proper orthogonal decomposition (POD) 
\cite{MHinze_SVolkwein_2005a,MGubisch_SVolkwein_2017a,cheung2023local,copeland2022reduced,carlberg2018conservative},     
balanced truncation and balanced POD 
\cite{ACAntoulas_2005a,PBenner_TBreiten_2017a},   
interpolation and moment-matching based approaches 
\cite{ACAntoulas_CABeattie_SGugercin_2020a,CGu_2011a,PBenner_TBreiten_2015a}, 
the Loewner framework 
\cite{AJMayo_ACAntoulas_2007a,ACAntoulas_IVGosea_ACIonita_2016a,IVGosea_ACAntoulas_2018a}, and the space--time POD \cite{choi2021space,kim2021efficient,choi2019space} that expands the POD modes to temporal domain.
Although LS-ROM approaches have been successful for a number of applications,
it is well known that for advection-dominated problems and problems with sharp gradients, LS-ROM based approaches cannot produce low-dimensional subspaces where the state is well-approximated.
More precisely, LS-ROM struggles when applied to problems with slowly decaying Kolmogorov $n$-width
\cite{MOhlberger_SRave_2016a}.

In recent years, a number of model reduction approaches have been developed to address the Kolmogorov $n$-width barrier.
For example, one class of approaches leverages knowledge of the advection behavior of the given problem to enhance the approximation capabilities of linear subspaces. 
These approaches include composing transport maps with the reduced bases
\cite{NJNair_MBalajewicz_2019a,AIollo_DLombardi_2014a,NCagniart_YMaday_BStamm_2019a}, 
shifting the POD basis \cite{JReiss_PSchulze_JSesterhenn_VMehrmann_2018a},
transforming the physical domain of the snapshots \cite{GWelper_2017a}, and
computing a reduced basis for a Lagrangian formulation of the PDE 
\cite{RMojgani_MBalajewicz_2017a}.
Other approaches consider the use of multiple linear subspaces, where instead of using a global reduced basis, one constructs multiple subspaces for separate regions in the
time domain \cite{cheung2023local,copeland2022reduced,MDihlmann_MDrohmann_BHaasdonk_2011a,MDrohmann_BHaasdonk_MOhlberger_2011a}, 
physical domain \cite{TTaddei_SPerotto_AQuarteroni_2015a}, 
or state space \cite{BPeherstorfer_KWillcox_2015a, DAmsallem_MJZahr_ChFarhat_2012a}.
However, each of these approaches relies upon a substantial amount of {\it a priori} knowledge of the governing PDE in order to improve the local approximation capabilities of linear subspaces. 
In contrast, another class of methods circumvents these drawbacks by approximating the FOM solution in a low-dimensional nonlinear manifold rather than a low-dimensional linear subspace. 
While LS-ROM approaches map the low-dimensional ROM state space to the high-dimensional FOM state space via an affine mapping, 
the approaches in \cite{JBarnett_CFarhat_2022a,RGeelen_SWright_KWillcox_2022a} consider the use of quadratic manifolds, where the ROM state space is mapped to the FOM state space via a quadratic mapping. 
As a further generalization of this mapping, researchers have investigated the use of neural networks to represent general nonlinear mappings from the ROM state space to the FOM state space. 
In particular, the use of autoencoders in the context of model reduction was first considered in the papers 
\cite{KKashima_2016a,DHartman_LKMestha_2017a}. 
Autoencoders are a type of neural network that aims to learn the identity mapping by first {\it encoding} the inputs to some latent representation via the {\it encoder}, 
then {\it decoding} the latent representation to the original input space via the {\it decoder}. 
In \cite{KLee_KTCarlberg_2020a}, the authors consider the use of deep convolutional autoencoders, which augment the autoencoder architecture with convolutional layers. 
While their approach was successful in addressing the Kolmogorov $n$-width issue, the computational speedup was limited because hyper-reduction (HR) was not incorporated into their framework to properly reduce the complexity caused by nonlinear terms.
The authors in \cite{YKim_YChoi_DWidemann_TZohdi_2022a,kim2020efficient} successfully apply HR in the context of NM-ROM and achieve a considerable speed-up, and do so by choosing a shallow, wide, and sparse architecture for the autoencoder. 
The approach in \cite{FRomor_GStabile_GRozza_2023a} also incorporates HR into an NM-ROM approach, but do so by employing a teacher-student training approach, where an autoencoder is first trained to reduce the entire state, and a second decoder is trained to only reproduce the HR nodes. 
This approach also permits the use of more general autoencoder architectures than the shallow, wide, and sparse architectures of \cite{YKim_YChoi_DWidemann_TZohdi_2022a,kim2020efficient}.
However, an advantage of the approach in \cite{YKim_YChoi_DWidemann_TZohdi_2022a,kim2020efficient} is that autoencoder training only happens once rather than requiring a teacher-student training approach. 
Furthermore, this approach allows for different choices of HR nodes after NM-ROM training, whereas the approach in \cite{FRomor_GStabile_GRozza_2023a} requires fixed HR nodes. 

In this paper, we extend the work of \cite{CHoang_YChoi_KCarlberg_2021a} on DD LS-ROM and integrate the NM-ROM approach with HR discussed in \cite{YKim_YChoi_DWidemann_TZohdi_2022a}. 
We incorporate the NM-ROM approach into this framework because of its success when applied to problems with slowly decaying Kolmogorov $n$-width. 
Specifically, to build ROMs on each subdomain of the DD problem, we apply NM-ROM with HR by using wide, shallow, sparse-masked autoencoders. 
The wide, shallow, and sparse architecture allows for hyper-reduction to be efficiently applied, thus reducing the complexity caused by nonlinearity and yielding computational speedup. 
Additionally, we modify the wide, shallow, and sparse architecture used in \cite{YKim_YChoi_DWidemann_TZohdi_2022a} to also include a sparsity mask for the encoder input layer as well as the decoder output layer. 
The sparsity mask at the encoder input layer results in an architecture that is symmetric across the latent layer of the autoencoder. 
Using {\it sparse} linear layers also allows one to make the encoders and decoders very wide while keeping memory costs low.
Integrating NM-ROM with DD allows one to compute the FOM training snapshots on subdomains, thus significantly reducing the number of NN parameters requiring training for each subdomain. 

A summary of the key contributions from this paper are as follows. 
\begin{itemize}
    \item We develop the first application of NM-ROM with HR to a DD problem. 
    \item We modify the autoencoder architecture discussed in \cite{YKim_YChoi_DWidemann_TZohdi_2022a} to also include sparsity in the encoder input layer as well as the decoder output layer. 
    \item We develop an inexact Lagrange-Newton sequential quadratic programming (SQP) method for the DD NM-ROM, and provide a theoretical convergence result for the SQP solver. 
    \item We provide {\it a priori} and {\it a posteriori} error estimates for the DD ROM which are valid for both LS-ROM and NM-ROM. 
    \item We numerically compare DD LS-ROM with DD NM-ROM, both with and without HR, for a number of different problem configurations using the 2D steady-state Burgers' equation.
\end{itemize} 

This paper is structured as follows. Section \ref{sec:dd_formulation} discusses the algebraic DD FOM formulation that we consider. 
Section \ref{sec:dd_lspg} discusses the constrained least-squares Petrov-Galerkin (LSPG) formulation for the ROM, which respects the DD FOM formulation.
We then review the LS-ROM approach based on POD in Section \ref{sec:lsrom}, and detail the NM-ROM approach in Section \ref{sec:nmrom}.
We develop an inexact Lagrange-Newton sequential quadratic programming (SQP) method for the constrained LSPG-ROM in Section \ref{sec:sqp_solver}, followed by standard theoretical convergence results for the SQP solver in Section \ref{sec:sqp_convergence}.
We then discuss the autoencoder architecture used in Section \ref{sec:autoencoder_architecture},
the application of hyper-reduction in Section \ref{sec:hyper_reduction}, 
and the construction of a HR subnet in Section \ref{sec:subnet}.
In Section \ref{sec:error_analysis}, we provide both {\it a posteriori} and {\it a priori} error bounds for the ROM solution in Theorems \ref{thm:a_posteriori_bound} and \ref{thm:a_priori_bound}, respectively. 
We numerically compare the DD LS-ROM and DD NM-ROM performance, both with and without HR, on the 2D steady-state Burgers' equation in Section \ref{sec:numerics} for a number of different problem configurations.
Lastly, we conclude the paper and discuss future directions in Section \ref{sec:conclusion}.

%% file: dd_formulation.tex

\section{Domain-decomposition FOM formulation} \label{sec:dd_formulation}
This section presents the algebraic domain-decomposition formulation 
\cite{CHoang_YChoi_KCarlberg_2021a}.
We consider a FOM parameterized by $\bmu \in \cD \subset \real^{N_\mu}$. 
Given $\bmu \in \cD$, the FOM is expressed as a parametrized system of nonlinear 
algebraic equations
\begin{equation}\label{eq:fom_residual}
    \br(\bx(\bmu); \bmu) = \bzero,
\end{equation}
where $\br: \real^{N_x}\times \cD \to \real^{N_x}$ denotes the residual and
$\bx(\bmu) \in \real^{N_x}$ denotes the state.
For notational simplicity, the dependence on $\bmu$ is suppressed until needed.
Typically $\br$ corresponds to a discretized PDE (e.g., using finite differences or finite elements)
and in our target applications $\br$ is nonlinear in $\bx$.

Next we decompose the system \eref{eq:fom_residual} into $n_\Omega$ {\em algebraic subdomains}.
Before giving the technical details, we describe the decomposition using a simple example
illustrated in Figure~\ref{fig:dd_example_2D}. 
In this example, suppose the system \eref{eq:fom_residual} is obtained 
from a finite difference discretization with a 5-point stencil  of a scalar PDE in a rectangular domain in $\real^2$
with Dirichlet boundary conditions.
The situation would be similar if the PDE was discretized using linear finite elements on a  regular grid.
Moreover, the decomposition into algebraic subdomains is not limited to the finite difference discretization 
with a 5-point stencil; this discretization is simply used for illustration.
In the left plot in Figure~\ref{fig:dd_example_2D}, the finite difference discretization uses 
a $14\times 5$ grid of $N_x = 70$ nodes. Each node corresponds to a component in 
the vector of unknown states $\bx$ and an  equation in the system \eref{eq:fom_residual}. 
The domain is subdivided into $n_\Omega =2$ subdomains. Nodes near the interface between the
two subdomains are marked by filled circles.
 Because the PDE is discretized using a  5-point stencil, residuals corresponding to nodes 
 marked by filled circles in
subdomain 1 depend on state variables corresponding to nodes marked by filled circles in subdomain 2. Similarly,
residuals corresponding to nodes marked by filled circles in subdomain 2 depend on state variables corresponding to  
nodes marked by filled circles in subdomain 1.
In the right plot in Figure~\ref{fig:dd_example_2D}, these variables are duplicated as components of
 the vectors of the {\em interface states} 
$\bx_1^\Gamma$ and $\bx_2^\Gamma$.  These have to satisfy $\bx_1^\Gamma = \bx_2^\Gamma$
and this compatibility condition will later be enforced via constraints.
The other state variables are the {\em interior states}  $\bx_i^\Omega$, $i =1,2$.
These are the state variables that only enter the residuals corresponding to subdomain $i$.
Next we provide a detailed description of the general case.

\begin{figure}[!htb]
   \begin{center}
	\begin{tikzpicture}[scale=0.85]	

            \fill[fill=red, fill opacity=0.2 ]  (0,0) rectangle (3.75,3);
            \fill[fill=blue, fill opacity=0.2 ]  (3.75,0) rectangle (7.5,3);
            
             \foreach \x in {1,..., 14}
    		\foreach \y in {1, 2, 3, 4, 5}
      			{
			 \draw (0.5*\x,0.5*\y) circle (2pt);
      			}
	    \foreach \x in {7, 8}
    		\foreach \y in {1, 2, 3, 4, 5}
      			{
			  \fill[black] (0.5*\x,0.5*\y) circle (2pt);
      			}
             \draw[decorate, decoration={calligraphic brace, amplitude=5pt}] (0, 3.2) -- (3.75, 3.2) node[pos=0.5, above=8pt, black]{$\br_1$};
            \draw[decorate, decoration={calligraphic brace, amplitude=5pt}] (3.75, 3.2) -- (7.5, 3.2) node[pos=0.5, above=8pt, black]{$\br_2$};
	
	    \fill[fill=blue, fill opacity=0.0 ]  (0.0,-1.1) rectangle (0.1,0.1); 

        \end{tikzpicture} \hfill
        \begin{tikzpicture}[scale=0.85]	

            \fill[fill=red, fill opacity=0.2 ]  (0,0) rectangle (4.3,3);
            \fill[fill=red, fill opacity=0.2 ]  (3.25,0) rectangle (4.3,3);
                
        	      \foreach \x in {1,..., 8}
    		\foreach \y in {1, 2, 3, 4, 5}
      			{
			 \draw (0.5*\x,0.5*\y) circle (2pt);
      			}
	    \foreach \x in {7, 8}
    		\foreach \y in {1, 2, 3, 4, 5}
      			{
			  \fill[black] (0.5*\x,0.5*\y) circle (2pt);
      			}
            \draw[decorate, decoration={calligraphic brace, amplitude=5pt}] (0, 3.2) -- (3.75, 3.2) node[pos=0.5, above=8pt, black]{$\br_1$};
            
            \draw[decorate, decoration={calligraphic brace, mirror, amplitude=5pt}] (0, -0.2) -- (3.25, -0.2) node[pos=0.5, below=4pt, black]{$\bx_1^\Omega$};
           \draw[decorate, decoration={calligraphic brace, mirror, amplitude=5pt}] (3.25, -0.2) -- (4.3, -0.2) node[pos=0.5, below=4pt, black]{$\bx_1^\Gamma$};

            \fill[fill=blue, fill opacity=0.2 ]  (5.2,0) rectangle (9.5,3);
            \fill[fill=blue, fill opacity=0.2 ]  (5.2,0) rectangle (6.2,3);
            
             \foreach \x in {8, ..., 14}
    		\foreach \y in {1, 2, 3, 4, 5}
      			{
			 \draw (2+0.5*\x,0.5*\y) circle (2pt);
      			}
	    \foreach \x in {7, 8}
    		\foreach \y in {1, 2, 3, 4, 5}
      			{
			  \fill[black] (2+0.5*\x,0.5*\y) circle (2pt);
      			}
	            
            \draw[decorate, decoration={calligraphic brace, amplitude=5pt}] (5.7, 3.2) -- (9.5, 3.2) node[pos=0.5, above=8pt, black]{$\br_2$};
            \draw[decorate, decoration={calligraphic brace, mirror, amplitude=5pt}] (6.2, -0.2) -- (9.5, -0.2) node[pos=0.5, below=4pt, black]{$\bx_2^\Omega$};
           \draw[decorate, decoration={calligraphic brace, mirror, amplitude=5pt}] (5.2, -0.2) -- (6.2, -0.2) node[pos=0.5, below=4pt, black]{$\bx_2^\Gamma$};
           
       	\end{tikzpicture} \hfill
   \caption{Left plot: Each node in the domain corresponds to an unknown $\bx$ and an equation in the 
                  system \eref{eq:fom_residual}. The domain is subdivided into $n_\Omega =2$ subdomains.
                  Residuals corresponding to nodes marked by filled circles near the bboundary in subdomain 1 depend on 
                  nodes  marked by filled circles in subdomain 2, 
                  and residuals corresponding to nodes  marked by filled circles near the boundary in subdomain 2 depend on 
                  nodes  marked by filled circles in subdomain 1.
                  Right plot: Variables that enter computations of  residuals in one or more subdomains are duplicated as  interface
                  state variables $\bx_1^\Gamma$ and $\bx_2^\Gamma$. Variables that only enter the computations of 
                  the  residuals in one subdomains are the interior  state variables $\bx_1^\Omega$ and $\bx_2^\Omega$,
                  respectively. Equality $\bx_1^\Gamma = \bx_2^\Gamma$ of  interface state variables will be enforced via constraints.
                  \label{fig:dd_example_2D}}
       \end{center}
 \end{figure}
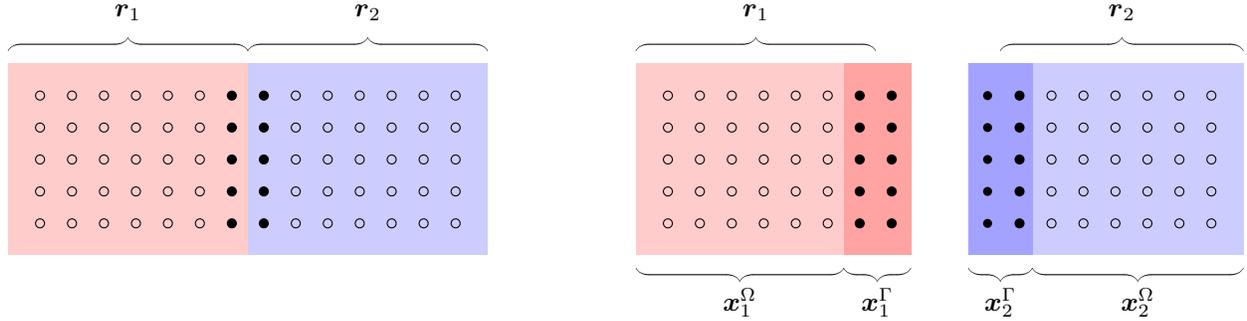

We decompose the system \eref{eq:fom_residual} into $n_\Omega \leq N_x$ {\em algebraic subdomains} by defining so-called residual sampling matrices $\bP_i^r\in\set{0,1}^{N_i^r\times N_x}$ 
and computing subdomain residuals as 
\[
          \bP_i^r\br(\bx)\in \real^{N_i^r}, \quad i=1,\dots, n_\Omega.
\]
The residual sampling matrices are assumed to be {\it algebraically non-overlapping}, i.e.
\[
        \bP_i^r(\bP_j^r)^T = \bzero, \quad \forall \; i\neq j,
\]
and $\sum_{i=1}^{n_\Omega} N_i^r = N_x$.
For problems \eref{eq:fom_residual} arising from a PDE discretization,
the sparsity structure of the monolithic residual function $\br$ implies that subdomain residuals 
$\bP_i^r\br(\bx)$ only depend on a subset of the full state $\bx$. 
Furthermore, the residual corresponding to points at the boundary of subdomain $i$ depend on the state 
$\bx$ at points within subdomain $i$ and at points that belong to neighboring subdomains. 
Therefore, for subdomain $i$, we decompose the state components into  {\it interior states} 
\begin{subequations}  \label{eq:subdomain-states}
\begin{equation}
                \bx_i^\Omega:= \bP_i^\Omega \bx \in \real^{N_i^\Omega}
\end{equation}
and \ {\it interface states} 
\begin{equation} 
       \bx_i^\Gamma := \bP_i^\Gamma \bx \in \real^{N_i^\Gamma},
\end{equation}
\end{subequations}
where
$\bP_i^\Omega \in \set{0, 1}^{N_i^\Omega \times N_x}$ denotes the $i$th interior-state sampling matrix and
$\bP_i^\Gamma \in \set{0, 1}^{N_i^\Gamma \times N_x}$ denotes the $i$th interface-state sampling matrix.
The interior states $\bx_i^\Omega:= \bP_i^\Omega \bx$ only enter the evaluation of the $i$th subdomain
residual $\bP_j^r \br(\bx)$. 
The interface states $\bx_i^\Gamma := \bP_i^\Gamma \bx$ also enter the evaluation of another  subdomain
residual $\bP_j^r \br(\bx)$, $j \not= i$.
Since the $i$th interior states only enter the evaluation of the $i$th subdomain,  the interior-state sampling 
matrices are algebraically non-overlapping,
\[
           \bP_i^\Omega (\bP_j^\Omega)^T = \bzero, \quad \forall \; i \neq j.
\]
The interface state variables are duplicated across one or more subdomains, and we will describe later
how to enforce equality among duplicated  interface state variables.

With these specifications we can now define subdomain residual functions 
$\br_i:\real^{N_i^\Omega}\times \real^{N_i^\Gamma}\to \real^{N_i^r}$ as
\begin{equation}\label{eq:subdomain_residual_def}
  \br_i(\bx_i^\Omega, \bx_i^\Gamma) 
  = \bP_i^r\br\Big( \big(\bP_i^\Omega\big)^T \bx_i^\Omega + \big(\bP_i^\Gamma \big)^T\bx_i^\Gamma \Big).
\end{equation}
Furthermore, the monolithic residual function  \eref{eq:fom_residual} can be decomposed as 
\begin{equation}\label{eq:fom_dd_residual}
  \br(\bx) = \sum_{i=1}^{n_\Omega} \left(\bP_i^r\right)^T \br_i(\bP_i^\Omega \bx, \bP_i^\Gamma \bx), \qquad \forall \; \bx  \in \real^{N_x}.
\end{equation}
Equations \eref{eq:fom_residual}, \eref{eq:subdomain_residual_def}, and \eref{eq:fom_dd_residual} imply that the solution 
\eref{eq:subdomain-states} of \eref{eq:fom_residual} restricted to the $i$th subdomain  satisfies
\begin{equation}\label{eq:fom_subdomain_residual}
  \br_i(\bx_i^\Omega, \bx_i^\Gamma)=\bzero, \quad i=1, \dots, n_\Omega.
\end{equation}

In addition to \eref{eq:fom_subdomain_residual}, 
compatibility conditions must be imposed that enforce equality between overlapping interface states for neighboring subdomains.
These compatibility conditions are enforced by defining  $n_p$ non-overlapping {\em ports}.
Geometrically, the $j$th port is a subset of subdomains. The $j$th port has $N_j^p \leq N_x$ overlapping interface-state variables.
The indices of subdomains that intersect with  the $j$th port are $P(j) \subseteq \set{1, \dots, n_\Omega}$.
Figure \ref{fig:dd_res_states} displays the ports for a $4$-subdomain example configuration.

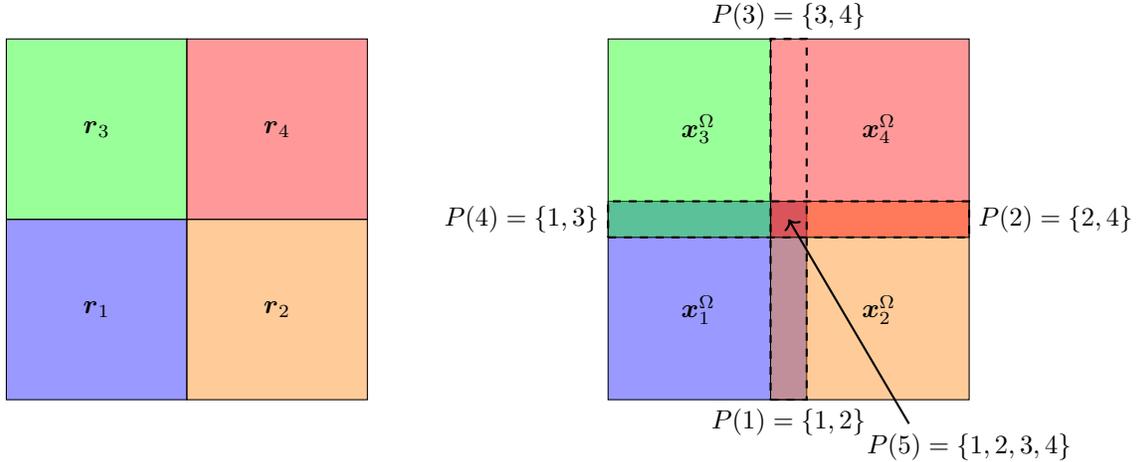
\begin{figure}[H]
  \centering
  
  	\begin{tikzpicture}[scale=0.8]	
            \draw[fill=blue, fill opacity=0.4 ]  (0,0) -- (3,0) -- (3,3) -- (0,3) -- cycle;
            \draw  (1.5,1.5) node{$\br_1$};
            \draw[fill=orange, fill opacity=0.4 ]  (3,0) -- (6,0) -- (6,3) -- (3,3) -- cycle;
            \draw  (4.5,1.5) node{$\br_2$};
            \draw[fill=green, fill opacity=0.4 ]  (0,3) -- (0,6) -- (3,6) -- (3,3) -- cycle;
            \draw  (1.5,4.5) node{$\br_3$};
            \draw[fill=red, fill opacity=0.4 ]  (3,3) -- (6,3) -- (6,6) -- (3,6) -- cycle;
            \draw  (4.5,4.5) node{$\br_4$};
                       
            \draw[fill=blue, fill opacity=0.4 ]  (10,0) -- (13.3,0) -- (13.3,3.3) -- (10,3.3) -- cycle;
            \draw  (11.5,1.5) node{$\bx_1^\Omega$};
            \draw[fill=orange, fill opacity=0.4 ]  (12.7,0) -- (16,0) -- (16,3.3) -- (12.7,3.3) -- cycle;
            \draw  (14.5,1.5) node{$\bx_2^\Omega$};
            \draw[fill=green, fill opacity=0.4 ]  (10,2.7) -- (10,6) -- (12.7,6) -- (12.7,2.7) -- cycle;
            \draw  (11.5,4.5) node{$\bx_3^\Omega$};
            \draw[fill=red, fill opacity=0.4 ]  (12.7,2.7) -- (16,2.7) -- (16,6) -- (12.7,6) -- cycle;
            \draw  (14.5,4.5) node{$\bx_4^\Omega$};
            
            \draw[dashed, thick]  (10,2.7) -- (16,2.7) -- (16,3.3) -- (10,3.3) -- cycle;
            \draw[dashed, thick]  (12.7,0) -- (13.3,0) -- (13.3,6) -- (12.7,6) -- cycle;
            \draw  (13,0) node[below]{$P(1)=\{ 1, 2\}$};
            \draw  (16,3) node[right]{$P(2)=\{2, 4\}$};
            \draw  (13,6) node[above]{$P(3)=\{3, 4\}$};
            \draw  (10,3) node[left]{$P(4)=\{1,3\}$};
            \draw  (16,-0.4) node[below]{$P(5)=\{1, 2, 3, 4\}$};
            \draw[ ->, thick]  (15,-0.4) -- (13,3);
            
        \end{tikzpicture} \hfill
  \caption{
      Left: Residual decomposition using $4$ subdomains. Notice that the residuals do not overlap.
      Right: State decomposition. 
      The regions without overlap correspond to interior states $\bx_i^\Omega$ while regions with overlap correspond to interface 
      states $\bx_i^\Gamma$. 
      The overlapping regions enclosed by black dashed lines represent the ports $P(j)\subset\set{1, \dots, n_\Omega}$.
  }
  \label{fig:dd_res_states}
\end{figure}

Using the ports, the compatibility conditions can be expressed as
\begin{equation}\label{eq:compatibility_conditions}
  \bP_i^j \bx_i^\Gamma = \bP_\ell^j \bx_\ell^\Gamma, \quad i, \ell \in P(j), \; j = 1, \dots, n_p,
\end{equation}
where  $\bP_i^j \in \set{0, 1}^{N_j^p \times N_i^\Gamma}$ denotes the $j$th port sampling matrix for subdomain $i$.
Because the ports are non-overlapping, if $Q(i):=\set{j \; | \: i \in P(j)}$ is the set of ports associated with subdomain $i$,
then 
\begin{align}\label{eq:fom_port_nonoverlapping}
    \bP_i^j(\bP_i^\ell)^T=\bzero, \qquad \forall \; j,\, \ell \in Q(i), \; j \neq \ell,
\end{align}
and the sum of numbers of variables in the ports  associated with subdomain $i$ is equal to the
number of interface variables in the $i$th subdomain, $\sum_{j\in Q(i)} N_j^p = N_i^\Gamma$.

As written, most conditions in  \eref{eq:compatibility_conditions} are redundant. Instead, for port $j$ one needs 
$(|P(j)|-1)N_j^p$ conditions, where $|P(j)|$ denotes cardinality of $P(j)$.  
For example, in Figure \ref{fig:dd_res_states} one needs the 
conditions $\bP_1^1 \bx_1^\Gamma = \bP_2^1 \bx_2^\Gamma$ for the first port $P(1) = \{1, 2\}$,
and the conditions $\bP_1^5 \bx_1^\Gamma = \bP_2^5 \bx_2^\Gamma$, $\bP_2^5 \bx_2^\Gamma = \bP_3^5 \bx_3^\Gamma$, 
$\bP_3^5 \bx_3^\Gamma = \bP_4^5 \bx_4^\Gamma$ for the fifth port $P(5) = \{1, 2, 3, 4\}$.
Removing redundant conditions  \eref{eq:compatibility_conditions}, the 
port compatibility conditions \eref{eq:compatibility_conditions} can be written as
\begin{equation}\label{eq:compatibility_conditions_mat}
        \sum_{i=1}^{n_\Omega} \bA_i \bx_i^\Gamma = \bzero,
\end{equation}
where the $\bA_i\in \set{-1, 0, 1}^{N_{A}\times N_i^\Gamma}$
denote the constraint matrices associated with the port compatibility conditions \eref{eq:compatibility_conditions} and
the total number of compatibility conditions is  $N_{A} = \sum_{j=1}^{n_p}(|P(j)|-1)N_j^p$.
The matrix $(\bA_1, \ldots, \bA_{n_\Omega})$ has full row rank.


In summary, the algebraic DD formulation of the FOM \eref{eq:fom_residual} is given by
\begin{subequations}\label{eq:fom_dd}
  \begin{align}
    \br_i(\bx_i^\Omega, \bx_i^\Gamma) &= \bzero, \qquad i = 1, \dots, n_\Omega, \\
   \sum_{i=1}^{n_\Omega} \bA_i \bx_i^\Gamma &= \bzero.
  \end{align}
\end{subequations}
Associated with \eref{eq:fom_dd} we also consider the nonlinear least-squares problem with equality constraints,
\begin{subequations}\label{eq:fom_dd_nlp}
  \begin{align}
    \min_{(\bx_i^\Omega, \bx_i^\Gamma), i=1, \dots, n_\Omega} \quad &
    \frac{1}{2}\sum_{i=1}^{n_\Omega} \norm{\br_i\left(\bx_i^\Omega, \bx_i^\Gamma\right)}_2^2 \\
    {\rm s.t.} \quad & \sum_{i=1}^{n_\Omega} \bA_i \bx_i^\Gamma = \bzero.
  \end{align}
\end{subequations}

The connections between the formulations \eref{eq:fom_residual}, \eref{eq:fom_dd}, and  \eref{eq:fom_dd_nlp}
are summarized in the following theorem.

\begin{theorem}   \label{th:connection-FOM-DDFOM}
     If $\bx$ solves the FOM \eref{eq:fom_residual} then $(\bx_i^\Omega, \bx_i^\Gamma)$
     with   $\bx_i^\Omega := \bP_i^\Omega \bx$ and $\bx_i^\Gamma := \bP_i^\Gamma \bx$,   $i=1, \dots, n_\Omega$,
     solves the  algebraic DD formulation of the FOM \eref{eq:fom_dd} and vice versa.
     A solution of  \eref{eq:fom_dd} also solves \eref{eq:fom_dd_nlp}, and a solution of \eref{eq:fom_dd_nlp}
     with objective function value equal to zero solves \eref{eq:fom_dd}.
\end{theorem}
The proof of Theorem~\ref{th:connection-FOM-DDFOM} follows immediately from the construction of \eref{eq:fom_dd}.

In the FOM context, the constrained nonlinear least-squares problem formulation \eref{eq:fom_dd_nlp} is not needed,
but we include it here because it will become important for the model reduction derivation in Section~\ref{sec:dd_lspg},
where we use a least squares formulation for the subdomain ROMs. 
The constrained nonlinear least-squares formulation  \eref{eq:fom_dd_nlp} of the FOM could be solved
using a Lagrange-Newton sequential quadratic programming (SQP) method with Gauss-Newton Hessian approximation
and the constrained nonlinear least-squares formulation of the NM-ROM corresponding to \eref{eq:fom_dd_nlp} will
be solved using the Lagrange-Newton SQP method discussed in Section~\ref{sec:sqp_solver}.

In principle, an alternative DD formulation of the FOM is possible, which reverses the role of satisfying the subdomain
equations and of the compatibility conditions. Instead of  \eref{eq:fom_dd}, one can impose  the subdomain
equations $\br_i(\bx_i^\Omega, \bx_i^\Gamma) = \bzero$, $i = 1, \dots, n_\Omega$, as constraints and
use a least squares formulation of the compatibility conditions as the objective. This is used, e.g., in
\cite{AIollo_GSambataro_TTaddei_2022a,TTaddei_XXu_LZhang_2023a}. 
However, since we use LSPG-ROMs, which in general do not have zero
residual, these cannot be incorporated as equality constraints. In contrast, the formulation \eref{eq:fom_dd_nlp} can
be used with subdomain  LSPG-ROMs, as we will describe in the next section.

%% file: dd_lspg.tex

\section{Domain-decomposition ROM}   \label{sec:dd_lspg}
The ROM construction is built on the assumption that the high-dimensional subdomain state variables $\bx_i^\Omega \in \real^{N_i^\Omega}$ and
$\bx_i^\Gamma \in \real^{N_i^\Gamma}$, $i = 1, \ldots, n_\Omega$, can be approximated using low-dimensional variables
$\hbx_i^\Omega \in \real^{n_i^\Omega}$, $n_i^\Omega \ll N_i^\Omega$ and
$\hbx_i^\Gamma \in \real^{N_i^\Gamma}$, $n_i^\Gamma \ll N_i^\Gamma$,  $i = 1, \ldots, n_\Omega$, respectively.
Specifically, we assume that for each subdomain $i$ there exist maps $\bg_i^\Omega: \real^{n_i^\Omega} \to \real^{N_i^\Omega}$ and 
$\bg_i^\Gamma: \real^{n_i^\Gamma} \to \real^{N_i^\Gamma}$ such that
\begin{align}\label{eq:rom_approximations}
    \bx_i^\Omega \approx \bg_i^\Omega(\hbx_i^\Omega), \qquad
    \bx_i^\Gamma \approx \bg_i^\Gamma(\hbx_i^\Gamma), \qquad i = 1, \dots, n_\Omega.
\end{align}
In the traditional LS-ROM the maps $\bg_i^\Omega$ and $\bg_i^\Gamma$ are linear, 
whereas in our NM-ROM these maps are computed via autoencoders/decoders.
Assuming that we have maps $\bg_i^\Omega$ and $\bg_i^\Gamma$ such that  \eref{eq:rom_approximations} holds,
we discuss how to construct the ROM in Section~\ref{sec:DD-LSPG-ROM}. 
Specifically, our ROM is based on the constrained nonlinear least-squares formulation \eref{eq:fom_dd_nlp} of the FOM.
One issue in the ROM construction based on  \eref{eq:fom_dd_nlp} is the formulation of compatibility constraints for the ROM.
In  Section~\ref{sec:DD-LSPG-ROM} we use a formulation following  \cite{CHoang_YChoi_KCarlberg_2021a} and
in Section \ref{sec:rom_port_compatibility} we provide an alternative formulation of the ROM compatibility constraints 
by constructing the maps $\bg_i^\Gamma(\hbx_i^\Gamma)$ in a suitable way.
The detailed construction of maps $\bg_i^\Omega$ and 
$\bg_i^\Gamma$ such that \eref{eq:rom_approximations} holds 
is discussed in Sections \ref{sec:lsrom}, \ref{sec:nmrom}, and \ref{sec:autoencoder_architecture}.
Specifically, we will review  the traditional LS-ROM in Section \ref{sec:lsrom}.
Sections \ref{sec:nmrom} and \ref{sec:autoencoder_architecture} discuss how to compute these maps via the NM-ROM approach.

\subsection{Least-squares formulation}  \label{sec:DD-LSPG-ROM}
Given maps $\bg_i^\Omega$ and $\bg_i^\Gamma$ such that \eref{eq:rom_approximations} holds, a naive way of computing the ROM
is to simply replace $\bx_i^\Omega$ and $\bx_i^\Gamma$ in the constrained nonlinear least-squares formulation  
\eref{eq:fom_dd_nlp} of the FOM by $\bg_i^\Omega(\hbx_i^\Omega)$ and $\bg_i^\Gamma(\hbx_i^\Gamma)$.
An evaluation of this ROM requires the solution of 
\begin{subequations}\label{eq:dd_lspg_NLP0}
  \begin{align}
    \min_{(\hbx_i^\Omega, \hbx_i^\Gamma), i=1, \dots, n_\Omega} \quad &
    \frac{1}{2}\sum_{i=1}^{n_\Omega} \norm{ \br_i\left(\bg_i^\Omega\big(\hbx_i^\Omega\big), \bg_i^\Gamma\big(\hbx_i^\Gamma \big)\right)}_2^2 \\
    {\rm s.t.} \quad & \sum_{i=1}^{n_\Omega}  \bA_i \bg_i^\Gamma(\hbx_i^\Gamma) = \bzero.
  \end{align}
\end{subequations}
This corresponds to a (naive) LSPG-ROM.
There are two issues with this formulation. 

The first issue is that, just as in the case of the classical LSPG-ROM, the complexity of the evaluation of the subdomain residuals, i.e.,
$\big(\hbx_i^\Omega, \hbx_i^\Gamma \big) \to$  
$\big( \bg_i^\Omega(\hbx_i^\Omega), \bg_i^\Gamma(\hbx_i^\Gamma) \big) \to$ $\br_i\big(\bg_i^\Omega\big(\hbx_i^\Omega\big), \bg_i^\Gamma\big(\hbx_i^\Gamma \big) \big)$
scales with the size $N_i^\Omega$ and $N_i^\Gamma$ of the FOM.  
This issue is addressed using so-called hyper-reduction (HR). 
See, e.g., \cite{CFarhat_SGrimberg_AManzoni_AQuarteroni_2021a} for an overview. 
HR replaces the residual $\br_i\big(\bg_i^\Omega\big(\hbx_i^\Omega\big), \bg_i^\Gamma\big(\hbx_i^\Gamma \big) \big)$ in (\ref{eq:dd_lspg_NLP0}a)
by $\bB_i \br_i\big(\bg_i^\Omega\big(\hbx_i^\Omega\big), \bg_i^\Gamma\big(\hbx_i^\Gamma \big) \big)$, where
 $\bB_i \in \real^{N_i^B \times N_i^r}$, $N_i^B \leq N_i^r$, is determined  by the HR approach.
For example, $\bB_i = \bI$ corresponds to vanilla LSPG, 
$\bB_i = \bZ_i$, where $\bZ_i \in \set{0, 1}^{N_i^z\times N_i^r}$, $N_i^z < N_i^r$, denotes a row-sampling matrix, corresponds to collocation HR, 
and $\bB_i = (\bZ_i \bPhi_i^r)^\dag \bZ_i$, where $\bZ_i$ is as before, $\bPhi_i^r \in \real^{N_i^r \times n_i^r}, i=1, \dots, n_\Omega$, 
denotes a reduced subspace for the corresponding subdomain residual and the superscript $\dag$ denotes the Moore-Penrose pseudoinverse,
corresponds to gappy POD HR \cite{REverson_LSirovich_1995,choi2020sns,lauzon2022s}.
Further details on HR for our DD NM-ROM are discussed in Section \ref{sec:hyper_reduction}. 
For the application of HR to DD LS-ROM, we refer the reader to \cite{CHoang_YChoi_KCarlberg_2021a}.

The second issue with \eref{eq:dd_lspg_NLP0} is that it involves the same number of constraints  (\ref{eq:dd_lspg_NLP0}b) as the FOM
\eref{eq:fom_dd_nlp}, but fewer degrees of freedom to satisfy them. 
In the extreme case, it may be impossible to satisfy the constraints (\ref{eq:dd_lspg_NLP0}b).
One approach, following \cite{CHoang_YChoi_KCarlberg_2021a}, is to replace $\bA_i$ in  (\ref{eq:dd_lspg_NLP0}b) by $\bC \bA_i$, 
where $\bC\in \real^{n_C\times N_{A}}$, $\; n_C \ll N_A$, is a test matrix that converts (\ref{eq:dd_lspg_NLP0}b) into a 
so-called ``weak compatibility constraint". 
We will choose $\bC$ to be a Gaussian matrix, but in principle other choices of $\bC$ can be used.

To summarize, given maps $\bg_i^\Omega: \real^{n_i^\Omega} \to \real^{N_i^\Omega}$ and 
$\bg_i^\Gamma: \real^{n_i^\Gamma} \to \real^{N_i^\Gamma}$  such that \eref{eq:rom_approximations} holds,
given HR matrices $\bB_i \in \real^{N_i^B \times N_i^r}$, $N_i^B \leq N_i^r$, $i = 1, \ldots, n_\Omega$, and given
$\bC\in \real^{n_C\times N_A}$, $\; n_C \ll N_A$, our DD-LSPG-ROM is evaluated by solving
\begin{subequations}\label{eq:dd_lspg_NLP}
  \begin{align}
    \min_{(\hbx_i^\Omega, \hbx_i^\Gamma), i=1, \dots, n_\Omega} \quad &
    \frac{1}{2}\sum_{i=1}^{n_\Omega} \norm{\bB_i \br_i\left(\bg_i^\Omega\left(\hbx_i^\Omega\right), \bg_i^\Gamma\left(\hbx_i^\Gamma \right)\right)}_2^2 \\
    {\rm s.t.} \quad & \sum_{i=1}^{n_\Omega} \bC \bA_i \bg_i^\Gamma(\hbx_i^\Gamma) = \bzero.
  \end{align}
\end{subequations}
The DD-LSPG-ROM formulation \eref{eq:dd_lspg_NLP} will be referred to as the weak FOM-port constraint (WFPC) formulation.

While the FOM \eref{eq:fom_dd} or \eref{eq:fom_dd_nlp} has linear constraints,
the WFPC formulation has nonlinear constraints in general.
Corresponding to $\bg_i^\Gamma: \real^{n_i^\Gamma} \to \real^{N_i^\Gamma}$ is a function
$\bh_i^\Gamma: \real^{N_i^\Gamma} \to \real^{n_i^\Gamma}$ such that
$\big\| \bg_i^\Gamma\big( \bh_i^\Gamma\big( \bx_i^{\Gamma,\rm train} \big) \big) - \bx_i^{\Gamma,\rm train} \big\|$
is small for some training/snapshot data  $\bx_i^{\Gamma,\rm train}$, $i = 1, \ldots, n_\Omega$,
that satisfy the linear FOM constraints (\ref{eq:fom_dd}b).
See Sections ~\ref{sec:lsrom} and \ref{sec:nmrom}.
Thus $\sum_{i=1}^{n_\Omega} \bC \bA_i \bg_i^\Gamma(\hbx_i^\Gamma)$ is guaranteed to be small
at these training/snapshot data. Existence of points that satisfy (\ref{eq:dd_lspg_NLP}b) in the nonlinear
case is still an open issue. However, in our numerical examples we have not observed any issues related to existence
of feasible points for \eref{eq:dd_lspg_NLP}.
The existence of solutions of \eref{eq:dd_lspg_NLP} can be guaranteed under mild conditions that are typical
for optimization problems.
\begin{theorem}   \label{th:dd_lspg_NLP_solution}
     Let $\widetilde{\bx}_i^\Gamma$, $i = 1, \ldots, n_\Omega$, satisfy the constraints (\ref{eq:dd_lspg_NLP}b)
     and let $\widetilde{\bx}_i^\Omega$, $i = 1, \ldots, n_\Omega$, be arbitrary.
     If the residual function $\br$ and the maps $\bg_i^\Omega$, $\bg_i^\Gamma$, $i = 1, \ldots, n_\Omega$, are continuous and if
     the level set
     \begin{align*}
        L =  \Big\{ (\hbx_1^\Omega, \hbx_1^\Gamma, \ldots, \hbx_{n_\Omega}^\Omega, \hbx_{n_\Omega}^\Gamma ) \; : \; &
                    \sum_{i=1}^{n_\Omega} \bC \bA_i \bg_i^\Gamma(\hbx_i^\Gamma) = \bzero, \\
               &    \sum_{i=1}^{n_\Omega} \norm{\bB_i \br_i\left(\bg_i^\Omega\left(\hbx_i^\Omega\right), 
                                                             \bg_i^\Gamma\left(\hbx_i^\Gamma \right)\right)}_2^2 
                  \le \sum_{i=1}^{n_\Omega} \norm{\bB_i \br_i\left(\bg_i^\Omega\left(\widetilde{\bx}_i^\Omega\right), 
                                                             \bg_i^\Gamma\left(\widetilde{\bx}_i^\Gamma \right)\right)}_2^2 
               \Big\}
     \end{align*}
     is bounded, then \eref{eq:dd_lspg_NLP} has a solution.
\end{theorem}
\begin{proof}
   If $(\hbx_i^\Omega, \hbx_i^\Gamma)$, $i=1, \dots, n_\Omega$, solves \eref{eq:dd_lspg_NLP},
   then it also solves the minimization problem with the constraint 
   $(\hbx_1^\Omega, \hbx_1^\Gamma, \ldots, \hbx_{n_\Omega}^\Omega, \hbx_{n_\Omega}^\Gamma ) \in L$ added.
   The feasible set of this new minimization problem is compact, the objective function is continuous, and therefore
   this  minimization problem has a solution, which is also a solution of \eref{eq:dd_lspg_NLP}.
 \end{proof}

Instead of the weak compatibility constraint (\ref{eq:dd_lspg_NLP}b) one can also construct the maps
$\bg_i^\Gamma$, $i = 1, \ldots, n_\Omega$, such that compatibility is enforced strongly for appropriate components 
of $\bg_i^\Gamma(\hbx_i^\Gamma)$, $i = 1, \ldots, n_\Omega$. This approach is introduced in the following  Section \ref{sec:rom_port_compatibility}.

%% file: rom_port_compatibility.tex

\subsection{Strong ROM-port constraints}\label{sec:rom_port_compatibility}
In the general formulation, the maps \eref{eq:rom_approximations} are computed separately for each subdomain $i$.
However, since the interface variables $\bx_i^\Gamma$,  $\bx_\ell^\Gamma$ are identical on each port $j$
associated with the subdomains $i, \ell$, i.e., on each port $j$ with $i, \ell \in P(j)$
(see \eref{eq:compatibility_conditions}), one can instead reduce the interface variables on each port
and then combine the reduced port interface variables to a reduced interface variable.
By \eref{eq:compatibility_conditions} the interface variables for port $j$ must satisfy
\begin{align}\label{eq:port_xp}
     \bx_j^p =  \bP_i^j  \bx_i^\Gamma \in \real^{N_j^p}, \qquad \forall \; i \in P(j).
\end{align}
For each port $j$ we reduce the FOM port variables $ \bx_j^p$, i.e., 
for each port $j$  we compute a single map $\bg_j^p:\real^{n_j^p}\to \real^{N_j^p}$, where
\begin{align}\label{eq:port_g}
    \bg_j^p\left( \hbx_j^p \right) \approx  \bx_j^p = \bP_i^j \bx_i^\Gamma, \qquad \forall \; i \in P(j).
\end{align}
The reduced interface variable $\hbx_i^\Gamma$ is now computed by concatenating 
all port variables $\hbx_j^p$ with $i \in P(j)$.
This leads to the ROM port sampling matrices $\hbP_i^j \in \set{0, 1}^{n_j^p \times n_i^\Gamma}$ 
which are defined through
\begin{equation}\label{eq:rom_port_matrix}
    \hbx_j^p = \hbP_i^j \hbx_i^\Gamma, \qquad i  \in P(j). 
\end{equation}
Equation \eref{eq:rom_port_matrix} implies that on the $j$th port we have
\begin{equation}\label{eq:rom_port_compatibility}
     \hbP_i^j \hbx_i^\Gamma = \hbP_\ell^j \hbx_\ell^\Gamma, \qquad i, \, \ell \in P(j). 
\end{equation}
To introduce parallelism, ROM interface variables $\hbx_i^\Gamma$ are introduced for each
subdomain, and are coupled by enforcing \eref{eq:rom_port_compatibility}.
By construction, the ROM ports are non-overlapping, i.e. 
\begin{align}\label{eq:rom_port_nonoverlapping}
    \hbP_i^j\left(\hbP_i^\ell\right)^T = \bzero, \qquad \forall \; j,\, \ell\in Q(i), \; j \neq \ell, 
\end{align}
and  $n_i^\Gamma = \sum_{j\in Q(i)} n_j^p$. 
As discussed in Section \ref{sec:dd_formulation}, after removing redundant conditions in \eref{eq:rom_port_compatibility}, 
one can write the ROM port compatibility conditions \eref{eq:rom_port_compatibility} as 
\begin{align}\label{eq:rom_strong_constraint}
    \sum_{i=1}^{n_\Omega} \hbA_i \hbx_i^\Gamma = \bzero, 
\end{align}
where $\hbA_i \in \set{-1, 0, 1}^{n_A \times n_i^\Gamma}$, \; $n_A = \sum_{j=1}^{n_p}(|P(j)|-1)n_j^p$, 
denote the constraint matrices associated with port compatibility conditions \eref{eq:rom_port_compatibility}. 
The matrix 
\begin{equation} \label{eq:dd_lspg_NLP_rom_port}
      (\hbA_1, \dots, \hbA_{n_\Omega}) \in \real^{n_A \times \sum_{i=1}^{n_\Omega} n_i^\Gamma}
\end{equation}
has full row rank $n_A$, and $n_A < \sum_{i=1}^{n_\Omega} n_i^\Gamma$.

The map $\bg_i^\Gamma:\real^{n_i^\Gamma}\to \real^{N_i^\Gamma}$ that approximates the interface state $\bx_i^\Gamma$ 
is implied by the port maps $\bg_j^p$.
To see this, note that the FOM compatibility conditions \eref{eq:compatibility_conditions} and the non-overlapping 
condition \eref{eq:fom_port_nonoverlapping} allow one to rewrite $\bx_i^\Gamma$ as 
\begin{align}\label{eq:interface_port_decomposition}
    \bx_i^\Gamma = \sum_{j \in Q(i)} (\bP_i^j)^T \bP_i^j \bx_i^\Gamma. 
\end{align}
Thus, using \eref{eq:port_g}, \eref{eq:rom_port_matrix}, and \eref{eq:interface_port_decomposition}
the map $\bg_i^\Gamma:\real^{n_i^\Gamma}\to \real^{N_i^\Gamma}$ that approximates the interface state $\bx_i^\Gamma$ 
is given by
\begin{align}\label{eq:gamma_map_port}
    \bg_i^\Gamma(\hbx_i^\Gamma) := \sum_{j \in Q(i)} (\bP_i^j)^T\bg_j^p
    \left(\hbP_i^j \hbx_i^\Gamma\right).
\end{align}
In particular, the definition \eref{eq:gamma_map_port} of $\bg_i^\Gamma$ and the ROM compatibility conditions \eref{eq:rom_port_compatibility} imply that 
\begin{align*}
    \bP_i^j\bg_i^\Gamma(\hbx_i^\Gamma) = 
    \bg_j^p\left(\hbP_i^j\hbx_i^\Gamma\right) = 
    \bg_j^p\left(\hbP_\ell^j \hbx_\ell^\Gamma\right) =
    \bP_\ell^j \bg_\ell^\Gamma(\hbx_\ell^\Gamma)
\end{align*}
for all $i, \ell\in P(j)$ and for all ports $P(j)$. This implies that strong compatibility holds for the FOM ports:
\begin{align*}
    \sum_{i=1}^{n_\Omega} \bA_i\bg_i^\Gamma(\hbx_i^\Gamma) = \bzero.
\end{align*}

In summary, if port maps $\bg_j^p$ are constructed such that  \eref{eq:port_g} holds and the implied
interface maps $\bg_i^\Gamma$ are \eref{eq:gamma_map_port}, then
the  DD-LSPG-ROM   is evaluated by solving 
\begin{subequations}\label{eq:dd_lspg_NLP_rom_port}
  \begin{align}
    \min_{(\hbx_i^\Omega, \hbx_i^\Gamma), i=1, \dots, n_\Omega} \quad &
    \frac{1}{2}\sum_{i=1}^{n_\Omega} \norm{\bB_i \br_i\left(\bg_i^\Omega\left(\hbx_i^\Omega\right), \bg_i^\Gamma\left(\hbx_i^\Gamma \right)\right)}_2^2 \\
    {\rm s.t.} \quad & \sum_{i=1}^{n_\Omega} \hbA_i \hbx_i^\Gamma = \bzero.
  \end{align}
\end{subequations}
The formulation \eref{eq:dd_lspg_NLP_rom_port} will be referred to as the strong ROM-port constraint (SRPC) formulation.

In contrast to \eref{eq:dd_lspg_NLP}  the constraints in \eref{eq:dd_lspg_NLP_rom_port} are linear
and the set of feasible points for  \eref{eq:dd_lspg_NLP_rom_port} is the null-space of the constraint
matrix \eref{eq:dd_lspg_NLP_rom_port}. Thus existence of feasible points for \eref{eq:dd_lspg_NLP_rom_port}
is now trivial. 
Existence of solutions of \eref{eq:dd_lspg_NLP_rom_port} can be guaranteed analogously to
Theorem~\ref{th:dd_lspg_NLP_solution}ii.

\begin{theorem}   \label{th:dd_lspg_NLP_rom_port_solution}
    i. The null-space of the constraint
       matrix \eref{eq:dd_lspg_NLP_rom_port} has dimension $(\sum_{i=1}^{n_\Omega} n_i^\Gamma) - n_A \ge 1$.

    ii.  Let $\widetilde{\bx}_i^\Gamma$, $i = 1, \ldots, n_\Omega$, satisfy the constraints (\ref{eq:dd_lspg_NLP_rom_port}b)
     and let $\widetilde{\bx}_i^\Omega$, $i = 1, \ldots, n_\Omega$, be arbitrary.
     If the residual function $\br$ and the maps $\bg_i^\Omega$, $\bg_i^\Gamma$, $i = 1, \ldots, n_\Omega$, are continuous and if
     the level set
     \begin{align*}
        L =  \Big\{ (\hbx_1^\Omega, \hbx_1^\Gamma, \ldots, \hbx_{n_\Omega}^\Omega, \hbx_{n_\Omega}^\Gamma ) \; : \; &
                    \sum_{i=1}^{n_\Omega}  \hbA_i \hbx_i^\Gamma = \bzero, \\
               &    \sum_{i=1}^{n_\Omega} \norm{\bB_i \br_i\left(\bg_i^\Omega\left(\hbx_i^\Omega\right), 
                                                             \bg_i^\Gamma\left(\hbx_i^\Gamma \right)\right)}_2^2 
                  \le \sum_{i=1}^{n_\Omega} \norm{\bB_i \br_i\left(\bg_i^\Omega\left(\widetilde{\bx}_i^\Omega\right), 
                                                             \bg_i^\Gamma\left(\widetilde{\bx}_i^\Gamma \right)\right)}_2^2 
               \Big\}
     \end{align*}
     is bounded, then \eref{eq:dd_lspg_NLP_rom_port} has a solution.
\end{theorem}
\begin{proof}
   The first part follows immediately from the properties of the constraint
    matrix \eref{eq:dd_lspg_NLP_rom_port}.
   The proof of ii.\ is analogous to the proof of Theorem~\ref{th:dd_lspg_NLP_solution}ii.
 \end{proof}

So far, we have specified our DD-LSPG-ROM  \eref{eq:dd_lspg_NLP}  or \eref{eq:dd_lspg_NLP_rom_port}
given maps $\bg_i^\Omega$ and $\bg_i^\Gamma$ such that \eref{eq:rom_approximations} holds,
or given maps $\bg_i^\Omega$, $\bg_j^p$ and implied interface maps  \eref{eq:gamma_map_port}
such that \eref{eq:rom_approximations} holds.
Next we discuss how these maps can be computed. 
In the following Section~\ref{sec:lsrom}
we first review traditional approaches based on linear subspaces to compute $\bg_i^\Omega$ and $\bg_i^\Gamma$ 
(or $\bg_i^\Omega$ and $\bg_j^p$).
In Section~\ref{sec:nmrom} we will then introduce the nonlinear-manifold ROM.

%% file: lsrom.tex

\subsection{Linear-subspace ROM} \label{sec:lsrom}
First we review linear subspace approximation to construct the maps $\bg_i^\Omega$ and $\bg_i^\Gamma$, 
or $\bg_i^\Omega$ and $\bg_j^p$. We will refer to resulting ROM as LS-ROM. 
The LS-ROM approach supposes that the state solutions of the FOM are contained in a low-dimensional linear subspace. 
A basis for the linear subspace is then computed,  resulting in a ROM whose state consists of the generalized coordinates 
of the state solution in the reduced subspace. 
The use of LS-ROM for the DD problem \eref{eq:dd_lspg_NLP} has already been considered in \cite{CHoang_YChoi_KCarlberg_2021a},
where the LS-ROM bases are computed using POD, but in principle any choice of basis can be used. 
The numerics in Section \ref{sec:numerics} also use POD for consistency with previous works.
We briefly review POD here for completeness. 
A thorough treatment of POD can be found in \cite{MHinze_SVolkwein_2005a}. 

As mentioned above, the LS-ROM approach approximates the FOM states $\bx_i^\Omega$, $\bx_i^\Gamma$ in a linear subspace. 
Hence $\bg_i^\Omega:\real^{n_i^\Omega} \to \real^{N_i^\Omega}$ and $\bg_i^\Gamma:\real^{n_i^\Gamma}\to \real^{N_i^\Gamma}$ 
are linear maps,
\begin{align*}
    \bg_i^\Omega: & \quad \hbx_i^\Omega \mapsto \bPhi_i^\Omega \hbx_i^\Omega, &
    \bg_i^\Gamma: & \quad \hbx_i^\Gamma \mapsto \bPhi_i^\Gamma\hbx_i^\Gamma, 
\end{align*}
where $\bPhi_i^\Omega \in \real^{N_i^\Omega \times n_i^\Omega}$ and $\bPhi_i^\Gamma \in \real^{N_i^\Gamma \times n_i^\Gamma}$ are basis matrices corresponding to the reduced linear subspaces. 
Consequently, the Jacobians $\frac{d}{d \hbx_i^\Omega}  \bg_i^\Omega(\hbx_i^\Omega) = \bPhi_i^\Omega$ and 
$\frac{d}{d \hbx_i^\Gamma} \bg_i^\Gamma (\hbx_i^\Gamma) = \bPhi_i^\Gamma$ are constant
and  do not need to be recomputed at each iteration of the SQP solver described in Section \ref{sec:sqp_solver}
that is used to solve  \eref{eq:dd_lspg_NLP} or   \eref{eq:dd_lspg_NLP_rom_port}. 

The POD bases are computed by minimizing the reconstruction error for a set of snapshots. 
First we focus on constructing POD bases for the WFPC formulation.
Recall that the residual functions $\br_i$ are parameterized with parameter space $\cD \subset\real^{N_\mu}$. 
Let $\set{\bmu_\ell^{\rm train}}_{\ell=1}^{n_\mu}\subset \cD$ be a set of training parameters, and solve the DD FOM \eref{eq:fom_dd} for each parameter $\bmu_\ell^{\rm train}$ to obtain FOM solutions 
$(\bx_i^\Omega(\bmu_\ell^{\rm train}), \bx_i^\Gamma(\bmu_\ell^{\rm train})), \; i=1, \dots, n_\Omega.$ 
Of course, one can solve the monolithic, single-domain FOM \eref{eq:fom_residual} at $\bmu = \bmu_\ell^{\rm train}$, and restrict the solution $\bx(\bmu_\ell^{\rm train})$ to the subdomain interior and interface states,
$\bx_i^\Omega(\bmu_\ell^{\rm train}) = \bP_i^\Omega \bx(\bmu_\ell^{\rm train})$, 
$\bx_i^\Gamma(\bmu_\ell^{\rm train}) = \bP_i^\Gamma \bx(\bmu_\ell^{\rm train})$.
One then computes bases $\bPhi_i^\Omega$ and $\bPhi_i^\Gamma$ using the SVD applied
to snapshot matrices for the interior and interface states 
\begin{subequations}\label{eq:snapshot_matrices}
    \begin{align}
    \bX_i^\Omega &= 
    \begin{bmatrix}
        \bx_i^\Omega(\bmu_1^{\rm train}) & \dots & \bx_i^\Omega(\bmu_{n_\mu}^{\rm train})
    \end{bmatrix} \in \real^{N_i^\Omega \times n_\mu}, \\
    \bX_i^\Gamma &= 
    \begin{bmatrix}
        \bx_i^\Gamma(\bmu_1^{\rm train}) & \dots & \bx_i^\Gamma(\bmu_{n_\mu}^{\rm train})
    \end{bmatrix} \in \real^{N_i^\Gamma \times n_\mu}. 
\end{align}
\end{subequations}
The process is the same for $\bPhi_i^\Omega$ and $\bPhi_i^\Gamma$ and therefore we drop
the superscript $\Omega$ or $\Gamma$ and describe the process to compute a basis  $\bPhi_i$ from
a generic snapshot matrix $\bX_i \in  \real^{N_i \times n_\mu}$.

One computes the `thin' SVD $\bX_i = \bU_i \bSigma_i \bV_i^T$ of the snapshot matrix,
where 
$\bU_i \in \real^{N_i \times m_i}$ is the matrix of left singular vectors, 
$\bSigma_i \in \real^{m_i \times m_i}$ is the diagonal matrix of singular values
$\sigma_1 \geq \sigma_2 \geq \dots \geq \sigma_{m_i} \ge 0$,
$\bV_i \in \real^{n_\mu \times m_i}$ is the matrix of right singular vectors, and
$m_i = \min\set{N_i, n_\mu}$. 
Given a tolerance $\nu_i \in (0, 1)$ one computes $n_i$ as the smallest integer such that
\begin{align}\label{eq:energy_criterion}
     \sum_{j=1}^{n_i} \sigma_j^2  \ge ( 1-\nu_i ) \sum_{j=1}^{m_i} \sigma_j^2,
\end{align}
and selects 
\[
    \bPhi_i = \bU_i(:, 1:n_i).
\]
The POD basis $\bPhi_i$ minimize the snapshot reconstruction error $\norm{\bX_i - \bPhi_i (\bPhi_i)^T\bX_i}_F^2$
among all possible orthogonal basis matrices of sizes $N_i \times n_i$. See, e.g., \cite{MHinze_SVolkwein_2005a}. 

POD basis construction for the SRPC formulation from Section~\ref{sec:rom_port_compatibility} is similar.
In fact,  the bases $\bPhi_i^\Omega$, $i = 1, \ldots, n_\Omega$, are computed as before, and
the bases  $\bPhi_i^\Gamma$, $i = 1, \ldots, n_\Omega$, are computed from bases $\bPhi_j^p$ on
the ports.
Because $\bx_j^p(\bmu_\ell^{\rm train}) = \bP_i^j\bx_i^\Gamma(\bmu_\ell^{\rm train})$ for any $i \in P(j)$
and all ports $P(j)$, the snapshot matrices restricted to port $P(j)$ are
\begin{equation}\label{eq:port_snapshot_matrix}
    \bX_j^p = \bP_i^j \bX_i^\Gamma \quad \mbox{ for any } i\in P(j).
\end{equation}
For each port $P(j)$, the POD basis $\bPhi_j^p = \bU_j^p(:, 1:n_j^p)$ is computed from
the `thin' SVD   $\bX_j^p = \bU_j^p \bSigma_j^p (\bV_j^p)^T$ as described before.
With the port basis matrices $\bPhi_j^p$ the linear map $\bg_i^\Gamma$ is constructed following \eref{eq:gamma_map_port},
\begin{align*}   
    \bg_i^\Gamma(\hbx_i^\Gamma) := \sum_{j\in Q(i)} (\bP_i^j)^T \bPhi_j^p \hbP_i^j \hbx_i^\Gamma 
    = \bPhi_i^\Gamma \hbx_i^\Gamma,
    \quad \mbox{ where } \quad
        \bPhi_i^\Gamma = \sum_{j\in Q(i)}(\bP_i^j)^T \bPhi_j^p \hbP_i^j.
\end{align*}

%% file: nmrom.tex

\subsection{Nonlinear-manifold ROM} \label{sec:nmrom}
The ROM approach that we focus on in this work is the nonlinear manifold approach, also referred to as NM-ROM. 
Rather than supposing that the state solutions of the FOM are contained in a low-dimensional linear subspace as in LS-ROM, 
one supposes that the FOM state solutions are contained in a low-dimensional nonlinear manifold. 
To build the NM-ROM, one must compute a suitable mapping from a low-dimensional coordinate space, often referred to as the {\it latent space}, to the manifold of candidate state solutions, or {\it trial manifold}. 
Solving the ROM then yields the generalized coordinates in the latent space for a solution in the trial manifold. 
The approach considered here for computing the nonlinear trial manifold 
is similar to the approach in \cite[Sec. 3]{YKim_YChoi_DWidemann_TZohdi_2022a}, which uses wide, shallow, and sparse autoencoders to compute suitable mappings from the latent space to the trial manifold. 
Further information regarding the autoencoder architecture we use is given in Section \ref{sec:autoencoder_architecture}.

To compute a DD ROM using the NM-ROM approach, one must compute continuously differentiable nonlinear mappings from a suitably chosen latent space to the trial manifold. 
Hence the nonlinear functions 
$\bg_i^\Omega: \real^{n_i^\Omega}\to\real^{N_i^\Omega}$ and
$\bg_i^\Gamma:\real^{n_i^\Gamma}\to \real^{N_i^\Gamma}$
defined in \eref{eq:rom_approximations} are computed as the {\it decoders} of autoencoders
$\ba_i^\Omega :\real^{N_i^\Omega} \to \real^{N_i^\Omega}$ and
$\ba_i^\Gamma :\real^{N_i^\Gamma} \to \real^{N_i^\Gamma}$.
The autoencoders $\ba_i^\Omega$ and $\ba_i^\Gamma$ consist of two parts each: encoders
$\bh_i^{\Omega}: \real^{N_i^\Omega}\to\real^{n_i^\Omega}$ and
$\bh_i^{\Gamma}: \real^{N_i^\Gamma}\to\real^{n_i^\Gamma},$
and decoders
$\bg_i^\Omega:\real^{n_i^\Omega}\to \real^{N_i^\Omega}$ and 
$\bg_i^\Gamma:\real^{n_i^\Gamma}\to \real^{N_i^\Gamma}.$
The encoders map inputs from the high-dimensional state space to a low-dimensional latent space,
while the decoders map elements from the low-dimensional space to the high-dimensional state space.
The autoencoders $\ba_i^\Omega$ and $\ba_i^\Gamma$ are then defined via the function compositions
$$\ba_i^\Omega = \bg_i^\Omega \circ \bh_i^\Omega, \qquad
\ba_i^\Gamma = \bg_i^\Gamma \circ \bh_i^\Gamma. $$
In this work, the encoders and decoders are neural networks such that
the autoencoder approximates its inputs:
\begin{align*}
  \bx_i^\Omega \approx \ba_i^\Omega(\bx_i^\Omega) = \bg_i^\Omega(\bh_i^\Omega(\bx_i^\Omega)), \qquad
  \bx_i^\Gamma \approx \ba_i^\Gamma(\bx_i^\Gamma)= \bg_i^\Gamma(\bh_i^\Gamma(\bx_i^\Gamma)), \qquad
  i=1, \dots, n_\Omega.
\end{align*}
Further details on the neural network architecture used can be found in Section \ref{sec:autoencoder_architecture}.
The decoders $\bg_i^\Omega$ and $\bg_i^\Gamma$ can be interpreted as approximate inverses of the encoders $\bh_i^\Omega$ and $\bh_i^\Gamma$. 
In the SRPC case, the autoencoder $\ba_i^\Gamma$ is composed of autoencoders $\ba_j^p:\real^{N_j^p}\to \real^{N_j^p}$ with encoder $\bh_j^p:\real^{N_j^p} \to \real^{n_j^p}$ and decoder $\bg_j^p:\real^{n_j^p} \to \real^{N_j^p}$ for each port $P(j)$.

The mean-square-error (MSE) losses for the interior, interface, and port states are defined as
\begin{subequations}\label{eq:autoencoder_AbsMSE_loss}
  \begin{align}
    \cL_i^\Omega &:= \frac{1}{n_\mu}\sum_{\ell=1}^{n_\mu} \norm{\bx_i^{\Omega}(\bmu_\ell^{\rm train})-\bg_i^\Omega(\bh_i^\Omega(\bx_i^{\Omega}(\bmu_\ell^{\rm train})))}_2^2, & i&=1, \dots, n_\Omega, \\
    \cL_i^\Gamma &:= \frac{1}{n_\mu}\sum_{\ell=1}^{n_\mu} \norm{\bx_i^{\Gamma}(\bmu_\ell^{\rm train})-\bg_i^\Gamma(\bh_i^\Gamma(\bx_i^{\Gamma}(\bmu_\ell^{\rm train})))}_2^2, & i&=1, \dots, n_\Omega,\\
    \cL_j^p &:= \frac{1}{n_\mu}\sum_{\ell=1}^{n_\mu} \norm{\bx_j^p(\bmu_\ell^{\rm train})-\bg_j^p(\bh_j^p(\bx_j^p(\bmu_\ell^{\rm train})))}_2^2, & j&=1, \dots, n_p,
  \end{align}
\end{subequations}
where $\bx_i^{\Omega}(\bmu_\ell^{\rm train})$ and $\bx_i^{\Gamma}(\bmu_\ell^{\rm train})$ are snapshots of the interior and interface states on subdomain $i$ at parameter $\bmu_\ell^{\rm train}$ and $\bx_j^p(\bmu_\ell^{\rm train})$ is the state on port $P(j)$,
as discussed in Section \ref{sec:lsrom}.
In the WPFC case, the interior state and interface state autoencoders $\ba_i^\Omega$ and $\ba_i^\Gamma$ are trained by minimizing the interior and interface losses $\cL_i^\Omega$ and $\cL_i^\Gamma$, respectively. 
In the SPRC case, the interior state autoencoders $\ba_i^\Omega$ and the port autoencoders $\ba_j^p$ are trained by minimizing 
the interior and interface losses $\cL_i^\Omega$ and $\cL_j^p$, respectively. 
The interface state autoencoders $\ba_i^\Gamma$ are implied by the port autoencoders. Specifically,  $\bh_i^\Gamma$ is
\begin{equation}\label{eq:gamma_port_encoder}
    \bh_i^\Gamma(\bx_i^\Gamma) = \sum_{j \in Q(i)} (\hbP_i^j)^T\bh_j^p(\bP_i^j\bx_i^\Gamma),
\end{equation}
and $\bg_i^\Gamma$ is defined using equation \eref{eq:gamma_map_port}.

Notice that minimizing the MSE loss is equivalent to minimizing the snapshot reconstruction error, 
which is exactly how POD bases are constructed, as discussed in Section \ref{sec:lsrom}.
Training the autoencoders can also be interpreted as ``learning" the forward and inverse mappings from the latent space of generalized coordinates to the nonlinear trial manifold. 
After training, the decoders $\bg_i^\Omega$ and $\bg_i^\Gamma$ are used for the DD NM-ROM \eref{eq:dd_lspg_NLP} or \eref{eq:dd_lspg_NLP_rom_port}.

%% file: sqp_solver.tex

\section{Sequential quadratic programming solver} \label{sec:sqp_solver}

\subsection{Lagrange-Gauss-Newton SQP method } \label{sec:LGN_sqp_solver}
The problems (\ref{eq:dd_lspg_NLP}) and \eref{eq:dd_lspg_NLP_rom_port} are nonlinear programs (NLPs) with equality constraints,
and can be solved using  sequential quadratic programming (SQP) \cite{PTBoggs_JWTolle_1995a}, \cite[Ch.~18]{JNocedal_SJWright_2006a}.
The SQP solver detailed below amounts to applying a Newton-type method to the Karush-Kuhn-Tucker (KKT) necessary optimality conditions.
Note that the SQP solver described in this section can also be applied to the FOM \eref{eq:fom_dd_nlp} by considering the case $\bB_i=\bI$ and $\bg_i^\Omega, \bg_i^\Gamma$ equal to the identity mapping. 

To develop a solver that is applicable to either the WFPC \eref{eq:dd_lspg_NLP} or the SPRC formulations \eref{eq:dd_lspg_NLP_rom_port}, we define the constraint functions $\tbA_i: \real^{n_i^\Gamma} \to \real^{n_A}$ 
and consider the general formulation
\begin{subequations}\label{eq:dd_lspg_NLP_general}
  \begin{align}
    \min_{(\hbx_i^\Omega, \hbx_i^\Gamma), i=1, \dots, n_\Omega} \quad &
    \frac{1}{2}\sum_{i=1}^{n_\Omega} \norm{\bB_i \br_i\left(\bg_i^\Omega\left(\hbx_i^\Omega\right), \bg_i^\Gamma\left(\hbx_i^\Gamma \right)\right)}_2^2 \\
    {\rm s.t.} \quad & \sum_{i=1}^{n_\Omega} \tbA_i \big( \hbx_i^\Gamma\big) = \bzero.
  \end{align}
\end{subequations}
In the WFPC case $\tbA_i(\hbx_i^\Omega) = \bC \bA_i \bg_i^\Gamma(\hbx_i^\Gamma)$ and the constraints are nonlinear in the case of NM-ROMs.
In the SRPC case $\tbA_i(\hbx_i^\Gamma) = \hbA_i \hbx_i^\Gamma$ and the constraints are always linear.

To apply the SQP solver, one first writes the Lagrangian 
\begin{equation}\label{eq:dd_lspg_lagrangian}
  \widehat{L}(\hbx_1^\Omega, \hbx_1^\Gamma, \dots, \hbx_{n_\Omega}^{\Omega}, \hbx_{n_\Omega}^\Gamma, \hblambda) =
  \frac{1}{2}\sum_{i=1}^{n_\Omega} \norm{\bB_i \br_i\left(\bg_i^\Omega\left(\hbx_i^\Omega\right), \bg_i^\Gamma\left(\hbx_i^\Gamma \right)\right)}_2^2 +
  \sum_{i=1}^{n_\Omega} \hblambda^T \tbA_i(\hbx_i^\Gamma)
\end{equation}
for the ROM NLP \eref{eq:dd_lspg_NLP_general},
where $\hblambda \in \real^{n_A}$ are the Lagrange multipliers associated with the DD-ROM constraints (\ref{eq:dd_lspg_NLP_general}b).

Let $\nabla_{\bv}$ and $\frac{\partial}{\partial \bv}$ denote the partial gradient and partial Jacobian with respect to $\bv$, respectively,
and let $\frac{d}{d \bv}$ denote the Jacobian.
The first order necessary optimality conditions are
\begin{subequations}\label{eq:dd_nm_lspg_1st_order_opt}
  \begin{align}
    \nabla_{\hbx_i^\Omega}\widehat{L}(\hbx_1^\Omega, \hbx_1^\Gamma, \dots, \hbx_{n_\Omega}^{\Omega}, \hbx_{n_\Omega}^\Gamma, \hblambda) &=
    \brho_i^\Omega(\hbx_i^\Omega, \hbx_i^\Gamma) = \bzero, &
    i &= 1, \dots, n_\Omega, \\
    \nabla_{\hbx_i^\Gamma}\widehat{L}(\hbx_1^\Omega, \hbx_1^\Gamma, \dots, \hbx_{n_\Omega}^{\Omega}, \hbx_{n_\Omega}^\Gamma, \hblambda) &=
    \brho_i^\Gamma(\hbx_i^\Omega, \hbx_i^\Gamma, \hblambda)
    = \bzero,
    & i &= 1, \dots, n_\Omega, \\
    \nabla_{\hblambda}\widehat{L}(\hbx_1^\Omega, \hbx_1^\Gamma, \dots, \hbx_{n_\Omega}^{\Omega}, \hbx_{n_\Omega}^\Gamma, \hblambda) &=
      \sum_{i=1}^{n_\Omega} \tbA_i(\hbx_i^\Gamma) = \bzero,
  \end{align}
\end{subequations}
where
\begin{subequations}\label{eq:optimality_residuals}
  \begin{align}
    \brho_i^\Omega(\hbx_i^\Omega, \hbx_i^\Gamma) &=
    \frac{d \bg_i^\Omega}{d \hbx_i^\Omega}(\hbx_i^\Omega)^T
    \frac{\partial \br_i}{\partial \bx_i^\Omega}\big(\bg_i^\Omega(\hbx_i^\Omega),\bg_i^\Gamma(\hbx_i^\Gamma)\big)^T \bB_i^T \bB_i
    \br_i\big(\bg_i^\Omega(\hbx_i^\Omega),\bg_i^\Gamma(\hbx_i^\Gamma)\big), \\
    \brho_i^\Gamma(\hbx_i^\Omega, \hbx_i^\Gamma, \hblambda) &=
    \frac{d \bg_i^\Gamma}{d \hbx_i^\Gamma}(\hbx_i^{\Gamma})^T
    \frac{\partial \br_i}{\partial \bx_i^\Gamma}\big(\bg_i^\Omega(\hbx_i^\Omega),\bg_i^\Gamma(\hbx_i^\Gamma)\big)^T \bB_i^T \bB_i
    \br_i\big(\bg_i^\Omega(\hbx_i^\Omega),\bg_i^\Gamma(\hbx_i^\Gamma)\big) 
    + \frac{d \tbA_i}{d \hbx_i^\Gamma}(\hbx_i^\Gamma)^T \hblambda
  \end{align}
\end{subequations}
are the gradients of the Lagrangian with respect to the subdomain variables 
$\hbx_i^\Omega$ and $\hbx_i^\Gamma$, respectively.

A Newton-type method applied to \eref{eq:dd_nm_lspg_1st_order_opt} yields the SQP iterations
\begin{align}\label{eq:sqp_system}
  \begin{bmatrix}
    \bH_1(\hbx_1^{\Omega(k)}, \hbx_1^{\Gamma(k)}) &  &\dots & \bE_1(\hbx_1^\Gamma)^T \\
    & \ddots &  & \vdots\\
     & & \bH_{n_\Omega}(\hbx_{n_\Omega}^{\Omega(k)}, \hbx_{n_\Omega}^{\Gamma(k)}) & \bE_{n_\Omega}(\hbx_1^\Gamma)^T \\
     \bE_1(\hbx_1^\Gamma) & \dots & \bE_{n_\Omega}(\hbx_{n_\Omega}^\Gamma) &\bzero
  \end{bmatrix}
  \begin{bmatrix}
    \bs_1^{(k)} \\ \vdots \\ \bs_{n_\Omega}^{(k)} \\ \bs^{\hblambda(k)}
  \end{bmatrix} =-
  \begin{bmatrix}
    \brho_1(\hbx_1^{\Omega(k)}, \hbx_1^{\Gamma(k)}, \hblambda^{(k)}) \\
    \vdots \\
    \brho_{n_\Omega}(\hbx_{n_\Omega}^{\Omega(k)}, \hbx_{n_\Omega}^{\Gamma(k)}, \hblambda^{(k)}) \\
    \sum_{i=1}^{n_\Omega} \tbA_i (\hbx_i^{\Gamma(k)})
  \end{bmatrix},
\end{align}
where $k$ is the SQP iteration index, 
$\bH_i(\hbx_i^{\Omega(k)}, \hbx_i^{\Gamma(k)})$ is the Hessian of the Lagrangian with respect to the subdomain variables 
$(\hbx_i^\Omega, \hbx_i^\Gamma)$ evaluated at $(\hbx_i^{\Omega(k)}, \hbx_i^{\Gamma(k)})$ or an approximation of this Hessian, and where
\begin{align}\label{eq:sqp_block_matrices}
  \bE_i(\hbx_i^\Gamma) &=
  \begin{bmatrix}\bzero & \frac{d \tbA_i}{d \hbx_i^\Gamma}(\hbx_i^\Gamma)
  \end{bmatrix}, &
  \bs_i^{(k)} &=
  \begin{bmatrix}
    \bs_i^{\Omega(k)} \\ \bs_i^{\Gamma(k)}
  \end{bmatrix}, &
  \brho_i(\hbx_i^{\Omega}, \hbx_i^{\Gamma}, \hblambda) &=
  \begin{bmatrix}
    \brho_i^\Omega(\hbx_i^{\Omega}, \hbx_i^{\Gamma}) \\
    \brho_i^\Gamma(\hbx_i^{\Omega}, \hbx_i^{\Gamma}, \hblambda)
  \end{bmatrix},
\end{align}
for $i=1, \dots, n_\Omega$. 
The next result on the unique solvability of \eref{eq:sqp_system} follows from adapting standard results to the block structure of \eref{eq:sqp_system}.

\begin{lemma}\label{lemma:sqp_system_solution}
       If for $i = 1, \ldots, n_\Omega$ the matrices $\bH_{i}(\hbx_{i}^{\Omega(k)}, \hbx_{i}^{\Gamma(k)})$
       are positive definite on the null-space of $\bE_i(\hbx_i^\Gamma)$, and if
       $\big( \bE_1(\hbx_1^\Gamma), \ldots, \bE_{n_\Omega}(\hbx_{n_\Omega}^\Gamma) \big)$ has full row  rank, 
       then \eref{eq:sqp_system} has a unique solution.
 \end{lemma}
For a proof see, e.g., \cite[Thm.~3.2]{MBenzi_GHGolub_JLiesen_2005a}, \cite[Lemma~16.12]{JNocedal_SJWright_2006a}.

We use a Gauss-Newton approximation of the Hessian, which is motivated by the following consideration.
If the residuals $\bB_i \br_i\big(\bg_i^\Omega(\hbx_i^\Omega),\bg_i^\Gamma(\hbx_i^\Gamma)\big)$ are small  at the solution 
of \eref{eq:dd_lspg_NLP_general}, then the first order optimality condition (\ref{eq:dd_nm_lspg_1st_order_opt}b) implies
that $\hblambda$ is small. Thus all second derivative terms in the true Hessians $\bH_i(\hbx_i^{\Omega(k)}, \hbx_i^{\Gamma(k)})$ are multiplied by small
residuals or small Lagrange multipliers. The Gauss-Newton Hessian approximation neglects these terms and approximates the Hessians by
\begin{subequations}\label{eq:gauss_newton_hessian}
  \begin{align}
    \bH_i(\hbx_i^{\Omega}, \hbx_i^{\Gamma})&= \bR_i(\hbx_i^\Omega, \hbx_i^\Gamma)^T \bB_i^T \bB_i \bR_i(\hbx_i^\Omega, \hbx_i^\Gamma),
   \end{align}
 where
  \begin{align}
   \bR_i(\hbx_i^\Omega, \hbx_i^\Gamma) &= 
    \begin{bmatrix}
      \frac{\partial \br_i}{\partial \bx_i^\Omega}\big(\bg_i^\Omega(\hbx_i^{\Omega}),\bg_i^\Gamma(\hbx_i^{\Gamma})\big)
      \frac{d \bg_i^\Omega}{d \hbx_i^\Omega}(\hbx_i^{\Omega}), &
      \frac{\partial \br_i}{\partial \bx_i^\Gamma}\big(\bg_i^\Omega(\hbx_i^{\Omega}),\bg_i^\Gamma(\hbx_i^{\Gamma})\big)
      \frac{d \bg_i^\Gamma}{d \hbx_i^\Gamma}(\hbx_i^{\Gamma})
    \end{bmatrix}
  \end{align}
\end{subequations}
is the Jacobian of $\br_i$ with respect to $(\hbx_i^\Omega, \hbx_i^\Gamma)$.
The advantage is that \eref{eq:gauss_newton_hessian} only requires first order derivatives.
Note that the  FOM solution satisfies \eref{eq:fom_dd}, i.e., the residual in the least squares formulation \eref{eq:fom_dd_nlp}
is zero. Thus, if the ROM well approximates the FOM \eref{eq:fom_dd} or, equivalently, its least squares formulation \eref{eq:fom_dd_nlp},
then we expect the residuals  $\bB_i \br_i\big(\bg_i^\Omega(\hbx_i^\Omega),\bg_i^\Gamma(\hbx_i^\Gamma)\big)$ to be small  
at the solution of \eref{eq:dd_lspg_NLP_general} and
the Gauss-Newton Hessian \eref{eq:gauss_newton_hessian} to be  good approximation of the true Hessian
of the Lagrangian \eref{eq:dd_lspg_lagrangian}.

Note that with the notation \eref{eq:gauss_newton_hessian}, the gradients $\brho_i(\hbx_i^{\Omega}, \hbx_i^{\Gamma}, \hblambda)$
in \eref{eq:sqp_block_matrices} can be written as 
\begin{equation}\label{eq:optimality_residuals_GN}
  \brho_i(\hbx_i^{\Omega}, \hbx_i^{\Gamma}, \hblambda) 
  =  \bR_i(\hbx_i^{\Omega}, \hbx_i^{\Gamma})^T \bB_i^T \bB_i \br_i\big(\bg_i^\Omega\big(\hbx_i^{\Omega}\big), \bg_i^\Gamma\big(\hbx_i^{\Gamma} \big)\big)
       +   \bE_i(\hbx_i^\Gamma)^T \hblambda, \quad i=1, \dots, n_\Omega.
\end{equation}
With the Gauss-Newton approximations \eref{eq:gauss_newton_hessian}, the SQP  system \eref{eq:sqp_system} is essentially  the
optimality system for the quadratic program
\begin{subequations}\label{eq:dd_lspg_NLP_general_QP}
  \begin{align}
    \min_{\bs_i = (\bs_i^\Omega, \bs_i^\Gamma), i=1, \dots, n_\Omega} \quad &
    \frac{1}{2}\sum_{i=1}^{n_\Omega} \norm{ \bB_i \br_i\big(\bg_i^\Omega\big(\hbx_i^{\Omega(k)}\big), \bg_i^\Gamma\big(\hbx_i^{\Gamma(k)} \big)\big) 
                                                                        +  \bB_i \bR_i(\hbx_i^{\Omega(k)}, \hbx_i^{\Gamma(k)})  \bs_i  }_2^2 \\
    {\rm s.t.} \quad & \sum_{i=1}^{n_\Omega} \tbA_i \big( \hbx_i^{\Gamma(k)} \big) 
                                 +   \frac{d}{d\hbx_i^\Gamma} \tbA_i \big( \hbx_i^{\Gamma(k)} \big)  \bs_i = \bzero.
  \end{align}
\end{subequations}
More precisely, the following result holds.

\begin{lemma}\label{lemma:sqp_system_solution2}
       If the assumptions of Lemma~\ref{lemma:sqp_system_solution} hold, then the quadratic program \eref{eq:dd_lspg_NLP_general_QP}
       has a unique solution $\bs_i^{(k)} = (\bs_i^{\Omega(k)}, \bs_i^{\Gamma(k)})$, $i=1, \dots, n_\Omega$, given by
       the solution of  \eref{eq:sqp_system}. The associated Lagrange multiplier for \eref{eq:dd_lspg_NLP_general_QP}
       is $\hblambda^{(k)} + \bs^{\hblambda(k)}$, where $\hblambda^{(k)}$ is the  Lagrange multiplier estimate in   \eref{eq:sqp_system} 
       and $\bs^{\hblambda(k)}$ is the last component in the solution vector of  \eref{eq:sqp_system}.
\end{lemma}
The proof of Lemma~\ref{lemma:sqp_system_solution2}
follows from the necessary and sufficient optimality conditions (e.g.,  \cite[Sec~16.1]{JNocedal_SJWright_2006a})
for the quadratic program \eref{eq:dd_lspg_NLP_general_QP}. The necessary and sufficient optimality conditions 
for \eref{eq:dd_lspg_NLP_general_QP} are given by  \eref{eq:sqp_system} with the terms 
$\bE_i(\hbx_i^{\Gamma(k)})^T \hblambda^{(k)}$ (see \eref{eq:optimality_residuals_GN}) moved from the right to the left hand side.

An advantage of the Gauss-Newton approximation is that no Lagrange multiplier estimate is needed in \eref{eq:dd_lspg_NLP_general_QP}
or the associated optimality system.
\sloppy 
Of course, quantities like $\bg_i^\Omega(\hbx_i^\Omega)$,
    $\bg_i^\Gamma(\hbx_i^\Gamma)$, 
    $\bB_i\br_i(\bg_i^\Omega(\hbx_i^\Omega), \bg_i^\Gamma(\hbx_i^\Gamma))$, 
    $\bB_i \bR_i(\hbx_i^\Omega, \hbx_i^\Gamma)$, 
    $\tbA_i(\hbx_i^\Gamma)$,  and
    $\frac{d \tbA_i}{d \hbx_i^\Gamma}(\hbx_i^\Gamma)$ can be computed in parallel across the subdomains.
Moreover, the block structure of the system \eref{eq:sqp_system} lends itself to a parallel solution strategy. 
However, since \eref{eq:sqp_system} corresponds to the ROM its size tends to be small and 
parallelism in its solution may yield less speedup than it would if applied to the DD formulation \eref{eq:fom_dd}
of the FOM \eref{eq:fom_residual}.
The parallel implementation of the approach discussed in this paper is left to future work.

Given the solution of the SQP  system \eref{eq:sqp_system}, the new iterate, i.e., the new approximate solution of  \eref{eq:dd_lspg_NLP_general}
is computed as
\begin{subequations}\label{eq:sqp_update}
  \begin{align}
    \hbx_i^{\Omega(k+1)} &= \hbx_i^{\Omega(k)} + \alpha^{(k)}\bs_i^{\Omega(k)}, & i=1, \dots, n_\Omega, \\
    \hbx_i^{\Gamma(k+1)} &= \hbx_i^{\Gamma(k)} + \alpha^{(k)}\bs_i^{\Gamma(k)}, & i=1, \dots, n_\Omega,
  \end{align}
\end{subequations}
with step size $\alpha^{(k)} \in (0, 1]$.  
If one chooses to keep a Lagrange multiplier estimate, then $\hblambda^{(k+1)} = \hblambda^{(k)} + \alpha^{(k)}\bs^{\hblambda(k)}$,   $i=1, \dots, n_\Omega$,
where  $\bs^{\hblambda(k)}$ is the last component in the solution vector of  \eref{eq:sqp_system}.
The step size $\alpha^{(k)}$ is computed via line-search using a merit function that coordinates progress of
the iterates (and Lagrange multipliers) towards feasibility and optimality. 
In this work, we simply use the norm of the gradients \eref{eq:dd_nm_lspg_1st_order_opt}. This is an appropriate criterion
if one starts sufficiently close to a (local) minimizer of \eref{eq:dd_lspg_NLP_general}, and this criterion yielded good results 
in our examples. In our examples, the step size is computed using a backtracking line search with the Armijo rule.

\subsection{Convergence of SQP Solver}\label{sec:sqp_convergence}
Convergence of the Lagrange-Gauss-Newton SQP method can be established using one of two related approaches.
The iteration \eref{eq:sqp_update} with $\bs_i^{\Omega(k)},  \bs_i^{\Gamma(k)}$, $i=1, \dots, n_\Omega$,
computed as the solution of \eref{eq:sqp_system} with Gauss-Newton Hessian approximation \eref{eq:gauss_newton_hessian}
can be interpreted and analyzed as a generalized Gauss-Newton iteration. See \cite{HGBock_EKostina_JPSchloeder_2007a}.
Alternatively, this iteration can also be viewed as an inexact Newton method applied to the first-order optimality conditions
\eref{eq:dd_nm_lspg_1st_order_opt}.
The local convergence result using either approach requires that the  Lagrange-Gauss-Newton SQP method is started
 sufficiently close to a (local) minimizer of \eref{eq:dd_lspg_NLP_general}.
 We summarize the convergence theory for inexact Newton methods (see, e.g. \cite[Thm. 11.3]{JNocedal_SJWright_2006a})
 in Theorem \ref{thm:sqp_convergence}. 
First we define the following notation to improve readability.
We group the vectors $(\hbx_1^\Omega, \hbx_1^\Gamma, \dots, \hbx_{n_\Omega}^\Omega, \hbx_{n_\Omega}^\Gamma)$ and 
$(\bs_1^{\Omega(k)}, \bs_1^{\Gamma(k)}, \dots,
    \bs_{n_\Omega}^{\Omega(k)}, \bs_{n_\Omega}^{\Gamma(k)})$ as 
\begin{subequations}\label{eq:sqp_notation}
    \begin{align}
        \hbx &= 
        \begin{bmatrix}
            \hbx_1^{\Omega}\\ \hbx_1^{\Gamma} \\ \vdots \\
            \hbx_{n_\Omega}^{\Omega}\\ \hbx_{n_\Omega}^{\Gamma}
        \end{bmatrix} \in \real^{n_D},
        &
        \bs_x^{(k)} &= 
        \begin{bmatrix}
            \bs_1^{\Omega(k)}\\ \bs_1^{\Gamma(k)} \\ \vdots \\
            \bs_{n_\Omega}^{\Omega(k)}\\ \bs_{n_\Omega}^{\Gamma(k)}
        \end{bmatrix}\in \real^{n_D}.
    \end{align}
where $n_D =\sum_{i=1}^{n_\Omega} (n_i^\Omega + n_i^\Gamma)$.
Furthermore  let $\hbF:\real^{n_D+n_A}\to \real^{n_D+n_A}$,
\begin{align}
    \hbF(\hbx, \hblambda) &=
    -
  \begin{bmatrix}
    \brho_1(\hbx_1^{\Omega(k)}, \hbx_1^{\Gamma(k)}, \hblambda^{(k)}) \\
    \vdots \\
    \brho_{n_\Omega}(\hbx_{n_\Omega}^{\Omega(k)}, \hbx_{n_\Omega}^{\Gamma(k)}, \hblambda^{(k)}) \\
    \sum_{i=1}^{n_\Omega} \tbA_i(\hbx_i^{\Gamma(k)})
  \end{bmatrix}.
\end{align}
denote the right hand side of the KKT system.
Recall that if  $\obx \in \real^{n_D}$ is a local minimizer of \eref{eq:dd_lspg_NLP_general} with associated Lagrange multiplier 
$\oblambda \in \real^{n_A}$, then $\hbF(\obx, \oblambda) = \bzero$.
Next define the Hessian approximation 
$\bH : \real^{n_D}\times \real^{n_A} \to \real^{n_D \times n_D}$, 
\begin{align}
    \bH(\hbx, \hblambda) &=
    \begin{bmatrix}
    \bH_1(\hbx_1^\Omega, \hbx_1^\Gamma, \hblambda) \\
    & \ddots \\
    && \bH_{n_\Omega}(\hbx_{n_\Omega}^\Omega, \hbx_{n_\Omega}^\Gamma, \hblambda)
    \end{bmatrix},
\end{align}
and constraint Jacobian $ \frac{d}{d\hbx^\Gamma}  \tbA : \real^{\sum_{i=1}^{n_\Omega} n_i^\Gamma} \to \real^{n_A \times n_D}$, 
\begin{align}
  \frac{d}{d\hbx^\Gamma}  \tbA (\hbx_1^\Gamma, \dots, \hbx_{n_\Omega}^\Gamma) &=
    \begin{bmatrix}
    \bzero & \frac{d}{d\hbx_1^\Gamma}  \tbA_1(\hbx_1^\Gamma) & \dots & 
    \bzero &  \frac{d}{d\hbx_{n_\Omega}^\Gamma}  \tbA_{n_\Omega}(\hbx_{n_\Omega}^\Gamma)
    \end{bmatrix}.
\end{align}
\end{subequations}
The convergence result can now be stated as follows.
\begin{theorem}\label{thm:sqp_convergence}
    Let $\br_i$, $\bg_i^\Omega$, and $\bg_i^\Gamma$ be continuously differentiable for all $i=1, \dots, n_\Omega$. 
    Let $\obx$ be a local minimizer of \eref{eq:dd_lspg_NLP_general} such that the Jacobian
    $\frac{d}{d\hbx^\Gamma}\tbA(\obx_1^\Gamma, \dots, \obx_{n_\Omega}^\Gamma)$ has full row rank,
    let $\oblambda$ denote the associated Lagrange multiplier and assume that
    $\bH(\obx, \oblambda)$ is positive definite on the null-space of 
    $\frac{d}{d\hbx^\Gamma} \tbA(\obx_1^\Gamma, \dots, \obx_{n_\Omega}^\Gamma)$. 
    Furthermore, assume that $\hbF'(\hbx, \hblambda)$ is Lipschitz continuous with Lipschitz constant $K$, and that the 
    steps $\bs_x^{(k)}$ satisfy 
    \begin{align*}
        \norm{\left(\nabla_{\hbx}^2 \widehat{L}(\hbx^{(k)}, \hblambda^{(k)}) - \bH(\hbx^{(k)}, \hblambda^{(k)})\right)\bs_x^{(k)}}_2 
        \leq \eta_k\norm{\hbF(\hbx^{(k)}, \hblambda^{(k)})}_2
    \end{align*}
    for some sequence of forcing parameters $\eta_k$, 
    and where $\widehat{L}$ is the Lagrangian defined in (\ref{eq:dd_lspg_lagrangian}).
    
    If $\set{\eta_k}$ satisfies $0 < \eta_k \leq \eta$ where $\eta$ is such that
    $4\eta \bar{\kappa}  < 1$  with $\bar{\kappa} =  \big\| \hbF'(\obx, \oblambda)^{-1} \big\|_2 \big\| \hbF'(\obx,\oblambda) \big\|_2$,
    then for all 
    $\sigma \in (4\eta \bar{\kappa} , 1)$
    there exists an $\epsilon > 0$ such that for any $(\hbx^{(0)}, \hblambda^{(0)})$ 
    with 
    $\big\| (\obx, \oblambda)-(\hbx^{(0)}, \hblambda^{(0)}) \big\|_2 < \epsilon$, the sequence of iterates 
    $\big\{ (\hbx^{(k)}, \hblambda^{(k)}) \big\}$ generated by the SQP solver converges to $(\obx, \oblambda)$,
   and the iterates satisfy
    \begin{align}\label{eq:sqp_convergence}
        \big\| (\hbx^{(k)}, \hblambda^{(k)})-(\obx, \oblambda) \big\|_2
        &\leq K \big\| \hbF'(\obx, \oblambda)^{-1} \big\|_2 \big\| (\hbx^{(k)}, \hblambda^{(k)})-(\obx, \oblambda) \big\|_2^2 
               + 4\eta_k \bar{\kappa} \big\| (\hbx^{(k)}, \hblambda^{(k)})-(\obx, \oblambda) \big\|_2 \nonumber \\
        &\leq \sigma \big\| (\hbx^{(k)}, \hblambda^{(k)})-(\obx, \oblambda) \big\|_2. \nonumber
    \end{align}
\end{theorem}

See, e.g., \cite[Thm. 11.3]{JNocedal_SJWright_2006a} for a proof of this theorem. 

\begin{remark}\label{rmk:kkt_is_necessary}
 As stated, Theorem~\ref{thm:sqp_convergence} only guarantees a solution to the KKT system \eref{eq:dd_nm_lspg_1st_order_opt}, 
 which are (first order) {\it necessary} optimality conditions. However one can give alternative inexactness conditions on 
 Gauss-Newton Hessian approximations that ensure local convergence to a point at which
 the second order sufficient optimality conditions are satisfied.
 See \cite[L.~2.5] {MHeinkenschloss_1993}  for the unconstrained case
 and  \cite[Sec.~3.5]{HGBock_EKostina_JPSchloeder_2007a}
 for the constrained case,  but with $\ell_1$ rather than $\ell_2$ (=least squares) objective.
\end{remark}

%% file: autoencoder_architecture.tex
\section{Autoencoder architecture}\label{sec:autoencoder_architecture}
Following \cite{YKim_YChoi_DWidemann_TZohdi_2022a}, we consider the use of shallow, wide, sparse-masked autoencoders with smooth activation functions for representing the maps, $\bg_i^\Omega$ and $\bg_i^\Gamma$. 
Shallow networks are used for computational efficiency; fewer layers correspond to fewer repeated matrix-vector multiplications when evaluating the decoders. 
The shallow depth necessitates a wide network to maintain enough expressiveness for use in NM-ROM. 
Sparsity is applied at the decoder output layer so that hyper-reduction can be applied. 
Further details on hyper reduction are addressed in Section \ref{sec:hyper_reduction}. 
Smooth activations are used to ensure that $\bg_i^\Omega$ and $\bg_i^\Gamma$ are continuously differentiable. 
In contrast with \cite{YKim_YChoi_DWidemann_TZohdi_2022a}, we also apply a sparsity mask to the encoder input layer so that the encoders and decoders are symmetric across the latent layer.
We found that applying a sparsity mask to the encoder input layer permitted the use of a wider network for the encoder, resulting in improved performance over a dense input layer. 
See Section \ref{sec:dense_vs_sparse_encoder} for further details.

\subsection{Weak FOM-port formulation}
First we detail the architectures used for the weak FOM-port constraint formulation.
We use a single-layer architecture for the encoders and decoders with a smooth, non-polynomial activation function. 
The encoders, $\bh_i^\Omega$ and $\bh_i^\Gamma$, and decoders, $\bg_i^\Omega$ and $\bg_i^\Gamma$, are of the form
\begin{subequations}\label{eq:autoencoder_architecture}
    \begin{align}
        \bh_i^\Omega&:\real^{N_i^\Omega}\to\real^{n_i^\Omega}, &
        \bh_i^\Omega(\bx_i^\Omega) = \bW_2^{\bh_i^\Omega}\bsigma_i^\Omega(\bW_1^{\bh_i^\Omega} \bx_i^\Omega + \bb_1^{\bh_i^\Omega}), \\
        \bg_i^\Omega&:\real^{n_i^\Omega}\to\real^{N_i^\Omega}, &
        \bg_i^\Omega(\hbx_i^\Omega) = \bW_2^{\bg_i^\Omega}\bsigma_i^\Omega(\bW_1^{\bg_i^\Omega} \hbx_i^\Omega + \bb_1^{\bg_i^\Omega}), \\
        \bh_i^\Gamma&:\real^{N_i^\Gamma}\to\real^{n_i^\Gamma}, &
        \bh_i^\Gamma(\bx_i^\Gamma) = \bW_2^{\bh_i^\Gamma}\bsigma_i^\Gamma(\bW_1^{\bh_i^\Gamma} \bx_i^\Gamma + \bb_1^{\bh_i^\Gamma}), \\
        \bg_i^\Gamma&:\real^{n_i^\Gamma}\to\real^{N_i^\Gamma}, &
        \bg_i^\Gamma(\hbx_i^\Gamma) = \bW_2^{\bg_i^\Gamma}\bsigma_i^\Gamma(\bW_1^{\bg_i^\Gamma} \hbx_i^\Gamma + \bb_1^{\bg_i^\Gamma}), 
    \end{align}
\end{subequations}
where
\begin{subequations}
    \begin{align}
        \bW_1^{\bg_i^\Omega}, \left(\bW_2^{\bh_i^\Omega} \right)^T &\in \real^{w_i^\Omega \times n_i^\Omega},&
        \bW_2^{\bg_i^\Omega}, \left(\bW_1^{\bh_i^\Omega} \right)^T &\in \real^{N_i^\Omega \times w_i^\Omega}, \\
        \bb_1^{\bh_i^\Omega} & \in \real^{w_i^\Omega}, &
        \bb_1^{\bg_i^\Omega} & \in \real^{w_i^\Omega}, \\
        \bW_1^{\bg_i^\Gamma}, \left(\bW_2^{\bh_i^\Gamma} \right)^T &\in \real^{w_i^\Gamma\times n_i^\Gamma},&
        \bW_2^{\bg_i^\Gamma}, \left(\bW_1^{\bh_i^\Gamma} \right)^T &\in \real^{N_i^\Gamma \times w_i^\Gamma}, \\
        \bb_1^{\bh_i^\Gamma} & \in \real^{w_i^\Gamma}, &
        \bb_1^{\bg_i^\Gamma} & \in \real^{w_i^\Gamma}, 
    \end{align}
\end{subequations}
$\bsigma_i^\Omega, \bsigma_i^\Gamma$ are smooth, non-polynomial activation functions (e.g. Sigmoid or Swish), 
and where $w_i^\Omega, w_i^\Gamma$ are the network widths
for all subdomains $i=1,\dots, n_\Omega.$
The weight matrices $\bW_2^{\bg_i^\Omega}, \bW_{1}^{\bh_i^\Omega}, \bW_{2}^{\bg_i^\Gamma}$ and $\bW_1^{\bh_i^\Gamma}$ are all sparse, while the remaining weights and biases are dense. 

The widths, $w_i^\Omega$ and $w_i^\Gamma$, as well as the sparsity patterns of $\bW_2^{\bg_i^\Omega}, \bW_{1}^{\bh_i^\Omega}, \bW_{2}^{\bg_i^\Gamma}$ and $\bW_1^{\bh_i^\Gamma}$ are hyper-parameters that require tuning. 
The use of a single-layer architecture of arbitrary width and non-polynomial activation, as defined in (\ref{eq:autoencoder_architecture}a-d), is motivated by the well-known universal approximation theorem \cite{GCybenko_1989a}, \cite{APinkus_1999a}. 
Furthermore, the use of a smooth activation function ensures that $\bg_i^\Omega$ and $\bg_i^\Gamma$ are continuously differentiable. 
This is important because the Jacobians $\frac{d \bg_i^\Omega}{d \hbx_i^\Omega}$ and $\frac{d \bg_i^\Gamma}{d \hbx_i^\Gamma}$ are required by the SQP solver discussed in Section \ref{sec:sqp_solver}.
The autoencoders are trained by minimizing the MSE loss defined in equation \eref{eq:autoencoder_AbsMSE_loss}.

\subsection{Strong ROM-port formulation}
Next we detail the architectures used for the strong ROM-port constraint formulation.
As before, we use a single-layer architecture for the encoders and decoders with a smooth, non-polynomial activation function. 
The interior state encoders $\bh_i^\Omega$ and decoders $\bg_i^\Omega$ have the same architecture as in the weak FOM-port constraint formulation.
Thus we focus on the interface encoders $\bh_i^\Gamma$ and decoders $\bg_i^\Gamma$. 
As stated in Section \ref{sec:nmrom}, in the strong ROM-port case, the interface encoders $\bh_i^\Gamma$ and decoders $\bg_i^\Gamma$ are composed of encoders $\bh_j^p$ and decoders $\bg_j^p$ for ports $P(j)$.
These encoders $\bh_j^p$ and decoders $\bg_j^p$ are of the form 
\begin{subequations}\label{eq:autoencoder_architecture_port}
    \begin{align}
        \bh_j^p&:\real^{N_j^p}\to\real^{n_j^p}, &
        \bh_j^p(\bx_j^p) = \bW_2^{\bh_j^p}\bsigma_j^p(\bW_1^{\bh_j^p} \bx_j^p + \bb_1^{\bh_j^p}), \\
        \bg_j^p&:\real^{n_j^p}\to\real^{N_j^p}, &
        \bg_j^p(\hbx_j^p) = \bW_2^{\bg_j^p}\bsigma_j^p(\bW_1^{\bg_j^p} \hbx_j^p + \bb_1^{\bg_j^p}), 
    \end{align}
\end{subequations}
where
\begin{subequations}
    \begin{align}
        \bW_1^{\bg_j^p}, \left(\bW_2^{\bh_j^p} \right)^T &\in \real^{w_j^p \times n_j^p},&
        \bW_2^{\bg_j^p}, \left(\bW_1^{\bh_j^p} \right)^T &\in \real^{N_j^p \times w_j^p}, \\
        \bb_1^{\bh_j^p} & \in \real^{w_j^p}, &
        \bb_1^{\bg_j^p} & \in \real^{w_j^p}, 
    \end{align}
\end{subequations}
$\bsigma_j^p$ are smooth, non-polynomial activation functions (e.g., Sigmoid or Swish), 
and where $w_j^p$ are the network widths
for all ports $P(j), j=1,\dots, n_p.$
The weight matrices $\bW_2^{\bg_j^p}$ and $\bW_1^{\bh_j^p}$ are sparse, while the remaining weights and biases are dense. 
As in the WFPC case, the width $w_j^p$ and the sparsity patterns of $\bW_2^{\bg_j^p}$ and $\bW_1^{\bh_j^p}$ are hyper-parameters that require tuning. 
The autoencoders are trained by minimizing the MSE loss defined in equation \eref{eq:autoencoder_AbsMSE_loss}.


Recall that the interface encoders $\bh_i^\Gamma$ and $\bg_i^\Gamma$ are computed using equations \eref{eq:gamma_port_encoder} and \eref{eq:gamma_map_port}, respectively. 
The encoders $\bh_i^\Gamma$ and decoders $\bg_i^\Gamma$ can be written in the form \seref{eq:autoencoder_architecture}{c, d} as follows. 
For a given subdomain $i$, let $j_1, \dots, j_{|Q(i)|}$ denote the indices of the subdomains contained in $Q(i).$
The weights and biases of $\bh_i^\Gamma$ and $\bg_i^\Gamma$ can then be assembled in block form as
\begin{subequations}\label{eq:interface_decoder_rom_port_blocks}
    \begin{align}
        \bW_1^{\bh_i^\Gamma} &= 
        \begin{bmatrix}
            \bW_1^{\bh_{j_1}^p} & & \\
            & \ddots & \\
            & & \bW_1^{\bh_{j_{|Q(i)|}}^p}
        \end{bmatrix}
        \begin{bmatrix}
            \bP_i^{j_1} \\ \vdots \\ \bP_i^{j_{|Q(i)|}}
        \end{bmatrix},&
        \bb_1^{\bh_i^\Gamma} &= 
        \begin{bmatrix}
            \bb_1^{\bh_{j_1}^p} \\ \vdots \\ \bb_1^{\bh_{j_{|Q(i)|}}^p}
        \end{bmatrix}, \\
        \bW_2^{\bh_i^\Gamma} &= 
        \begin{bmatrix}
            (\hbP_i^{j_1})^T & \dots & (\hbP_i^{j_{|Q(i)|}})^T
        \end{bmatrix}
        \begin{bmatrix}
            \bW_2^{\bh_{j_1}^p} & & \\
            & \ddots & \\
            & & \bW_2^{\bh_{j_{|Q(i)|}}^p}
        \end{bmatrix}, \\
        \bW_1^{\bg_i^\Gamma} &= 
        \begin{bmatrix}
            \bW_1^{\bg_{j_1}^p} & & \\
            & \ddots & \\
            & & \bW_1^{\bg_{j_{|Q(i)|}}^p}
        \end{bmatrix}
        \begin{bmatrix}
            \hbP_i^{j_1} \\ \vdots \\ \hbP_i^{j_{|Q(i)|}}
        \end{bmatrix},&
        \bb_1^{\bg_i^\Gamma} &= 
        \begin{bmatrix}
            \bb_1^{\bg_{j_1}^p} \\ \vdots \\ \bb_1^{\bg_{j_{|Q(i)|}}^p}
        \end{bmatrix}, \\
        \bW_2^{\bg_i^\Gamma} &= 
        \begin{bmatrix}
            (\bP_i^{j_1})^T & \dots & (\bP_i^{j_{|Q(i)|}})^T
        \end{bmatrix}
        \begin{bmatrix}
            \bW_2^{\bg_{j_1}^p} & & \\
            & \ddots & \\
            & & \bW_2^{\bg_{j_{|Q(i)|}}^p}
        \end{bmatrix},
    \end{align}
    with activation
    \begin{equation}
        \bsigma_i^\Gamma(\cdot) = 
        \begin{bmatrix}
            \bsigma_{j_1}^p(\cdot) \\ \vdots \\ \bsigma_{j_{|Q(i)|}}^p(\cdot)
        \end{bmatrix}.
    \end{equation}
\end{subequations}

%% file: hyper_reduction.tex

\subsection{Hyper-Reduction} \label{sec:hyper_reduction}
\input{autoencoder_tikz.tex}

If no hyper-reduction (HR) is applied (i.e. $\bB_i=\bI$) when solving DD ROM \eref{eq:dd_lspg_NLP}, the computational savings from the ROM is limited because the evaluation of the residuals $\br_i$ and their Jacobians still scales with the dimension of the FOM.
Thus HR is applied to decrease the computational complexity caused by the nonlinearity of $\br_i$, and increase the computational speedup. 
Possible HR approaches include collocation ($\bB_i = \bZ_i$)
and gappy POD ($\bB_i = (\bZ_i \bPhi_i^r)^\dag \bZ_i$) 
\cite{REverson_LSirovich_1995}, 
\cite{CHoang_YChoi_KCarlberg_2021a}.
In both cases, only a subsample of the residual components and their corresponding Jacobian components are computed. 
This subsample is determined by the row-sampling matrix $\bZ_i$, which is typically computed greedily (see Remark \ref{rmk:greedy_hr}). 

Now for both cases $\bB_i = \bZ_i$ and $\bB_i = (\bZ_i \bPhi_i^r)^\dag \bZ_i$, one must compute the products $\bZ_i \br_i$, $\bZ_i \frac{\partial \br_i}{\partial \bx_i^\Omega}$, and $\bZ_i\frac{\partial \br_i}{\partial \bx_i^\Gamma}$.
In implementation, instead of computing matrix-vector or matrix-matrix products, 
one only needs to compute the entries of $\br_i$ and rows of $\frac{\partial \br_i}{\partial \bx_i^\Omega}$ and $\frac{\partial \br_i}{\partial \bx_i^\Gamma}$ that are sampled by $\bZ_i$.
Hence the application of HR is typically code-intrusive.
Moreover, since only a subset of the entries of $\bg_i^\Omega(\hbx_i^\Omega)$ and $\bg_i^\Gamma(\hbx_i^\Gamma)$ are needed to evaluate $\bZ_i \br_i$, $\bZ_i \frac{\partial \br_i}{\partial \bx_i^\Omega}$, and $\bZ_i\frac{\partial \br_i}{\partial \bx_i^\Gamma}$,
only the corresponding outputs of the decoders $\bg_i^{\Omega}$ and $\bg_i^\Gamma$ need to be kept track of. 
This motivates the use of a sparsity mask in the last layer of the decoders, which was first introduced in the context of NM-ROM in \cite{YKim_YChoi_DWidemann_TZohdi_2022a}. 

Indeed, in the case of a dense linear layer, each node in the hidden layer is needed to compute one node in the output layer, thus limiting the computational savings gained through HR.
If instead a sparsity mask is applied to the layer, only a subset of the hidden nodes is required to compute a node in the output layer. 
This allows for the computation of a {\it subnet}, which only keeps track of the nodes used to compute the output nodes remaining after HR. 
We discuss the computation of a subnet in Section \ref{sec:subnet}. 
Figure \ref{fig:autoencoder_architecture} provides a visualization of a subnet.
In this paper, we also apply the transpose of decoder sparsity mask to the input layer of the encoder, resulting in autoencoders whose architectures are (approximately) symmetric across the latent layer.
We found that this choice gave improved performance (i.e. ROM accuracy) over the architectures used in \cite{YKim_YChoi_DWidemann_TZohdi_2022a}, which use dense encoder input layers. 

\begin{remark}\label{rmk:greedy_hr}
    Following \cite{CHoang_YChoi_KCarlberg_2021a} and \cite{YKim_YChoi_DWidemann_TZohdi_2022a}, we use \cite[Algo. 3]{KTCarlberg_CFarhat_JCortial_DAmsallem_2013a} to greedily compute a row sampling matrix $\bZ_i$. This approach relies upon the computation of a residual basis $\bPhi_i^r$ for each subdomain. In practice, these residual bases are computed by applying POD to {\it residual snapshots}, which are collected from the iteration history of Newton's method when computing interior- and interface-state snapshots for ROM training. 
\end{remark}

\begin{remark}
    For the WFPC case, the products $\bC\bA_i\bg_i^\Gamma(\hbx_i^\Gamma)$ and
    $\hblambda^T\bC\bA_i\frac{d \bg_i^\Gamma}{d \hbx_i^\Gamma}(\hbx_i^\Gamma)$ 
    coming from the equality constraint appear in the KKT system \eref{eq:sqp_system} for the SQP solver. 
    For LS-ROM, since the Jacobian of $\bg_i^\Gamma$ is nothing but the POD basis matrix $\bPhi_i^\Gamma$, one can easily precompute $\bC\bA_i\bPhi_i^\Gamma$, thus making HR unnecessay for these quantities. 
    However for NM-ROM, $d \bg_i^\Gamma /d \hbx_i^\Gamma$ must be re-evaluated at each iteration of the SQP solver, 
    thus introducing additional computation expense that is not present in LS-ROM. 
    Currently, these quantities do not undergo HR in the NM-ROM case, and hence the decoders $\bg_i^\Gamma$ for the entire interface states must be kept track of. 
    While this limits the computational savings that can be obtained through HR, in practice the dimension of the FOM interface states $N_i^\Gamma$ is much smaller than the dimension of the FOM interior states $N_i^\Omega$, thus making the HR of $\bg_i^\Gamma$ less critical than the HR of $\bg_i^\Omega$. 
    This issue is not present in the SRPC case because the constraints are purely linear.
\end{remark}

\begin{remark}
    In practice, the pattern of the sparsity mask is determined by a number of hyper parameters to be tuned by the user. 
    Further details on the sparsity pattern used for our numerical results is discussed in Section \ref{sec:numerics}.
\end{remark}

\subsection{Construction of a Subnet}\label{sec:subnet}
The authors in \cite[Sec. 4.4.1]{YKim_YChoi_DWidemann_TZohdi_2022a} discuss the construction of a subnet in terms of gradients of the loss function with respect to the weights and biases of the sparse decoder. 
In this section, we present an alternative method for constructing the subnet solely by keeping track of indices of HR nodes. 
For simplicity, we consider a generic decoder $\bg:\real^n\to \real^N$ of the form 
(\ref{eq:autoencoder_architecture}b, d) with width $w$. Computing subnets for each decoder, e.g., $\bg_i^\Omega$ and $\bg_i^\Gamma$, follows the same procedure. Recall that in the architecture considered in this paper, $\bW_2\in \real^{N\times 
 w}$ is sparse and $\bW_1\in \real^{w\times n}$ is dense. 

Let $\cI_{o} \subset \set{1, \dots, N}$ denote the indices of the outputs of $\bg$ that are selected through HR. To find the indices of the hidden nodes required to compute the HR nodes, find the index set $\cI_h$ where
$$\cI_h=\set{j \in \set{1, \dots, w} \; | \; \exists \; i \in \cI_o \; {\rm s.t.} \; (\bW_2)_{ij} \neq 0}.$$
The index sets $\cI_o$ and $\cI_h$ contain the indices of all nonzero elements in $\bW_2$. 
Now let 
$i_1, \dots, i_{|\cI_o|}$ and $j_1, \dots, j_{|\cI_h|}$ denote the elements of $\cI_o$ and $\cI_h$ in ascending order, respectively.
Define the matrix $\tbW_2 \in \real^{|\cI_o|\times |\cI_h|}$ as 
$$
(\tbW_2)_{\ell, k} = (\bW_2)_{i_\ell, j_k}, \quad \forall \; 
\ell = 1, \dots, |\cI_o|, \;
k = 1, \dots, |\cI_h|. 
$$
The matrix $\tbW_2$ precisely consists of the connections in the subnet that remain after HR. 
Next, since the activation $\bsigma$ acts element-wise, the connections that remain in the first layer of the subnet can be represented by $\tbW_1\in \real^{|\cI_h|\times n}$, which consists of nothing but the rows in $\bW_1$ corresponding to the index set $\cI_h$:
$$
(\tbW_1)_{k, :} = \bW_{j_k, :}, \quad \forall \; k = 1, \dots, |\cI_h|. 
$$
Lastly, the subnet bias $\tbb_1 \in \real^{|\cI_h|}$ is similarly defined as 
$(\tbb_1)_k = (\bb_1)_{j_k}$ for all $k=1, \dots, |\cI_h|$. 
The subnet $\tbg:\real^n\to \real^{|\cI_o|}$ is then defined as 
\begin{equation}\label{eq:subnet}
    \tbg(\hbx) = \tbW_2 \bsigma (\tbW_1 \hbx + \tbb_1). 
\end{equation}

\begin{remark}
    This framework for computing a subnet can easily be extended to neural networks with arbitrarily many sparse linear layers provided that the sparsity patterns for each layer's weight matrix is known. Thus one could construct deep sparse autoencoders with narrower width than the architectures considered here. However, for this paper we only consider single-layer, wide, sparse decoders. 
\end{remark}

%% file: autoencoder_tikz.tex
\def\layersep{1.7cm}
\def\nodesep{0.8}
\def\nodesize{18}
\def\hrcolor{blue!65}
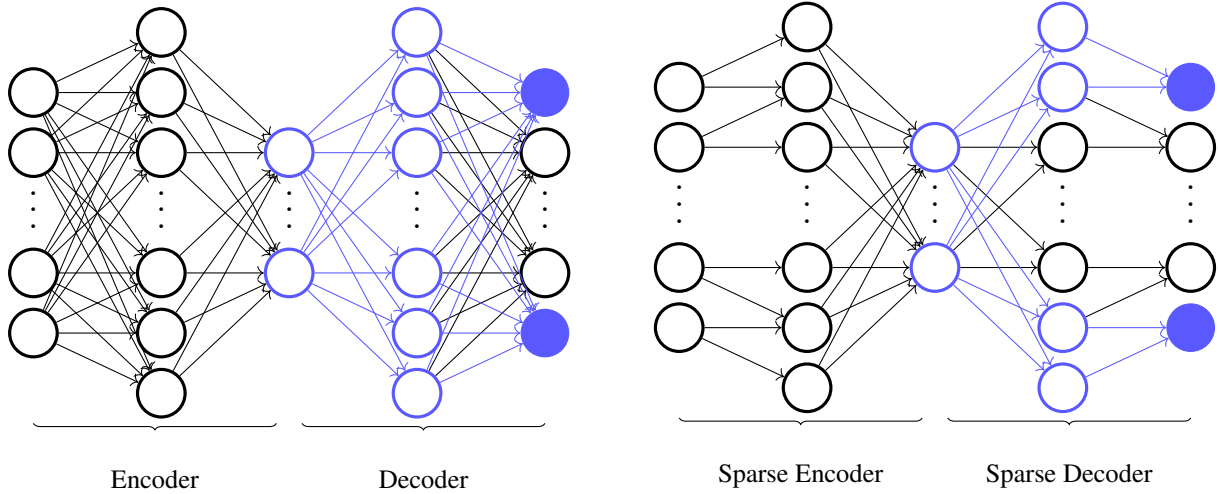
\begin{figure}[ht]
    \centering
    \begin{subfigure}[b]{0.48\textwidth}
    \begin{tikzpicture}[draw=black, node distance=0.05cm]
    \tikzstyle{neuron}=[circle, draw=black, very thick, minimum size=\nodesize,inner sep=0pt]
    \tikzstyle{hr_hidden}=[circle, draw=\hrcolor, very thick, minimum size=\nodesize,inner sep=0pt]
    \tikzstyle{hr_out}=[circle, draw=\hrcolor, fill=\hrcolor, minimum size=\nodesize,inner sep=0pt]
    
    \foreach \name / \y in {1, ..., 2}
        {\node[neuron] (I+\name) at (0, \y*\nodesep) {} ;
        \node[neuron] (I-\name) at (0, -\y*\nodesep) {} ;}
    \path (I+1)--(I-1) node [black, font=\Large, midway, sloped] {$\dots$};

    \foreach \name / \y in {1, ..., 3}
        {\node[neuron] (EH+\name) at (\layersep, \y*\nodesep) {} ;
        \node[neuron] (EH-\name) at (\layersep, -\y*\nodesep) {} ;}
    \path (EH+1)--(EH-1) node [black, font=\Large, midway, sloped] {$\dots$};

    \node[hr_hidden] (L+1) at (2*\layersep, 1*\nodesep) {} ;
    \node[hr_hidden] (L-1) at (2*\layersep, -1*\nodesep) {} ;
    \path (L+1)--(L-1) node [black, font=\Large, midway, sloped] {$\dots$};

    \foreach \name / \y in {1, ..., 3}
        {\node[hr_hidden] (DH+\name) at (3*\layersep, \y*\nodesep) {} ;
        \node[hr_hidden] (DH-\name) at (3*\layersep, -\y*\nodesep) {} ;}
    \path (DH+1)--(DH-1) node [black, font=\Large, midway, sloped] {$\dots$};

    \node[neuron] (O+1) at (4*\layersep, 1*\nodesep) {} ;
    \node[neuron] (O-1) at (4*\layersep, -1*\nodesep) {} ;
    \node[hr_out] (O+2) at (4*\layersep, 2*\nodesep) {} ;
    \node[hr_out] (O-2) at (4*\layersep, -2*\nodesep) {} ;
    \path (O+1)--(O-1) node [black, font=\Large, midway, sloped] {$\dots$};

    \foreach \source in {1, ..., 2}
        \foreach \dest in {1, ..., 3}
            {\draw[->] (I+\source) -- (EH+\dest); 
            \draw[->] (I-\source) -- (EH+\dest);
            \draw[->] (I+\source) -- (EH-\dest);
            \draw[->] (I-\source) -- (EH-\dest);}
            
    \foreach \source in {1, ..., 3}
        {\draw[->] (EH+\source) -- (L+1); 
        \draw[->] (EH+\source) -- (L-1);
        \draw[->] (EH-\source) -- (L+1); 
        \draw[->] (EH-\source) -- (L-1);}
        
    \foreach \dest in {1, ..., 3}
        {\draw[->, \hrcolor] (L+1) -- (DH+\dest); 
        \draw[->, \hrcolor] (L-1) -- (DH+\dest);
        \draw[->, \hrcolor] (L+1) -- (DH-\dest); 
        \draw[->, \hrcolor] (L-1) -- (DH-\dest);} 
        
    \foreach \source in {1, ..., 3}
        {\draw[->] (DH+\source) -- (O+1); 
        \draw[->] (DH+\source) -- (O-1);
        \draw[->] (DH-\source) -- (O+1); 
        \draw[->] (DH-\source) -- (O-1);} 
    \foreach \source in {1, ..., 3}
        {\draw[->, \hrcolor] (DH+\source) -- (O+2); 
        \draw[->, \hrcolor] (DH+\source) -- (O-2);
        \draw[->, \hrcolor] (DH-\source) -- (O+2); 
        \draw[->, \hrcolor] (DH-\source) -- (O-2);} 

    \draw[decorate, decoration={brace, mirror}] (0, -3.5*\nodesep) -- (1.9*\layersep, -3.5*\nodesep) 
    node[pos=0.5, below=14pt, black]{Encoder};
    \draw[decorate, decoration={brace, mirror}] (2.1*\layersep, -3.5*\nodesep) -- (4*\layersep, -3.5*\nodesep) 
    node[pos=0.5, below=14pt, black]{Decoder};
    \end{tikzpicture}
    \end{subfigure}
    \hfill
    \begin{subfigure}[b]{0.48\textwidth}
    \begin{tikzpicture}[draw=black, node distance=0.05cm]
    \tikzstyle{neuron}=[circle, draw=black, very thick, minimum size=\nodesize,inner sep=0pt]
    \tikzstyle{hr_hidden}=[circle, draw=\hrcolor, very thick, minimum size=\nodesize,inner sep=0pt]
    \tikzstyle{hr_out}=[circle, draw=\hrcolor, fill=\hrcolor, minimum size=\nodesize,inner sep=0pt]

    \foreach \name / \y in {1, ..., 2}
        {\node[neuron] (I+\name) at (0, \y*\nodesep) {} ;
        \node[neuron] (I-\name) at (0, -\y*\nodesep) {} ;}
    \path (I+1)--(I-1) node [black, font=\Large, midway, sloped] {$\dots$};

    \foreach \name / \y in {1, ..., 3}
        {\node[neuron] (EH+\name) at (\layersep, \y*\nodesep) {} ;
        \node[neuron] (EH-\name) at (\layersep, -\y*\nodesep) {} ;}
    \path (EH+1)--(EH-1) node [black, font=\Large, midway, sloped] {$\dots$};

    \node[hr_hidden] (L+1) at (2*\layersep, 1*\nodesep) {} ;
    \node[hr_hidden] (L-1) at (2*\layersep, -1*\nodesep) {} ;
    \path (L+1)--(L-1) node [black, font=\Large, midway, sloped] {$\dots$};

    \node[neuron] (DH+1) at (3*\layersep, 1*\nodesep) {};
    \node[neuron](DH-1) at (3*\layersep, -1*\nodesep) {};
    \node[neuron] (O+1) at (4*\layersep, 1*\nodesep) {};
    \node[neuron] (O-1) at (4*\layersep, -1*\nodesep) {};
    
    \foreach \name / \y in {2, ..., 3}
        {\node[hr_hidden] (DH+\name) at (3*\layersep, \y*\nodesep) {} ;
        \node[hr_hidden] (DH-\name) at (3*\layersep, -\y*\nodesep) {} ;}
    \path (DH+1)--(DH-1) node [black, font=\Large, midway, sloped] {$\dots$};

    \node[hr_out] (O+2) at (4*\layersep, 2*\nodesep) {} ;
    \node[hr_out] (O-2) at (4*\layersep, -2*\nodesep) {} ;
    \path (O+1)--(O-1) node [black, font=\Large, midway, sloped] {$\dots$};

    \foreach \dest in {1, ..., 2}
        {\draw[->] (I+1) -- (EH+\dest);
        \draw[->]  (I-1) -- (EH-\dest);}
    \foreach \dest in {2, ..., 3}
        {\draw[->] (I+2) -- (EH+\dest);
        \draw[->]  (I-2) -- (EH-\dest);}
            
    \foreach \source in {1, ..., 3}
        {\draw[->] (EH+\source) -- (L+1); 
        \draw[->] (EH+\source) -- (L-1);
        \draw[->] (EH-\source) -- (L+1); 
        \draw[->] (EH-\source) -- (L-1);}

    \draw[->] (L-1) -- (DH+1);
    \draw[->] (L+1) -- (DH+1);
    \draw[->] (L-1) -- (DH-1);
    \draw[->] (L+1) -- (DH-1);
    \foreach \dest in {2, ..., 3}
        {\draw[->, \hrcolor] (L+1) -- (DH+\dest); 
        \draw[->, \hrcolor] (L-1) -- (DH+\dest);
        \draw[->, \hrcolor] (L+1) -- (DH-\dest); 
        \draw[->, \hrcolor] (L-1) -- (DH-\dest);} 
        
    \foreach \source in {1, ..., 2}
        {\draw[->] (DH+\source) -- (O+1);
        \draw[->] (DH-\source) -- (O-1);}
    \foreach \source in {2, ..., 3}
        {\draw[->, \hrcolor] (DH+\source) -- (O+2);
        \draw[->, \hrcolor] (DH-\source) -- (O-2);}
        
    \draw[decorate, decoration={brace, mirror}] (0, -3.5*\nodesep) -- (1.9*\layersep, -3.5*\nodesep) 
    node[pos=0.5, below=14pt, black]{Sparse Encoder};
    \draw[decorate, decoration={brace, mirror}] (2.1*\layersep, -3.5*\nodesep) -- (4*\layersep, -3.5*\nodesep) 
    node[pos=0.5, below=14pt, black]{Sparse Decoder};
    \end{tikzpicture}
    \end{subfigure}
    \caption{Left: Dense autoencoder. The HR nodes are represented by solid blue neurons, and the nodes required to compute the HR nodes are outlined in blue. Notice that each node in the decoder hidden layer are required to compute the HR nodes in the output layer.
    Right: Sparse autoencoder. The encoder input layer and decoder output layer are sparsely connected, and only the blue-outlined hidden nodes are required to compute the HR nodes. The sparse output layer allows one to only keep track of the blue connections to evaluate $\bg_i^\Omega$, $\bg_i^\Gamma$ and their Jacobians, resulting in computational speedup.}
    \label{fig:autoencoder_architecture}
\end{figure}

%% file: error_analysis.tex

\section{Error Analysis}\label{sec:error_analysis}
We present {\it a priori} and {\it a posteriori} error bounds analogous to those found in \cite{CHoang_YChoi_KCarlberg_2021a}. 
To simplify notation, analogous to the notation in Section~\ref{sec:sqp_convergence}, we denote the optimal solutions to 
the FOM \eref{eq:fom_dd_nlp}, to the ROM \eref{eq:dd_lspg_NLP_general}, 
and  the ROM solution lifted to the FOM state space as
\begin{align}\label{eq:error_notation}
    \bx^* &= 
    \begin{bmatrix}
        \bx_1^{\Omega*}\\ \bx_1^{\Gamma*} \\ \vdots \\
        \bx_{n_\Omega}^{\Omega*}\\ \bx_{n_\Omega}^{\Gamma*}
    \end{bmatrix} \in \real^{N_{D}}, &
    \hbx^* &= 
        \begin{bmatrix}
            \hbx_1^{\Omega*}\\ \hbx_1^{\Gamma*} \\ \vdots \\
            \hbx_{n_\Omega}^{\Omega*}\\ \hbx_{n_\Omega}^{\Gamma*}
        \end{bmatrix} \in \real^{n_{D}}, &
    \bg(\hbx^*) &= 
    \begin{bmatrix}
        \bg_1^\Omega(\hbx_1^{\Omega*})\\ \bg_1^\Gamma(\hbx_1^{\Gamma*}) \\ \vdots \\
        \bg_{n_\Omega}^\Omega(\hbx_{n_\Omega}^{\Omega*})\\ \bg_{n_\Omega}^\Gamma(\hbx_{n_\Omega}^{\Gamma*})
    \end{bmatrix} \in \real^{N_{D}},
\end{align}
respectively,
where $N_{D}=\sum_{i=1}^{n_\Omega} (N_i^\Omega + N_i^\Gamma)$ and, as before, 
$n_{D}=\sum_{i=1}^{n_\Omega} (n_i^\Omega + n_i^\Gamma)$.
We also define the FOM constraint matrix 
\begin{equation}\label{eq:constraint_mat_large}
    \bA = 
    \begin{bmatrix}
        \bzero & \bA_1 & \dots & \bzero & \bA_{n_\Omega}
    \end{bmatrix} \in \real^{N_{A}\times N_D}
\end{equation}
so that the constraints (\ref{eq:fom_dd_nlp}b) can be written as $\bA\bx = \bzero$.
As in Section~\ref{sec:sqp_solver}, we define the constraint functions $\tbA_i: \real^{n_i^\Gamma} \to \real^{n_A}$,
where $\tbA_i(\hbx_i^\Omega) = \bC \bA_i \bg_i^\Gamma(\hbx_i^\Gamma)$ in the WFPC case \eref{eq:dd_lspg_NLP}
and $\tbA_i(\hbx_i^\Gamma) = \hbA_i \hbx_i^\Gamma$ in the SRPC case  \eref{eq:dd_lspg_NLP_rom_port},
so that the DD-LSPG-ROMs  \eref{eq:dd_lspg_NLP} and  \eref{eq:dd_lspg_NLP_rom_port} can be written as
\eref{eq:dd_lspg_NLP_general}.
We define the ROM constraint function $\tbA:\real^{n_D} \to \real^{n_A}$ as 
\begin{equation}\label{eq:constraint_function_rom}
    \tbA(\hbx) = \sum_{i=1}^{n_\Omega}\tbA_i(\hbx_i^\Gamma),
\end{equation}
so that the constraints  (\ref{eq:dd_lspg_NLP_general}b) can be written as $\tbA(\hbx) = \bzero$.
Lastly we define the feasible set 
\begin{align}\label{eq:feasible_sets}
    \cS_{\rm ROM} &= \set{\hbx \in \real^{n_D} \; : \; \tbA(\hbx) = \bzero}
\end{align}
for \eref{eq:dd_lspg_NLP_general}.

The next two results provide basic error bounds between a solution to the FOM (\ref{eq:fom_dd}) 
and solutions to the DD-LSPG-ROM   \eref{eq:dd_lspg_NLP} or  \eref{eq:dd_lspg_NLP_rom_port}.

\begin{theorem}[{\it A posteriori}  error bound]\label{thm:a_posteriori_bound}
         Let $\bx^*\in \real^{N_D}$ be a solution to the FOM (\ref{eq:fom_dd}) and let
         $\hbx^*\in \real^{n_D}$ be a (local) solution to the DD-LSPG-ROM   \eref{eq:dd_lspg_NLP} or 
         \eref{eq:dd_lspg_NLP_rom_port}. 
        If the residual is inverse Lipschitz continuous, that is, if there exists $\kappa_\ell > 0$ such that 
        \begin{subequations}\label{eq:error_bound_assumptions_1}
        \begin{equation}
            \left(\sum_{i=1}^{n_\Omega}\norm{\br_i(\by_i^\Omega, \by_i^\Gamma)-\br_i(\bz_i^\Omega, \bz_i^\Gamma)}_2^2\right)^{1/2}
            \geq \kappa_\ell \norm{\by - \bz}_2 \qquad \forall \; \by, \bz \in \real^{N_D},
        \end{equation}
        and if there exists $P > 0$ such that 
        \begin{equation}
            \left(\sum_{i=1}^{n_\Omega}\norm{\bB_i\br_i(\by_i^\Omega, \by_i^\Gamma)}_2^2\right)^{1/2}
            \geq P
            \left(\sum_{i=1}^{n_\Omega}\norm{\br_i(\by_i^\Omega, \by_i^\Gamma)}_2^2\right)^{1/2}
            \qquad \forall \; \by \in \bg(\cS_{\rm ROM}),
        \end{equation}
        \end{subequations}
        then 
        \begin{align}\label{eq:a_posteriori_bound}
            \norm{\bx^* - \bg(\hbx^*)}_2 
             \leq \frac{1}{P\kappa_\ell}
            \left(\sum_{i=1}^{n_\Omega}\norm{\bB_i\br_i(\bg_i^\Omega(\hbx^{\Omega *}), \bg_i^\Gamma(\hbx^{\Gamma *})}_2^2\right)^{1/2}. 
        \end{align}
\end{theorem}
\begin{proof}
    Using (\ref{eq:error_bound_assumptions_1}a) and the fact  that $\bx^*\in \real^{N_D}$ solves the FOM (\ref{eq:fom_dd})
    gives
    \begin{align*}
        \norm{\bx^*-\bg(\hbx^*)}_2^2 
        \leq \frac{1}{\kappa_\ell^2}
        \sum_{i=1}^{n_\Omega}\norm{\br_i(\bx_i^{\Omega *}, \bx_i^{\Gamma *})-\br_i(\bg_i^\Omega(\hbx_i^{\Omega *}), \bg_i^\Gamma(\hbx_i^{\Gamma *}))}_2^2
         \leq \frac{1}{\kappa_\ell^2}
        \sum_{i=1}^{n_\Omega}\norm{ \br_i(\bg_i^\Omega(\hbx_i^{\Omega *}), \bg_i^\Gamma(\hbx_i^{\Gamma *}))}_2^2.
    \end{align*}
    Applying (\ref{eq:error_bound_assumptions_1}b) with
    $(\by_i^\Omega, \by_i^\Gamma) = (\bg_i^\Omega(\hbx_i^{\Omega *}), \bg_i^\Gamma(\hbx_i^{\Gamma *}))$ gives
    the desired result.
\end{proof}

\begin{theorem}[{\it A priori}  error bound]\label{thm:a_priori_bound}
    Let $\bx^*\in \real^{N_D}$ be a solution to the FOM (\ref{eq:fom_dd}) and let
         $\hbx^*\in \real^{n_D}$ be a solution to the DD-LSPG-ROM   \eref{eq:dd_lspg_NLP} or 
         \eref{eq:dd_lspg_NLP_rom_port}. 
     If the inequalities (\ref{eq:error_bound_assumptions_1}a, b) hold and the HR residual is  Lipschitz continuous, i.e., 
     there exists $\kappa_u>0$ such that 
    \begin{equation}\label{eq:lipschitz_assumption}
        \left(\sum_{i=1}^{n_\Omega}\norm{\bB_i\br_i(\by_i^\Omega, \by_i^\Gamma)-\bB_i\br_i(\bz_i^\Omega, \bz_i^\Gamma)}_2^2\right)^{1/2} 
        \leq \kappa_u\norm{\by-\bz}_2 \qquad \forall \; \by, \bz \in \real^{N_D},
    \end{equation}
    then 
    \begin{equation}\label{eq:a_priori_bound}
        \norm{\bx^* - \bg(\hbx^*)}_2 \leq 
        \frac{\kappa_u}{P\kappa_\ell} \inf_{\hbw \in \cS_{\rm ROM}}\norm{\bx^*-\bg(\hbw)}_2.
    \end{equation}
\end{theorem}

\begin{proof}
     Since  $\hbx^*\in \real^{n_D}$ is a solution to the DD-LSPG-ROM   \eref{eq:dd_lspg_NLP} or 
     \eref{eq:dd_lspg_NLP_rom_port}, any feasible $\hbw$, i.e., any  $\hbw \in \cS_{\rm ROM}$
     satisfies
    \begin{equation}\label{eq:a_priori_bound_proof1}
     \sum_{i=1}^{n_\Omega}\norm{\bB_i \br_i(\bg_i^\Omega(\hbx_i^{\Omega *}), \bg_i^\Gamma(\hbx_i^{\Gamma *}))}_2^2
      \leq \sum_{i=1}^{n_\Omega}\norm{\bB_i \br_i(\bg_i^\Omega(\hbw_i^{\Omega *}), \bg_i^\Gamma(\hbw_i^{\Gamma *}))}_2^2.
     \end{equation}
     Moreover, since $\bx^*\in \real^{N_D}$ solves the FOM (\ref{eq:fom_dd}), 
     $\br_i(\bx_i^{\Omega *}, \bx_i^{\Gamma *}) = \bzero$, for all $i=1, \dots, n_\Omega$,
     \eref{eq:lipschitz_assumption} and  \eref{eq:a_priori_bound_proof1} imply
    \begin{align*}
     \sum_{i=1}^{n_\Omega}\norm{\bB_i \br_i(\bg_i^\Omega(\hbx_i^{\Omega *}), \bg_i^\Gamma(\hbx_i^{\Gamma *}))}_2^2
    &\leq
      \sum_{i=1}^{n_\Omega}\norm{\bB_i\br_i(\bx_i^{\Omega *}, \bx_i^{\Gamma *})-\bB_i \br_i(\bg_i^\Omega(\hbw_i^{\Omega}),    
                                                                     \bg_i^\Gamma(\hbw_i^{\Gamma}))}_2^2 \\
    &\leq \kappa_u  \norm{\bx^* - \bg(\hbw)}_2 \qquad \mbox{ for all }  \hbw \in \cS_{\rm ROM}.
    \end{align*}
     Combining this result with the a-posterior bound \eref{eq:a_posteriori_bound} in Theorem \ref{thm:a_posteriori_bound} yields 
     $\norm{\bx^* - \bg(\hbx^*)}_2 \leq   \frac{\kappa_u}{P\kappa_\ell} \norm{\bx^*-\bg(\hbw)}_2$ for all $\hbw \in \cS_{\rm ROM}$,
    which implies \eref{eq:a_priori_bound}.
\end{proof}

\begin{remark}
As a consequence of Theorem \ref{thm:a_priori_bound}, if (\ref{eq:error_bound_assumptions_1}a, b) and \eref{eq:lipschitz_assumption} hold, and if $\bx^*$ is in the image of the $\bg$ over the feasible set $\cS_{\rm ROM}$ of (\ref{eq:dd_lspg_NLP}), i.e. if $\bx^*\in \bg(\cS_{\rm ROM})$, then $\bx^* = \bg(\hbx^*)$.
\end{remark}

The error bounds in Theorems \ref{thm:a_posteriori_bound} and \ref{thm:a_priori_bound} only involve the FOM and ROM
states, but not the Lagrange multipliers. However, the present error bounds require stronger assumptions such as 
(\ref{eq:error_bound_assumptions_1}a). 
Alternatively, one could try to extend the error analysis for ROMs applied to nonlinear systems.
such as those in \cite[Sec.~11.5]{AQuarteroni_AManzoni_FNegri_2016a}, \cite{ASchmidt_DWittwar_BHaasdonk_2020a}.
In the context of  the FOM (\ref{eq:fom_dd}) and  the DD-LSPG-ROM   \eref{eq:dd_lspg_NLP} or 
 \eref{eq:dd_lspg_NLP_rom_port} the role of the nonlinear residual in \cite[Sec.~11.5]{AQuarteroni_AManzoni_FNegri_2016a}, \cite{ASchmidt_DWittwar_BHaasdonk_2020a} would now be played by the system of first order necessary optimality conditions,
 given by \eref{eq:dd_nm_lspg_1st_order_opt} for the general ROM formulation \eref{eq:dd_lspg_NLP_general} and correspondingly
 for the  FOM (\ref{eq:fom_dd}). 
 However, these residuals involve the FOM states and Lagrange multipliers associated with  (\ref{eq:fom_dd}b),
 and ROM states and Lagrange multipliers associated with (\ref{eq:dd_lspg_NLP_general}b). Moreover, this
 analysis requires to relate the ROM states with the FOM states and the ROM Lagrange multipliers with the FOM Lagrange multipliers.
 The former is done via $ \bg(\hbx) \approx  \bx^*$. However, the connection between the Lagrange multipliers in the general
 case is still open. For linear PDEs and a PDE-based (as opposed to our algebraic) DD formulation, 
 \cite{AdeCastro_PBochev_PKuberry_ITezaur_2023a}  construct appropriate, so called trace-compatible
 reduced bases for the Lagrange multipliers (see e.g., \cite[Eq.~(14)]{AdeCastro_PBochev_PKuberry_ITezaur_2023a}).
 In the context of \cite{AdeCastro_PBochev_PKuberry_ITezaur_2023a}, the  reduced bases for the Lagrange multipliers
 determine the ROM constraints (\ref{eq:dd_lspg_NLP_general}b).
 In our setting we first derive ROM constraints (\ref{eq:dd_lspg_NLP_general}b), which yield ROM Lagrange multipliers,
 but no explicit construction of a reduced bases for these ROM Lagrange multipliers. This is subject of future research.

%% file: numerics.tex

\section{Numerics}\label{sec:numerics}
We apply LS-ROM and NM-ROM with and without HR to the DD ROM with WFPC \eref{eq:dd_lspg_NLP} and with SRPC \eref{eq:dd_lspg_NLP_rom_port} for the 2D steady-state Burgers equation.
We use the following formula for computing the relative error between the FOM and ROM solutions:
\begin{equation}\label{eq:relative_error}
    e = \left(\frac{1}{n_\Omega}\sum_{i=1}^{n_\Omega} \frac{\norm{\bx_i^\Omega-\bg_i^\Omega(\hbx_i^\Omega)}_2^2+\norm{\bx_i^\Gamma-\bg_i^\Gamma(\hbx_i^\Gamma)}_2^2}{\norm{\bx_i^\Omega}_2^2+\norm{\bx_i^\Gamma}_2^2}\right)^{1/2}.
\end{equation}
The autoencoder training and subsequent computations in this section were performed on the Lassen machine at Lawrence Livermore National Laboratory, which consists of an IBM Power9 processor with NVIDIA V100 (Volta) GPUs, clock speed between 2.3-3.8 GHz, and 256 GB DDR4 memory. 

The implementation of the DD FOM, DD LS-ROM, and DD NM-ROM is done sequentially. 
However, to highlight potential advantages of a parallel implementation, the recorded wall clock time for the computation of the 
subdomain-specific quantities required by the SQP solver
is taken to be the largest wall clock time incurred among all subdomains. 
The wall clock time for the remaining steps of the SQP solver (e.g. assembling and solving the KKT system \eref{eq:sqp_system}, updating the interior- and interface-states and Lagrange multipliers \eref{eq:sqp_update}, etc.) is set to the overall wall clock time to execute the steps.
\subsection{2D Burgers' Equation}
In this experiment, we consider the 2D steady-state Burgers' equation on the domain $[-1, 1] \times [0, 0.05]$
\begin{subequations}\label{eq:burgers_pde_2d}
    \begin{align}
        u \frac{\partial u}{\partial x} + v \frac{\partial u}{\partial y} &= \nu \left(\frac{\partial^2 u}{\partial x^2} + \frac{\partial^2 u}{\partial y^2}\right), 
        \qquad (x, y) \in [-1, 1] \times [0, 0.05],
        \\
        u \frac{\partial v}{\partial x} + v \frac{\partial v}{\partial y} &= \nu \left(\frac{\partial^2 v}{\partial x^2} + \frac{\partial^2 v}{\partial y^2}\right),
        \qquad (x, y) \in [-1, 1] \times [0, 0.05],
    \end{align}
    where $\nu > 0$ is the viscosity. As in \cite{CHoang_YChoi_KCarlberg_2021a}, we consider the following exact solution, and its restriction to the boundary as Dirichlet boundary conditions: 
\end{subequations}
\begin{subequations}\label{eq:burgers_pde_exact_solution}
    \begin{align}
        u_{ex}(x, y; a, \lambda) &= -2\nu \left[a + \lambda\left(e^{\lambda(x-1)} - e^{-\lambda(x-1)}\right)\cos(\lambda y)\right]/\psi(x, y; a, \lambda), \\
        v_{ex}(x, y; a, \lambda) &= 2\nu \left[\lambda\left(e^{\lambda(x-1)} + e^{-\lambda(x-1)}\right)\sin(\lambda y)\right]/\psi(x, y; a, \lambda), 
    \end{align}
    where 
    \begin{equation}
        \psi(x, y; a, \lambda) = a(1+x) + \left(e^{\lambda(x-1)} + e^{-\lambda(x-1)}\right)\cos(\lambda y),
    \end{equation}
    and where $(a, \lambda)$ are parameters. 
\end{subequations}
The PDE is discretized using centered finite differences with $n_x+2$ uniformly spaced grid points in the $x$-direction and $n_y+2$ uniformly spaced grid points in the $y$-direction, resulting in grid points $(x_i, y_j)$ where 
\begin{align*}
x_i &= -1 + i h_x, & i &= 0, \dots, n_x+1, \\
y_j &=  j h_y, & j &= 0, \dots, n_y+1,
\end{align*}
where $h_x = 2/(n_x+1)$ and $h_y = 0.05/(n_y+1)$.
The solutions $u, v$ on the grid points are denoted $u_{ij} \approx u(x_i, y_j)$ and $v_{ij} \approx v(x_i, y_j)$. 
The PDE is then discretized using centered finite differences for the first and second derivative terms. The fully discretized system is given by
\begin{subequations}\label{eq:burgers_discretized}
    \begin{align}
        \bzero &= \br_u(\bu, \bv) = \bu \odot (\bB_x \bu - \bb_{u, x}) + \bv\odot(\bB_y \bu - \bb_{u, y}) + \bC \bu + \bc_u, \\
        \bzero &= \br_v(\bu, \bv) = \bu \odot (\bB_x \bv - \bb_{v, x}) + \bv\odot(\bB_y \bv - \bb_{v, y}) + \bC \bv + \bc_v, 
    \end{align}
\end{subequations}
where $\odot$ represents the Hadamard product, and where
\begin{subequations}
    \begin{align}
        \bu &=
        \begin{bmatrix}
            \bu^{[1]} \\ \vdots \\ \bu^{[n_y]}
        \end{bmatrix} \in \real^{n_xn_y}, \qquad 
        \bu^{[j]} = 
        \begin{bmatrix}
            u_{1, j} \\ \vdots \\ u_{n_x, j}
        \end{bmatrix} \in \real^{n_x}, \quad j = 1, \dots, n_y,\\
        \bv &=
        \begin{bmatrix}
            \bv^{[1]} \\ \vdots \\ \bv^{[n_y]}
        \end{bmatrix} \in \real^{n_xn_y}, \qquad 
        \bv^{[j]} = 
        \begin{bmatrix}
            v_{1, j} \\ \vdots \\ v_{n_x, j}
        \end{bmatrix}\in \real^{n_x}, \quad j = 1, \dots, n_y,
    \end{align}
    \begin{align}
        \bB_x &= -\frac{1}{2h_x} \left(\bI_{n_y} \otimes \tbB_x\right)\in \real^{n_xn_y \times n_x n_y}, \qquad
        \tbB_x = 
        \begin{bmatrix}
            0 & 1 \\
            -1 & \ddots & 1 \\
            & -1 & 0
        \end{bmatrix} \in \real^{n_x \times n_x}, \\
        \bB_y &= -\frac{1}{2h_y}\left(\tbB_y \otimes \bI_{n_x}\right) \in \real^{n_xn_y \times n_x n_y} \qquad
        \tbB_y = 
        \begin{bmatrix}
            0 & 1 \\
            -1 & \ddots & 1 \\
            & -1 & 0
        \end{bmatrix} \in \real^{n_y \times n_y}, \\
        \bC &= \frac{\nu}{h_x^2}\left(\bI_{n_y} \otimes \tbC_x\right) + \frac{\nu}{h_y^2} \left(\tbC_y \otimes \bI_{n_x}\right) \in \real^{n_xn_y \times n_x n_y}, \\
        \tbC_x &= 
        \begin{bmatrix}
            -2 & 1 \\
            1 & \ddots \\
            & 1 & -2 \\
        \end{bmatrix} \in \real^{n_x \times n_x}, \qquad
        \tbC_y =
        \begin{bmatrix}
            -2 & 1 \\
            1 & \ddots \\
            & 1 & -2 \\
        \end{bmatrix} \in \real^{n_y \times n_y}, 
    \end{align}
    \begin{align}
        \bb_{u,x} &= -\frac{1}{2h_x} (\bb_{ux\ell} - \bb_{uxr}), \qquad
        \bb_{u, y} = -\frac{1}{2h_y}(\bb_{uy\ell} - \bb_{uyr}), \\
        \bc_u &= \frac{\nu}{h_x^2}(\bb_{ux\ell} + \bb_{uxr}) + \frac{\nu}{h_y^2} (\bb_{uy\ell} + \bb_{uyr}) \\
        \bb_{v, x} &= -\frac{1}{2h_x} (\bb_{vx\ell} - \bb_{vxr}), \qquad
        \bb_{v, y} = -\frac{1}{2h_y}(\bb_{vy\ell} - \bb_{vyr}), \\
        \bc_v &= \frac{\nu}{h_x^2}(\bb_{vx\ell} + \bb_{vxr}) + \frac{\nu}{h_y^2} (\bb_{vy\ell} + \bb_{vyr}) 
    \end{align}
    \begin{align}
        \bb_{ux\ell} &= 
        \begin{bmatrix}
            u_{ex}(x_0, y_1) \\ \vdots \\ u_{ex}(x_0, y_{n_y})
        \end{bmatrix} \otimes 
        \begin{bmatrix}
            1 \\ 0 \\ \vdots \\ 0
        \end{bmatrix}_{n_x \times 1} \in \real^{n_x n_y}, \qquad
        \bb_{uxr} = 
        \begin{bmatrix}
            u_{ex}(x_{n_x+1}, y_1) \\ \vdots \\ u_{ex}(x_{n_x+1}, y_{n_y})
        \end{bmatrix} \otimes 
        \begin{bmatrix}
            0 \\ \vdots \\ 0 \\ 1
        \end{bmatrix}_{n_x \times 1} \in \real^{n_x n_y}, \\
        \bb_{uyb} &= 
        \begin{bmatrix}
            1 \\ 0 \\ \vdots \\ 0
        \end{bmatrix}_{n_y \times 1} 
        \otimes 
        \begin{bmatrix}
            u_{ex}(x_1, y_0) \\ \vdots \\ u_{ex}(x_{n_x}, y_0)
        \end{bmatrix} \in \real^{n_x n_y} \qquad
        \bb_{uyt} = 
        \begin{bmatrix}
            0 \\ \vdots \\ 0 \\ 1
        \end{bmatrix}_{n_y \times 1} 
        \otimes 
        \begin{bmatrix}
            u_{ex}(x_1, y_{n_y+1}) \\ \vdots \\ u_{ex}(x_{n_x}, y_{n_y+1})
        \end{bmatrix} \in \real^{n_x n_y} \\
        \bb_{vx\ell} &= 
        \begin{bmatrix}
            v_{ex}(x_0, y_1) \\ \vdots \\ v_{ex}(x_0, y_{n_y})
        \end{bmatrix} \otimes 
        \begin{bmatrix}
            1 \\ 0 \\ \vdots \\ 0
        \end{bmatrix}_{n_x \times 1} \in \real^{n_x n_y}, \qquad
        \bb_{vxr} = 
        \begin{bmatrix}
            v_{ex}(x_{n_x+1}, y_1) \\ \vdots \\ v_{ex}(x_{n_x+1}, y_{n_y})
        \end{bmatrix} \otimes 
        \begin{bmatrix}
            0 \\ \vdots \\ 0 \\ 1
        \end{bmatrix}_{n_x \times 1} \in \real^{n_x n_y}, \\
        \bb_{vyb} &= 
        \begin{bmatrix}
            1 \\ 0 \\ \vdots \\ 0
        \end{bmatrix}_{n_y \times 1} 
        \otimes 
        \begin{bmatrix}
            v_{ex}(x_1, y_0) \\ \vdots \\ v_{ex}(x_{n_x}, y_0)
        \end{bmatrix} \in \real^{n_x n_y} \qquad
        \bb_{vyt} = 
        \begin{bmatrix}
            0 \\ \vdots \\ 0 \\ 1
        \end{bmatrix}_{n_y \times 1} 
        \otimes 
        \begin{bmatrix}
            v_{ex}(x_1, y_{n_y+1}) \\ \vdots \\ v_{ex}(x_{n_x}, y_{n_y+1})
        \end{bmatrix} \in \real^{n_x n_y} .
    \end{align}
\end{subequations}

For the monolithic (single domain) FOM, we take $n_x = 480$, $n_y=24$, viscosity $\nu=0.1$, and parameters $(a, \lambda)\in \cD = [1, 10^4]\times [5, 25]$. 
The parameter $a$ corresponds to the distance of the shock from the left boundary, whereas $\lambda$ corresponds to the steepness of the shock, as illustrated in Fig. \ref{fig:burgers_parameters}. 
We use the ROMs to predict the case where $(a, \lambda) = (7692.5384, 21.9230).$
The SQP solver for the DD FOM, DD LS-ROM, and DD NM-ROM terminates when the $2$-norm of the right hand side of \eref{eq:sqp_system} is less than $10^{-4}$, or after $15$ iterations.

\begin{figure}[H]
    \centering
    \includegraphics[width=0.49\textwidth]{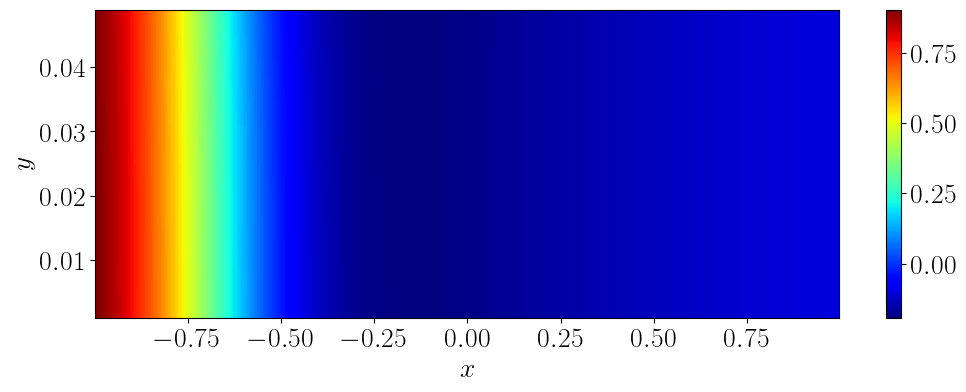}
    \includegraphics[width=0.49\textwidth]{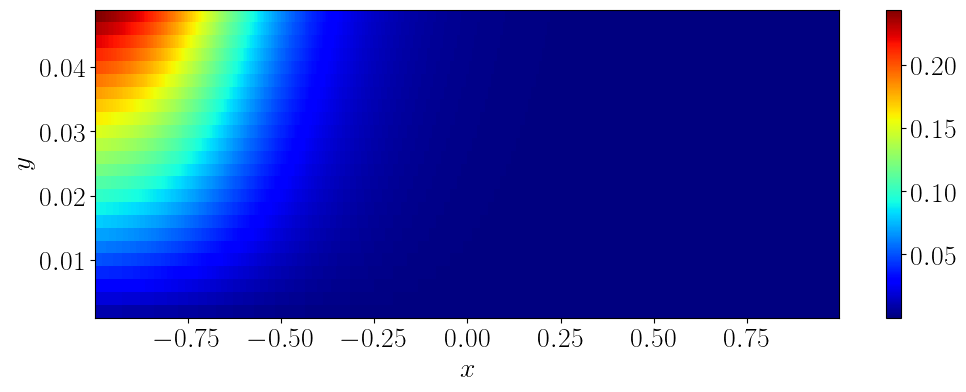}

    \includegraphics[width=0.49\textwidth]{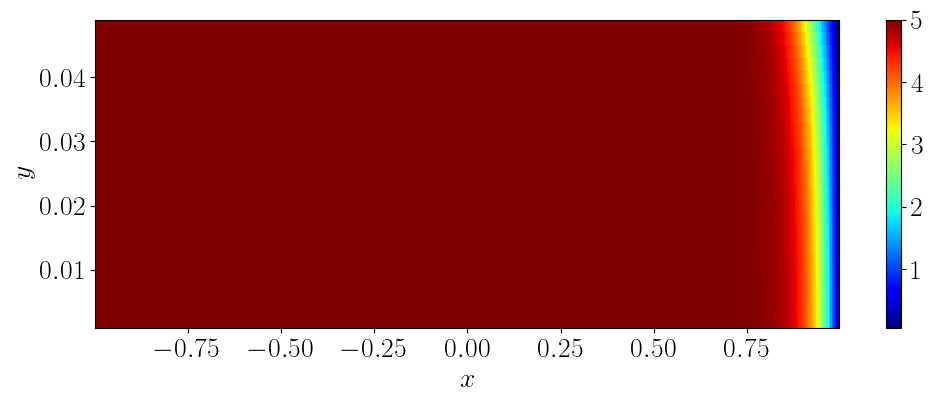}
    \includegraphics[width=0.49\textwidth]{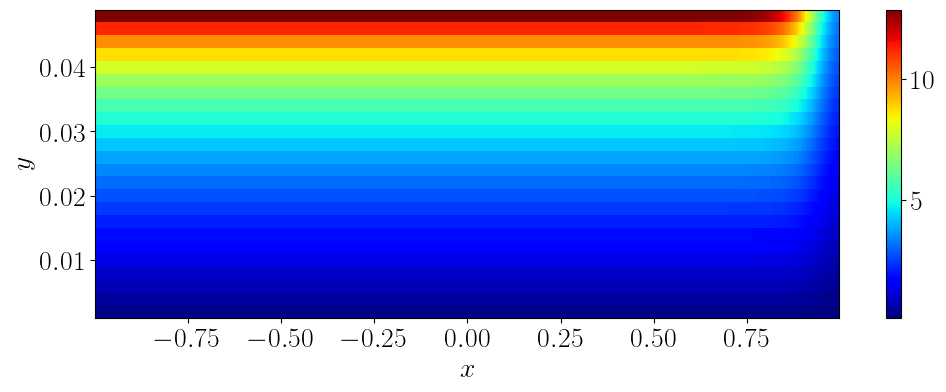}
    \caption{Top left: $u$-component with $(a, \lambda)=(10^4, 5)$; 
    Top right: $v$-component with $(a, \lambda)=(10^4, 5)$;
    Bottom left: $u$-component with $(a, \lambda)=(1, 25)$;
    Bottom right: $v$-component with $(a, \lambda)=(1, 25)$
    }
    \label{fig:burgers_parameters}
  \end{figure}

\subsection{Snapshot data collection}\label{sec:numerics_data_collection}
To compute ROMs, we first collect $6400$ snapshots for training with parameters $(a, \lambda)$ uniformly sampled in a $80 \times 80$ grid for the full-domain problem. These full-domain snapshots are then restricted to the interior, interface, and port states, which are then used for training. This is the so-called ``top-down" approach. 
The residual bases $\bPhi_i^r$ for each subdomain are computed by taking the Newton iteration history for $400$ state snapshots sampled on a $20\times 20$ $(a, \lambda)$ grid, and computing a POD basis with energy criterion $\nu = 10^{-10}.$ 
These $400$ state snapshots are then used to train RBF interpolator models (using Scipy's \texttt{RBFInterpolator} function) for each subdomain's interior and interface states, which are then used to compute an initial iterate for $(\hbx_i^\Omega, \hbx_i^\Gamma)$ for the SQP solver. 
The $(\hbx_i^\Omega, \hbx_i^\Gamma)$ initial iterates are then used to compute an initial iterate for the Lagrange multipliers $\blambda$ by applying a least-squares solver to equation (\ref{eq:optimality_residuals}b). 
The wall clock time to compute the initial iterates $(\hbx_i^\Omega, \hbx_i^\Gamma)$ is taken to be the largest wall clock time incurred among all subdomains, 
while the wall clock time to compute the initial iterate for $\blambda$ is the time required to sequentially solve the least-squares problem (\ref{eq:optimality_residuals}b).
The wall clock time to compute the initial iterate for NM-ROM using the RBF interpolator is included in the computation times and speedups reported.


\subsection{Autoencoder training}\label{sec:autoencoder_training}
For NM-ROM, we randomly split the state snapshots into $5760$ training snapshots and $640$ testing snapshots. 
For training the autoencoders, we use the MSE loss, the Adam optimizer over $2000$ epochs, and a batch size of $32$. 
We also normalize the snapshots so that all snapshot components are in $[-1,1]$, and apply a de-normalization layer to the output of the autoencoder. We apply early stopping with a stopping patience of $300$, and reduce the learning rate on plateau with an initial learning rate of $10^{-3}$ and a patience of $50$. 
The implementation was done using PyTorch, as well as the Pytorch Sparse and SparseLinear packages.

The sparsity masks used for the output layers of the decoders have a banded structure inspired by 2D finite difference stencils. 
Each row has three bands, where each band consists of contiguous nonzero entries, and where the band shifts to the right a specified amount from one row to the next. 
The number of nonzero entries per band and the number of columns the band shifts over are hyper-parameters
The separation between the bands in each row is equal to the product of the number of nonzeros per band and the column-shift per row.
These parameters, as well as the dimension of the interior and interface states, determine the width of the decoders.
Figure \ref{fig:sparsity_mask} provides a visualization for the decoder mask used. 
The transpose of these masks is used at input layer of the encoders.
\begin{figure}[H]
    \centering
    \includegraphics[width=0.75\textwidth]{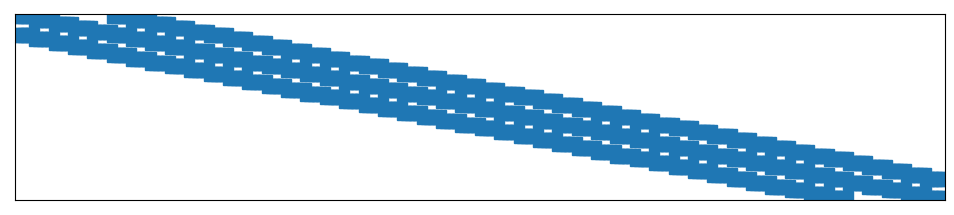}
    \caption{
        Three-banded sparsity mask for decoder
    }
    \label{fig:sparsity_mask}
\end{figure}


\subsection{Dense vs Sparse Encoder}\label{sec:dense_vs_sparse_encoder}
We first briefly compare the performance of a dense encoder with a sparse encoder.
We train two autoencoders, one with a fully dense encoder and one with a sparse encoder, on a coarse, single domain problem with $n_x=240$ and $n_y=12$.
Both decoders share the same sparse architecture.
The data collection, training, and relevant sparsity masks are identical to the procedures discussed in Section \ref{sec:autoencoder_training}. 
The discretization results in a FOM size of $5760$ and we used a ROM of size $4$.
Table \ref{tbl:dense_vs_sparse_encoder} summarizes the key performance differences between the two encoders. 
Importantly, the sparse encoder architecture achieves losses $4$ orders of magnitude smaller than the dense encoder with $2$ orders of magnitude fewer parameters.
\begin{table}[H]
    \centering
    \begin{tabular}{|c|c|c|c|c|}
        \hline
         & Train loss & Test loss & Width & \# encoder parameters\\
        \hline
        Dense encoder & $5.74 \times 10^{-2}$ & $1.56 \times 10^{-1}$ & $17280$ & $9.96 \times 10^{7}$\\
        \hline
        Sparse encoder & $7.21 \times 10^{-5}$ & $7.48 \times 10^{-5}$ & $28800$ & $2.3 \times 10^{5}$\\
        \hline
      \end{tabular}
    \caption{Comparison of dense and sparse encoder architectures.}
    \label{tbl:dense_vs_sparse_encoder}
\end{table}


\subsection{Single-domain NM-ROM vs DD NM-ROM}
Next we examine the per-subdomain reduction in the required number of autoencoder parameters for different DD configurations compared to the monolithic single-domain NM-ROM. See Table \ref{tbl:NN_params}.
In the single domain case, we solve the LSPG problem 
\begin{equation}\label{eq:monolithic_nmrom}
    \min_{\hbx} \; \norm{\bB\br\left(\bg\left(\hbx\right)\right)}_2^2
\end{equation}
using the Gauss-Newton method. 
The function $\bg:\real^{n_x}\to \real^{N_x}$ is the decoder of an autoencoder trained on snapshots of the monolithic single-domain FOM as discussed in Sections \ref{sec:nmrom}, \ref{sec:autoencoder_architecture}, and \ref{sec:autoencoder_training}, 
and $\bB\in \set{0,1}^{N_B\times N_x}$ is a collocation HR matrix, as discussed in Section \ref{sec:hyper_reduction}. 

\begin{table}[H]
    \centering
    \begin{tabular}{|c|c|c|c|c|}
        \hline
         Subdomains & Max \# subdomain params.\ & Reduction &  Total \# params. & Error\\
        \hline
          $1 \times 1$  & $2.995 \times 10^{6}$ &  $0.0$ \% & $2.995 \times 10^{6}$ & $1.08 \times 10^{-3}$\\
          \hline
          $2 \times 1$  & $1.147 \times 10^{6}$ & $61.7$ \% & $2.307 \times 10^{6}$ & $1.27 \times 10^{-3}$\\
          \hline
          $2 \times 2$  & $5.257 \times 10^{5}$ & $82.4$ \% & $2.384 \times 10^{6}$ & $2.42 \times 10^{-3}$\\
          \hline
          $4 \times 2$  & $2.617 \times 10^{5}$ & $91.3$ \% & $2.391 \times 10^{6}$ & $4.26 \times 10^{-3}$\\
          \hline
          $8 \times 2$  & $1.297 \times 10^{5}$ & $95.7$ \% & $2.406 \times 10^{6}$ & $4.58 \times 10^{-2}$\\
        \hline
      \end{tabular}
    \caption{Max number of NN parameters per subdomain, the per-subdomain reduction in number of NN parameters, the total number of parameters, and the corresponding error for different subdomain configurations. For the single-domain case, an NM-ROM of dimension $n=9$ is used. For the DD cases, $(n_i^\Omega, n_i^\Gamma)=(6, 3)$, resulting in $9$ DoF per subdomain. HR was not used to evaluate the NM-ROMs in these examples.}
    \label{tbl:NN_params}
\end{table}
We use the notation $2\times 1$ subdomains to indicate $2$ subdomains in the $x$-direction and $1$ subdomain in the $y$-direction.
As expected, from Table \ref{tbl:NN_params}, we see that the maximum number of NN parameters per subdomain decreases significantly as more subdomains are used. 
Furthermore, the total number of NN parameters in the DD cases also decreases relative to the single-domain case. 
We also note that the error increases as more subdomains are used.
We kept $(n_i^\Omega, n_i^\Gamma)=(6, 3)$ constant for each subdomain configuration to isolate the effect of DD on the number of NN parameters, but this may cause overfitting in the $16$ subdomain case. 
More careful hyper-parameter tuning is necessary to mitigate increases in error as the number of subdomains is increased. 


\subsection{LS-ROM vs NM-ROM comparison}\label{sec:numerics_comparison}
Next we compare the DD LS-ROM and DD NM-ROM. We first focus on a DD configuration with $2$ uniformly sized subdomains in the $x$-direction and $2$ uniformly sized subdomains in the $y$-direction ($4$ subdomains total) using the WFPC formulation \eref{eq:dd_lspg_NLP}.
The interior and interface states for the FOM were of dimension $N_i^\Omega=5238$ and $N_i^\Gamma = 1006$, respectively, resulting in $25056$ degrees of freedom (DoF) aggregated across all subdomains.
For both the LS-ROM and NM-ROM, we use reduced state dimensions of $n_i^\Omega = 8$ for the interior states $\hbx_i^\Omega$ and $n_i^\Gamma = 4$ for the interface states $\hbx_i^\Gamma$ for each subdomain, 
resulting in $48$ DoF aggregated across all subdomains.
In the HR case, $N_i^B=100$ HR nodes are used for each subdomain, resulting in $400$ total HR nodes aggregated all subdomains. 

Each interior-state autoencoder has
input/output dimension $N_i^\Omega=5238$,
width $w_i^\Omega=26290$, 
latent dimension $n_i^\Omega=8$, 
and Swish activation 
$\bsigma_i^\Omega(z) = z/(1+e^{-z}).$
Each interface-state autoencoder has
input/output dimension $N_i^\Gamma=1006$,
width $w_i^\Gamma=5030$,
latent dimension $n_i^\Gamma=4$, and Swish activation.
The number of nonzeros per row and column-shift were both set to $5$ for both the interior and interface state decoders.
The number of nonzeros for the interior-states masks is $78820$, resulting in $99.94\%$ sparsity, while the number of nonzeros for the interface-states masks is $15040$, resulting in $99.70\%$ sparsity.

Figure \ref{fig:uv_states_no_hr} shows the FOM, LS-ROM, and NM-ROM solutions without HR, and Figure \ref{fig:uv_states_with_hr} shows the solutions with collocation HR using $48$ DoF in both cases.
In both the HR and non-HR cases with the same DoF, NM-ROM achieves an order of magnitude lower relative error than LS-ROM, as evidenced in Figures \ref{fig:uv_rel_error_no_hr} and \ref{fig:uv_rel_error_with_hr} and Table \ref{tbl:error_speedup}.
Without HR, NM-ROM achieves a relative error of $1.28\times 10^{-3}$ while LS-ROM achieves a relative error $1.98 \times 10^{-2}$ using the same number of DoF. 
LS-ROM also achieves a speedup of $30.0$, whereas NM-ROM achieves a $21.7$ times speedup.
With HR, NM-ROM achieves a relative error of $1.64\times 10^{-3}$ while LS-ROM achieves a relative error $1.44 \times 10^{-2}$ using the same number of DoF. 
In the HR case, LS-ROM achieves a speedup of $347.6$, whereas NM-ROM achieves a $43.9$ speedup.

\begin{figure}[H]
    \centering
    \includegraphics[width=0.49\textwidth]{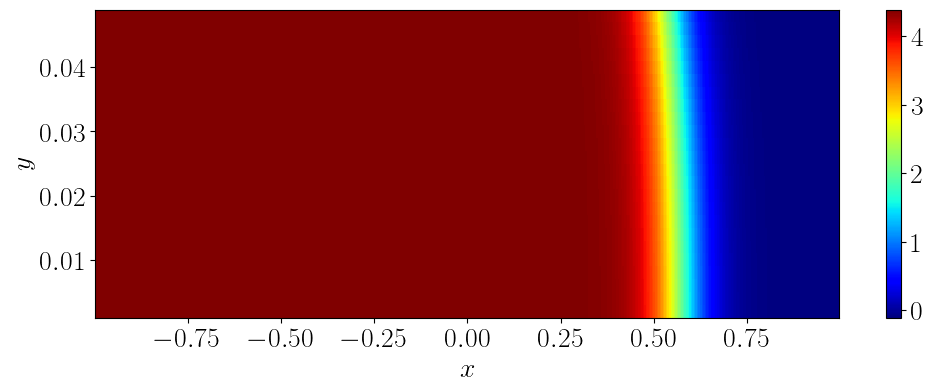}
    \includegraphics[width=0.49\textwidth]{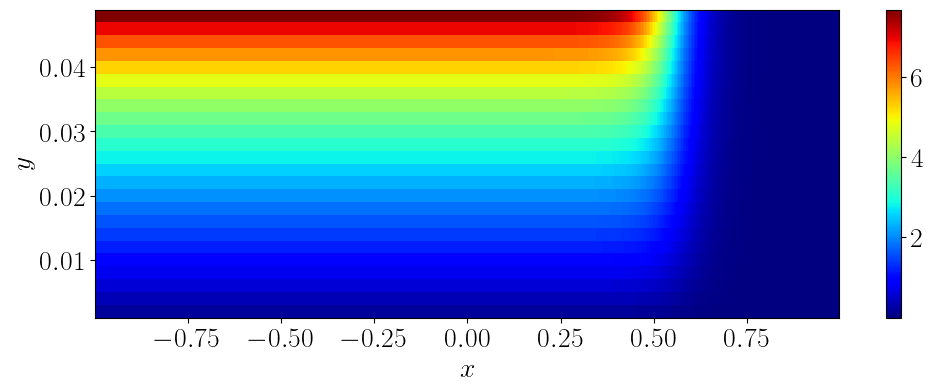}

    \includegraphics[width=0.49\textwidth]{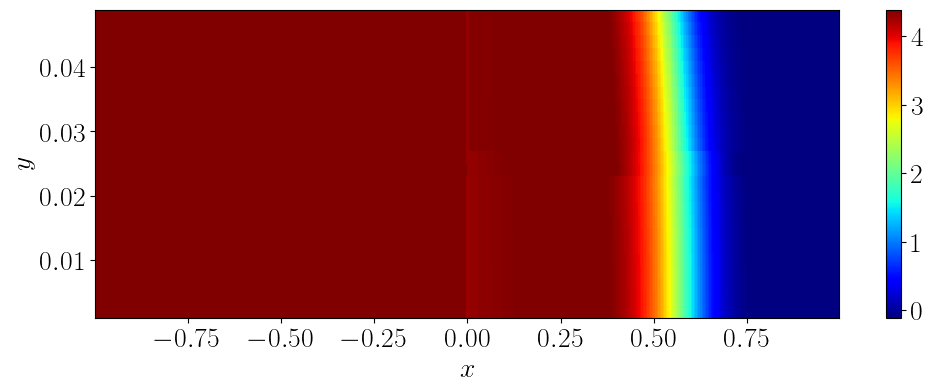}
    \includegraphics[width=0.49\textwidth]{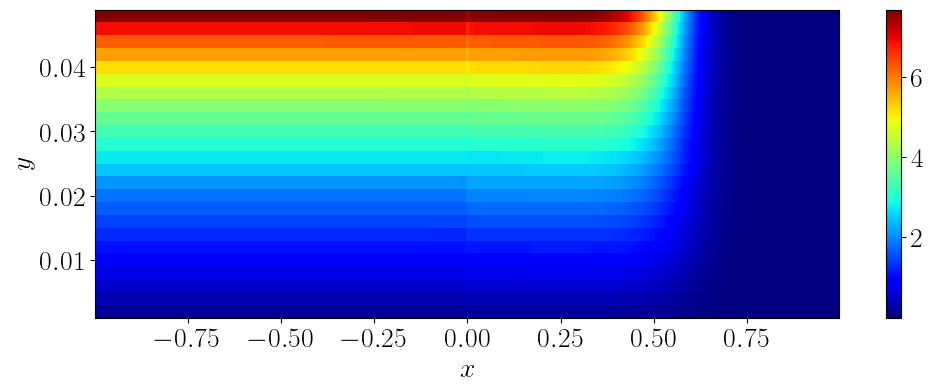}

    \includegraphics[width=0.49\textwidth]{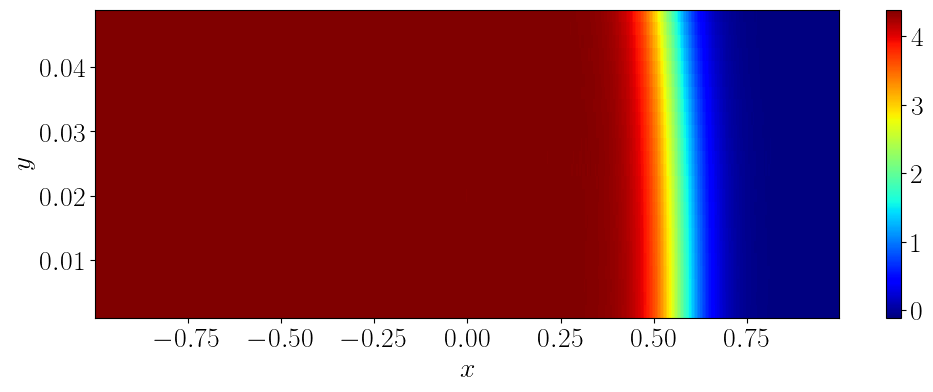}
    \includegraphics[width=0.49\textwidth]{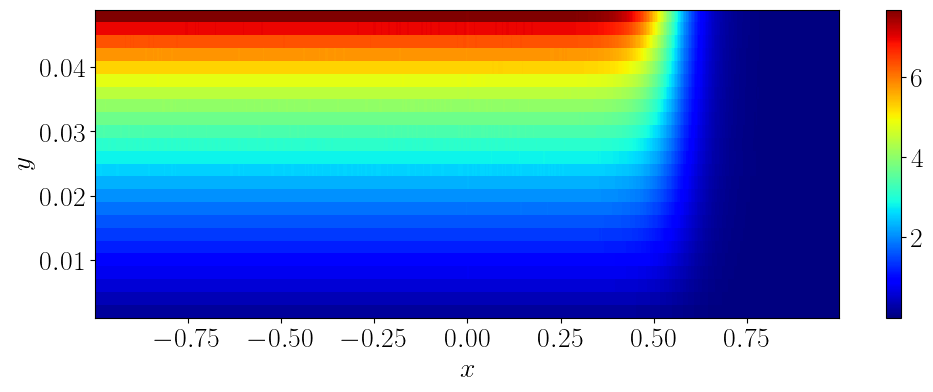}
    \caption{
        Top left: $u$-component of FOM;
        Top right: $v$-component of FOM;
        Middle left: $u$-component of LS-ROM without HR, WFPC, $48$ DoF;
        Middle right: $v$-component of LS-ROM without HR, WFPC, $48$ DoF;
        Bottom left: $u$-component of NM-ROM without HR, WFPC, $48$ DoF;
        Bottom right: $v$-component of NM-ROM without HR, WFPC, $48$ DoF.
    }
    \label{fig:uv_states_no_hr}
\end{figure}

\begin{figure}[H]
    \centering
    \includegraphics[width=0.49\textwidth]{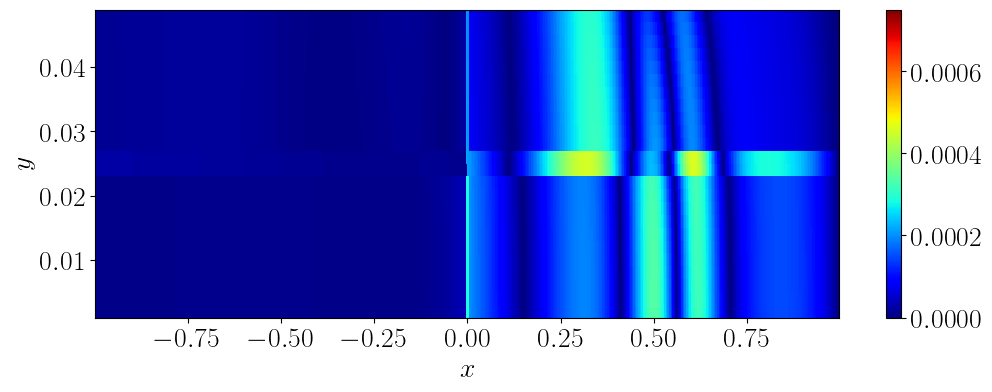}
    \includegraphics[width=0.49\textwidth]{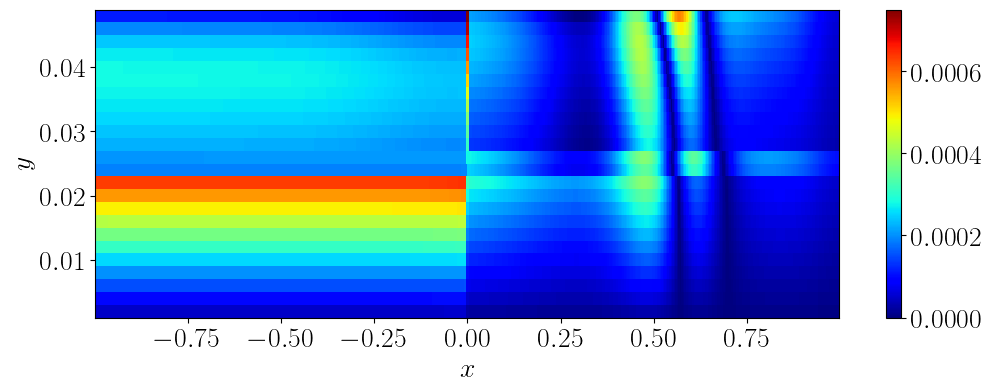}

    \includegraphics[width=0.49\textwidth]{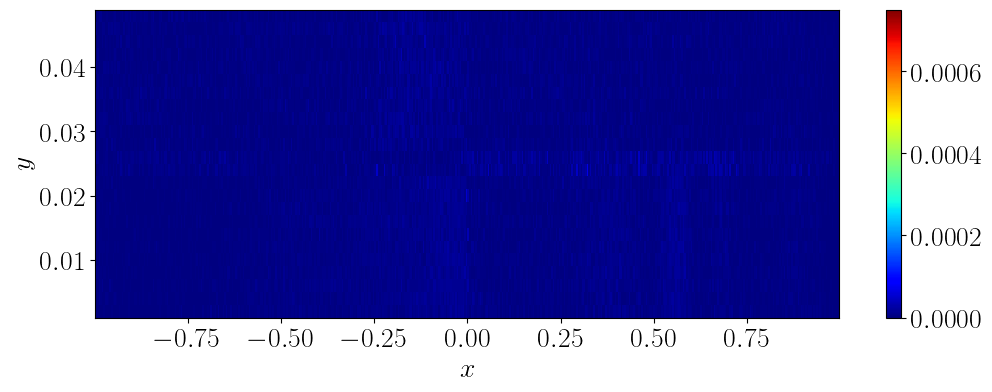}
    \includegraphics[width=0.49\textwidth]{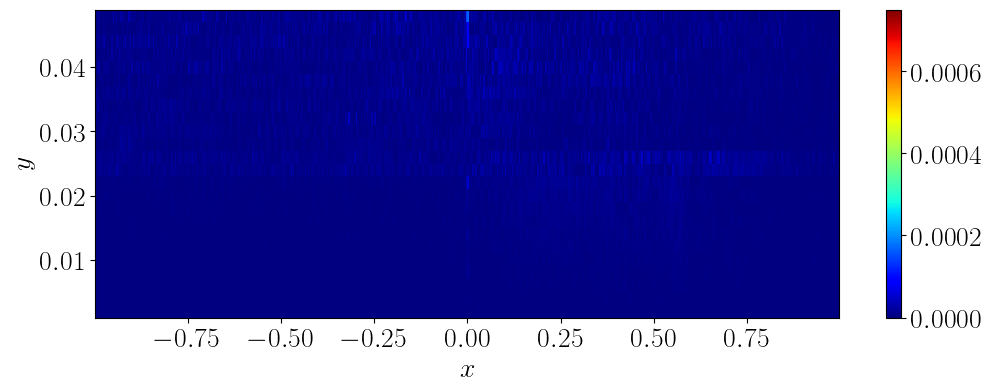}
    \caption{
        Top left: $u$ relative error for LS-ROM without HR, WFPC, $48$ DoF;
        Top right: $v$ relative error for LS-ROM without HR, WFPC, $48$ DoF;
        Bottom left: $u$ relative error for RM-ROM without HR, WFPC, $48$ DoF;
        Bottom right: $v$ relative error for NM-ROM without HR, WFPC, $48$ DoF.
    }
    \label{fig:uv_rel_error_no_hr}
\end{figure}

\begin{figure}[H]
    \centering
    \includegraphics[width=0.49\textwidth]{u_fom.png}
    \includegraphics[width=0.49\textwidth]{v_fom.png}

    \includegraphics[width=0.49\textwidth]{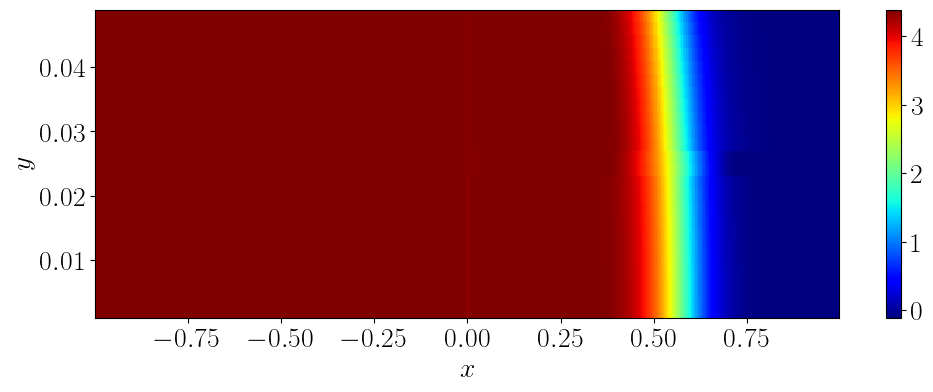}
    \includegraphics[width=0.49\textwidth]{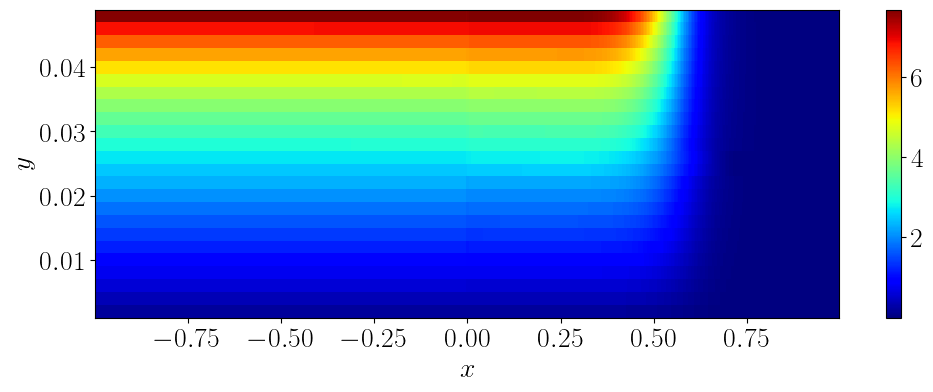}

    \includegraphics[width=0.49\textwidth]{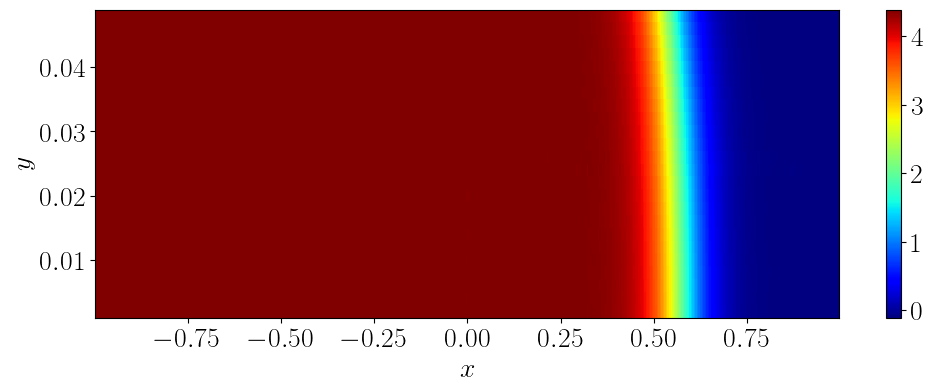}
    \includegraphics[width=0.49\textwidth]{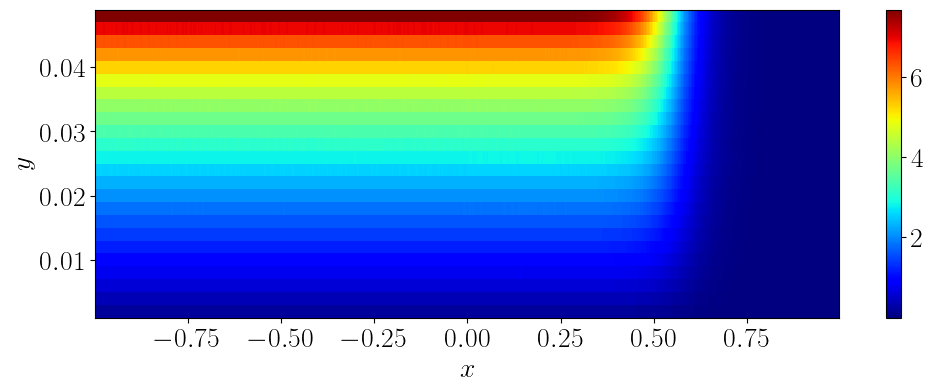}
    \caption{
        Top left: $u$-component of FOM;
        Top right: $v$-component of FOM;
        Middle left: $u$-component of LS-ROM with collocation HR, WFPC, $48$ DoF;
        Middle right: $v$-component of LS-ROM with collocation HR, WFPC, $48$ DoF;
        Bottom left: $u$-component of NM-ROM with collocation HR, WFPC, $48$ DoF;
        Bottom right: $v$-component of NM-ROM with collocation HR, WFPC, $48$ DoF.
    }
    \label{fig:uv_states_with_hr}
\end{figure}

\begin{figure}[H]
    \centering
    \includegraphics[width=0.49\textwidth]{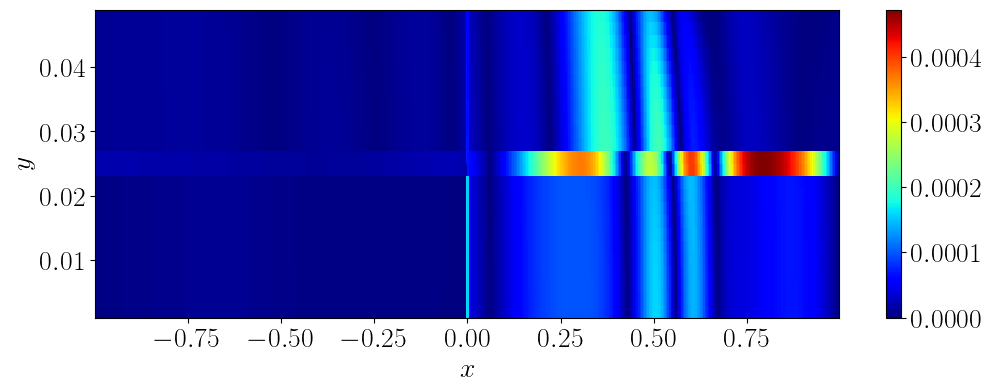}
    \includegraphics[width=0.49\textwidth]{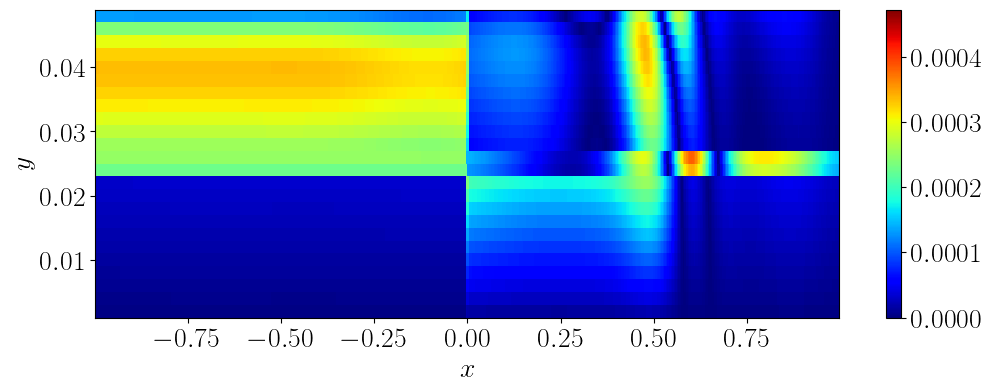}

    \includegraphics[width=0.49\textwidth]{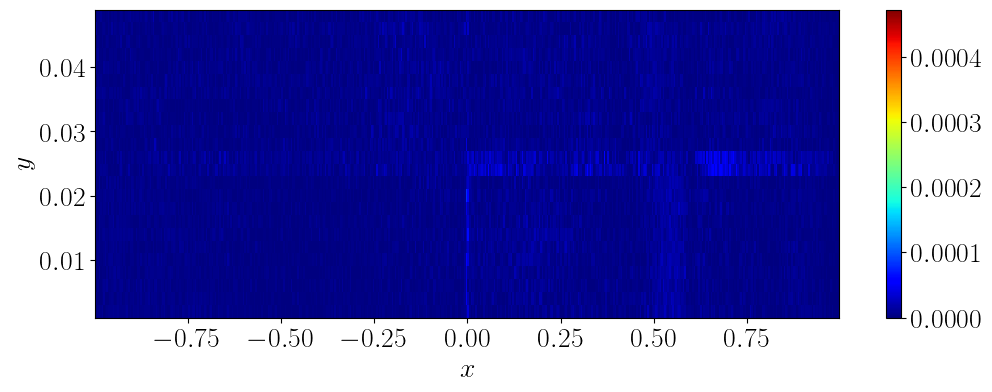}
    \includegraphics[width=0.49\textwidth]{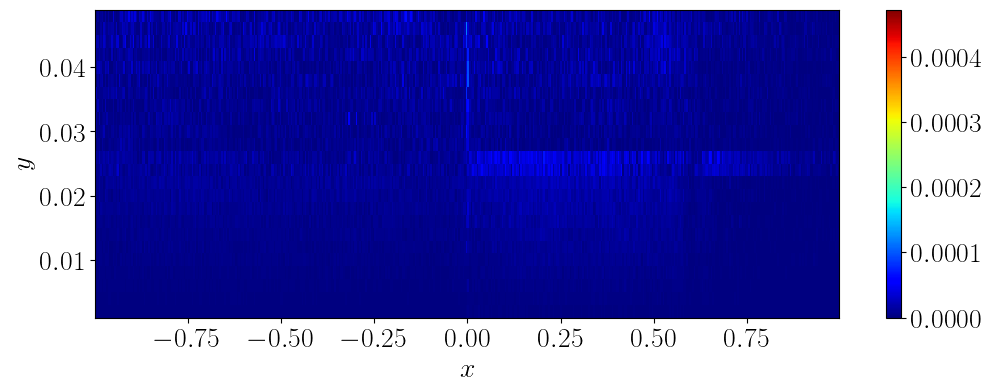}
    \caption{
        Top left: $u$ relative error for LS-ROM with collocation HR, WFPC, $48$ DoF;
        Top right: $v$ relative error for LS-ROM with collocation HR, WFPC, $48$ DoF;
        Bottom left: $u$ relative error for NM-ROM with collocation HR, WFPC, $48$ DoF;
        Bottom right: $v$ relative error for NM-ROM with collocation HR, WFPC, $48$ DoF.
    }
    \label{fig:uv_rel_error_with_hr}
\end{figure}

Now we compare the performance of LS-ROM and NM-ROM for both the WFPC and SRPC formulations while varying the interior and interface ROM state dimensions $n_i^\Omega$ and $n_i^\Gamma$. 
For WFPC, the decoders $\bg_i^\Omega$ and $\bg_i^\Gamma$ use Swish activations, and their sparsity masks have 5 nonzeros per band and a column-shift of $5$ for each test.  
For SRPC, the interior-state decoders $\bg_i^\Omega$ are reused from the WFPC formulation, 
and the port-state decoders $\bg_j^p$ were chosen to have Sigmoid activation and sparsity masks with 3 nonzero entries per band with column-shift of 3. 
We also define $\widetilde{n}_j^p$, which determines the port latent dimensions $n_j^p$ via the relation
$$
n_j^p = \max\set{\min\set{N_j^p-1,\; \widetilde{n}_j^p},1}, \qquad \forall \; j = 1, \dots, n_p.
$$
This rule was chosen to ensure that $1 \leq n_j^p < N_j^p$ because some DD configurations have ports with a very small number of nodes (e.g., $N_j^p=2$).
Recall that for SRPC, the ROM interface-state dimension is $n_i^\Gamma = \sum_{j\in Q(i)} n_j^p.$
Table \ref{tbl:error_speedup} shows the relative error and speedup for LS-ROM and NM-ROM for both WFPC and SRPC on the $2\times 2$ subdomain problem with and without HR while varying $n_i^\Omega$, $\widetilde{n}_j^p$, and $n_i^\Gamma$.


From Table \ref{tbl:error_speedup}, we see that NM-ROM consistently achieves an order of magnitude lower error than LS-ROM both with and without HR when comparing ROMs of the same dimensions and constraint formulations.
In the non-HR case with WFPC, LS-ROM only achieves an order $10^{-3}$ error for a ROM with $96$ total DoF (rel. error = $2.66 \times 10^{-3}$), while NM-ROM can achieve a similar error with only $36$ DoF (rel. error = $2.42 \times 10^{-3}$) and a higher speedup (speedup = $26.2$) compared to LS-ROM with similar accuracy (speedup = $18.3$). 
For SRPC, LS-ROM was only able to achieve order $10^{-2}$ accuracy for all cases tested.
We also see that LS-ROM achieves a much higher speedup in the HR cases while retaining similar errors from the non-HR cases. 
NM-ROM also retains high accuracy after HR, and gains an extra $15$-$20\times$ speedup after applying HR. 

\begin{table}[H]
    \centering
    \begin{tabular}{|c|c|c|c|c|c|c|c|c|c|}
        \hline
         & Constraints & $n_i^\Omega$ & $\widetilde{n}_j^p$ & $n_i^\Gamma$ & Total DoF & Error & Speedup & Error (HR) & Speedup (HR)\\
        \hline
        \multirow{9}{*}{LS-ROM} 
        & \multirow{5}{*}{WFPC} 
        &   $4$ & - & $2$ & $24$ & $4.12 \times 10^{-2} $ & $32.1$ & $3.45 \times 10^{-2}$ & $352.7$\\
        &&  $6$ & - & $3$ & $36$ & $2.06 \times 10^{-2} $ & $48.7$ & $1.78 \times 10^{-2} $ & $340.0$\\
        &&  $8$ & - & $4$ & $48$ & $1.98 \times 10^{-2} $ & $30.0$ & $1.44 \times 10^{-2} $ & $347.6$\\
        && $10$ & - & $5$ & $60$ & $1.50 \times 10^{-2} $ & $16.3$ & $1.16 \times 10^{-2} $ & $329.6$\\
        && $16$ & - & $8$ & $96$ & $2.66 \times 10^{-3} $ & $18.3$ & $3.23 \times 10^{-3} $ & $280.4$\\
        \cline{2-10}
        & \multirow{4}{*}{SRPC}
        &   $6$ & $1$ & $5$ & $44$ & $5.17 \times 10^{-2} $ & $24.5$ & $5.12 \times 10^{-2} $ & $315.6$\\
        &&  $8$ & $2$ & $7$ & $60$ & $3.75 \times 10^{-2} $ & $22.0$ & $4.22 \times 10^{-2} $ & $313.9$\\
        && $10$ & $3$ & $9$ & $76$ & $1.87 \times 10^{-2} $ & $19.4$ & $2.94 \times 10^{-2} $ & $262.6$\\
        && $16$ & $4$ & $11$ & $108$ & $2.00 \times 10^{-2} $ & $17.8$ & $3.37 \times 10^{-2} $ & $253.8$\\
        \hline
        \multirow{9}{*}{NM-ROM} 
        & \multirow{5}{*}{WFPC} 
        &   $4$ & - & $2$ & $24$ & $6.94 \times 10^{-3} $ & $22.7$ & $7.04 \times 10^{-3} $ & $37.4$\\
        &&  $6$ & - & $3$ & $36$ & $2.42 \times 10^{-3} $ & $26.2$ & $2.60 \times 10^{-3} $ & $44.7$\\
        &&  $8$ & - & $4$ & $48$ & $1.28 \times 10^{-3} $ & $21.7$ & $1.64 \times 10^{-3} $ & $43.9$\\
        && $10$ & - & $5$ & $60$ & $1.09 \times 10^{-3} $ & $15.0$ & $1.19 \times 10^{-3} $ & $43.6$\\
        && $16$ & - & $8$ & $96$ & $7.87 \times 10^{-4} $ & $13.9$ & $9.80 \times 10^{-4} $ & $37.5$ \\
        \cline{2-10}
        & \multirow{4}{*}{SRPC}
        &   $6$ & $1$ & $5$ & $44$ & $2.75 \times 10^{-2} $ & $27.6$ & $3.17 \times 10^{-2} $ & $41.1$\\
        &&  $8$ & $2$ & $7$ & $60$ & $1.19 \times 10^{-3} $ & $28.8$ & $1.70 \times 10^{-3} $ & $42.6$\\
        && $10$ & $3$ & $9$ & $76$ & $1.46 \times 10^{-3} $ & $17.0$ & $2.54 \times 10^{-3} $ & $43.1$\\
        && $16$ & $4$ & $11$ & $108$ & $1.11 \times 10^{-3}$ & $16.8$ & $2.45 \times 10^{-3} $ & $47.1$\\
        \hline
      \end{tabular}
    \caption{Relative error and speedup for LS-ROM and NM-ROM with and without HR applied to the WFPC and SRPC formulations. We use $N_i^B=100$ HR nodes per subdomain in the HR case.}
    \label{tbl:error_speedup}
\end{table}


Next we examine the effect of subdomain configuration on the accuracy of LS-ROM and NM-ROM. 
Again let $n_\Omega^x$ and $n_\Omega^y$ denote the number of subdomains in the $x$- and $y$-directions, respectively. 
In each case, the subdomains are of uniform size.
For the WFPC cases, we used $(n_i^\Omega, n_i^\Gamma)=(8,4)$ for each subdomain, 
and for the SRPC cases, we used $(n_i^\Omega, \widetilde{n}_j^p)=(8, 2)$ for each subdomain and port, respectively.
For the HR cases, we used $N_i^B=100$ HR nodes. 
The results for the non-HR and HR cases are reported in Table \ref{tbl:error_speedup_dd_config}.

Table \ref{tbl:error_speedup_dd_config} shows that LS-ROM is more sensitive to subdomain configuration than NM-ROM in the WFPC cases.
Indeed, when using $2$ subdomains in the $y$-direction, the relative error for LS-ROM increases to the order of $10^{-2}$, compared to errors on the order of $10^{-3}$ when only $1$ subdomain is used in the $y$-direction. 
In contrast, relative error for NM-ROM with WFPC is more stable with respect to subdomain configuration. 
The relative errors are on the order of $10^{-3}$ for all configuration considered.
For SRPC, LS-ROM only achieves order $10^{-3}$ relative error for the $(n_\Omega^x, n_\Omega^y) = (2, 1)$ configuration, while the remaining configurations have order $10^{-2}$ 
relative error.
For NM-ROM with SRPC, all configurations have order $10^{-3}$ relative error 
except for $(n_\Omega^x, n_\Omega^y) = (4, 2)$, which attains order $10^{-2}$ relative error. 
For both WFPC and SRPC, the NM-ROM error increases slightly as more subdomains are used, but we expect this error can be decreased by adjusting hyper-parameters during ROM training. 
Hyper-parameter tuning was only done for the $2\times2$ subdomain configuration.

\begin{table}[H]
    \centering
    \begin{tabular}{|c|c|c|c|c|c|c|c|c|c|c|}
        \hline
         & Constraints & $n_\Omega^x$ & $n_\Omega^y$ & \# subdomains  & Error & Speedup & Error (HR) & Speedup (HR)\\
        \hline
        \multirow{8}{*}{LS-ROM} 
        & \multirow{4}{*}{WFPC} 
        &  $2$ & $1$ & $2$ & $6.36 \times 10^{-3}$ & $25.1$ & $6.64 \times 10^{-3}$ & $285.7$\\
        && $2$ & $2$ & $4$ & $1.98 \times 10^{-2}$ & $30.0$ & $1.44 \times 10^{-2}$ & $347.6$\\
        && $4$ & $1$ & $4$ & $7.34 \times 10^{-3}$ & $37.1$ & $7.47 \times 10^{-3}$ & $373.1$\\
        && $4$ & $2$ & $8$ & $2.29 \times 10^{-2}$ & $35.8$ & $4.21 \times 10^{-2}$ & $259.2$\\
        \cline{2-9}
        & \multirow{4}{*}{SRPC} 
        &  $2$ & $1$ & $2$ & $6.85 \times 10^{-3}$ & $30.5$ & $9.49 \times 10^{-3}$ & $293.0$\\
        && $2$ & $2$ & $4$ & $3.75 \times 10^{-2}$ & $22.0$ & $4.22 \times 10^{-2}$ & $313.9$\\
        && $4$ & $1$ & $4$ & $1.04 \times 10^{-2}$ & $24.6$ & $5.94 \times 10^{-2}$ & $287.5$\\
        && $4$ & $2$ & $8$ & $4.96 \times 10^{-2}$ & $12.2$ & $5.19 \times 10^{-2}$ & $181.6$\\
        \hline
        \multirow{8}{*}{NM-ROM} 
        & \multirow{4}{*}{WFPC} 
        &  $2$ & $1$ & $2$ & $1.34 \times 10^{-3}$ & $16.8$ & $1.36 \times 10^{-3}$ & $30.5$ \\
        && $2$ & $2$ & $4$ & $1.28 \times 10^{-3}$ & $21.7$ & $1.64 \times 10^{-3}$ & $43.9$ \\
        && $4$ & $1$ & $4$ & $3.14 \times 10^{-3}$ & $27.8$ & $4.98 \times 10^{-3}$ & $38.6$\\
        && $4$ & $2$ & $8$ & $4.82 \times 10^{-3}$ & $26.3$ & $5.98 \times 10^{-3}$ & $40.4$ \\
        \cline{2-9}
        & \multirow{4}{*}{SRPC} 
        &  $2$ & $1$ & $2$ & $1.00 \times 10^{-3}$ & $17.0$ & $1.37 \times 10^{-3}$ & $35.9$\\
        && $2$ & $2$ & $4$ & $1.19 \times 10^{-3}$ & $28.8$ & $1.70 \times 10^{-3}$ & $42.6$\\
        && $4$ & $1$ & $4$ & $1.68 \times 10^{-3}$ & $27.4$ & $2.12 \times 10^{-3}$ & $39.1$\\
        && $4$ & $2$ & $8$ & $1.67 \times 10^{-2}$ & $24.2$ & $2.39 \times 10^{-2}$ & $32.5$\\
        \hline
      \end{tabular}
    \caption{Relative error and speedup for LS-ROM and NM-ROM with and without HR and different subdomain configurations for the WFPC and SRPC formulations. We use $n_i^\Omega=8$ for all cases, $n_i^\Gamma=4$ for the WFPC cases, $\widetilde{n}_j^p=2$ for the SRPC cases, and $N_i^B=100$ for the HR cases.}
    \label{tbl:error_speedup_dd_config}
\end{table}

Tables \ref{tbl:error_speedup} and \ref{tbl:error_speedup_dd_config} show that SRPC has slightly worse performance than WFPC. 
In particular, Table \ref{tbl:error_speedup} shows that the relative errors for both LS-ROM and NM-ROM with SRPC do not decrease monotonically as $n_i^\Omega$ and $\widetilde{n}_j^p$ increase. 
In contrast, the relative errors do decrease monotonically as $n_i^\Omega$ and $n_i^\Gamma$ increase for both LS-ROM and NM-ROM with WFPC. 
Furthermore, Table \ref{tbl:error_speedup_dd_config} shows that LS-ROM with SRPC consistently has larger errors than with WFPC for each subdomain configuration. 
For NM-ROM, the relative errors and speedups are similar between WFPC and SRPC for each subdomain configuration except for $(n_\Omega^x, n_\Omega^y) = (4, 2)$, which has an order of magnitude higher error than the other configurations. 
Since SRPC performs as well or worse compared to WFPC for the cases tested, we only consider WFPC in the remainder of this section. 


\subsection{Pareto fronts}\label{sec:numerics_pareto}
Next we compute Pareto fronts to compare LS-ROM and NM-ROM with WFPC while varying different parameters. 
The relative error reported is the same as defined in equation \eref{eq:relative_error}.
Figure \ref{fig:pareto_romsize} shows the Pareto fronts for varying $(n_i^\Omega, n_i^\Gamma)$ for both the non-HR and HR cases. In the HR case, we use $N_i^B=100$ for each subdomain. 
We see that LS-ROM wins in terms of relative wall time, while
NM-ROM attains an order of magnitude lower error in comparison to LS-ROM in each case. 
\begin{figure}[H]
    \centering
    \includegraphics[width=0.49\textwidth]{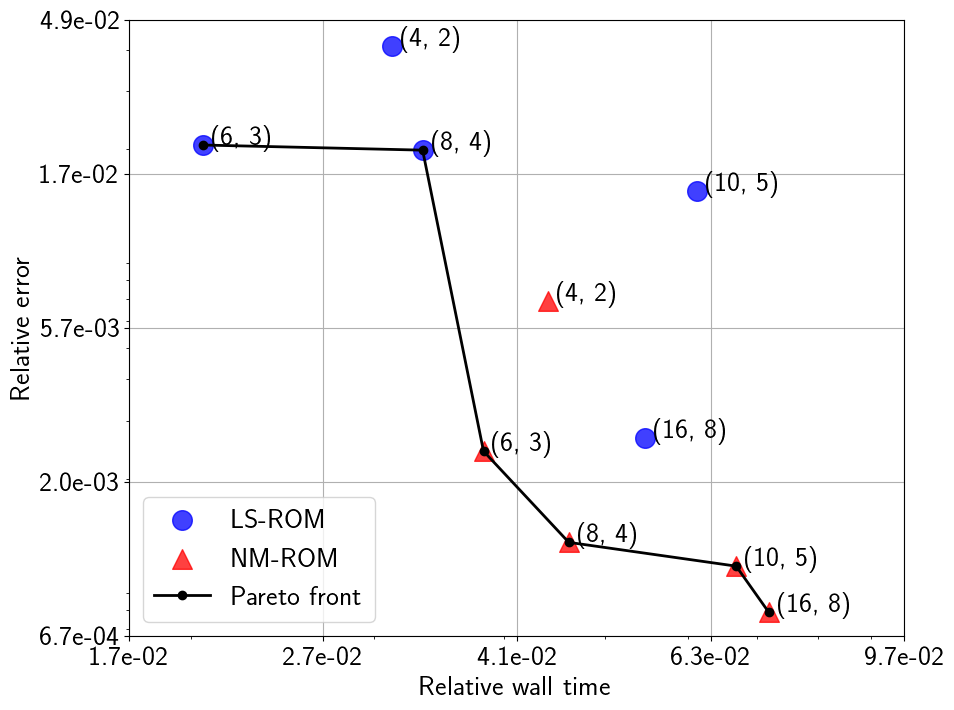}
    \includegraphics[width=0.49\textwidth]{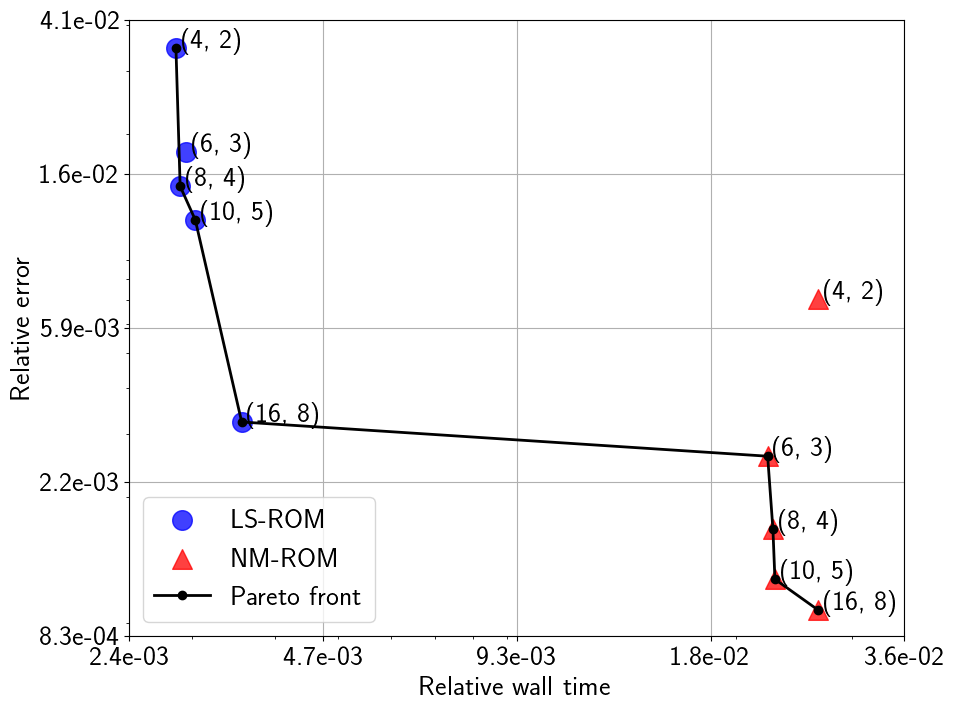}

    \caption{
        Left: Pareto front for LS-ROM and NM-ROM without HR with varying $(n_i^\Omega, n_i^\Gamma)$ for WFPC formulation;
        Right: Pareto front for LS-ROM and NM-ROM with varying $(n_i^\Omega, n_i^\Gamma)$ and $N_i^B=100$ HR nodes per subdomain for WFPC formulation
    }
    \label{fig:pareto_romsize}
\end{figure}

Figure \ref{fig:pareto_hr} shows the Pareto front for varying number of HR nodes per subdomain, $N_i^B$. In this case, $(n_i^\Omega, n_i^\Gamma)=(8,4)$ for each experiment. As in the case for varying $(n_i^\Omega, n_i^\Gamma)$, NM-ROM attains an order of magnitude lower relative error. Moreover, both the relative error and relative wall time for NM-ROM remains small for each value of $N_i^B$, whereas the relative error and relative wall time for LS-ROM has more variability with respect to number of HR nodes. 
\begin{figure}[H]
    \centering
    \includegraphics[width=0.49\textwidth]{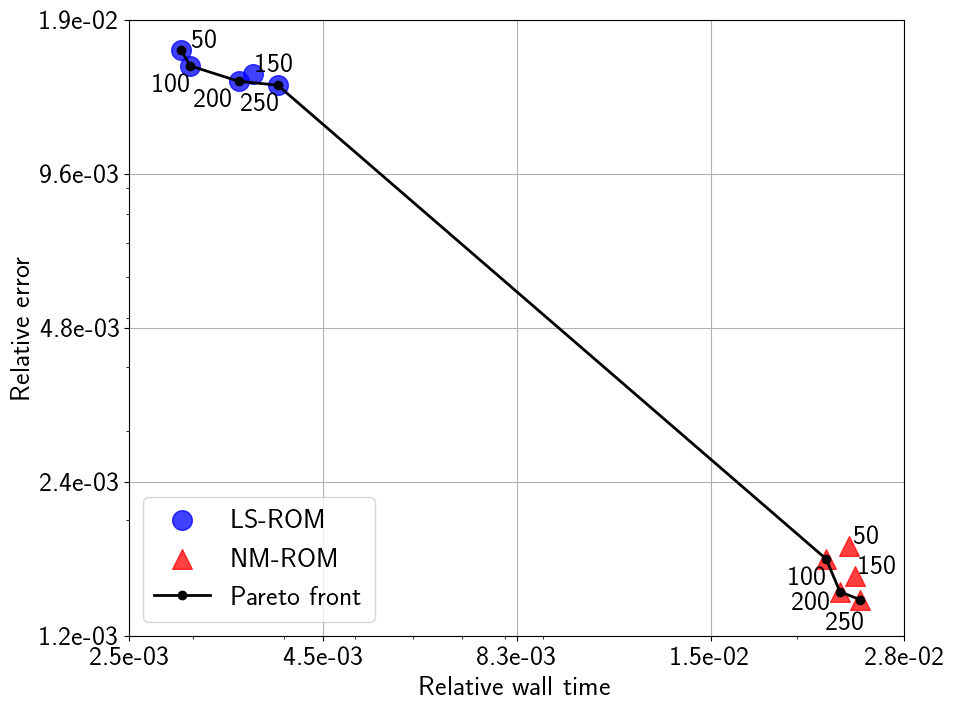}
    \caption{
        Pareto front for LS-ROM and NM-ROM with $(n_i^\Omega, n_i^\Gamma)=(8,4)$ and varying number of HR nodes per subdomain $N_i^B$ for WFPC formulation
    }
    \label{fig:pareto_hr}
\end{figure}

Figure \ref{fig:pareto_nsnaps} shows the Pareto fronts for varying number of training snapshots in the non-HR and HR cases. We used $(n_i^\Omega, n_i^\Gamma)=(8,4)$ for each experiment, and used $N_i^B=100$ HR nodes in the HR case. For LS-ROM, the whole training set was used to compute the POD bases, whereas for NM-ROM, the displayed number of snapshots underwent a random 90-10 split for training and validation, respectively. 
In the case of LS-ROM, the relative error remained constant for the number of training snapshots used, whereas the relative error for NM-ROM decreased as more training snapshots were used. 
\begin{figure}[H]
    \centering
    \includegraphics[width=0.49\textwidth]{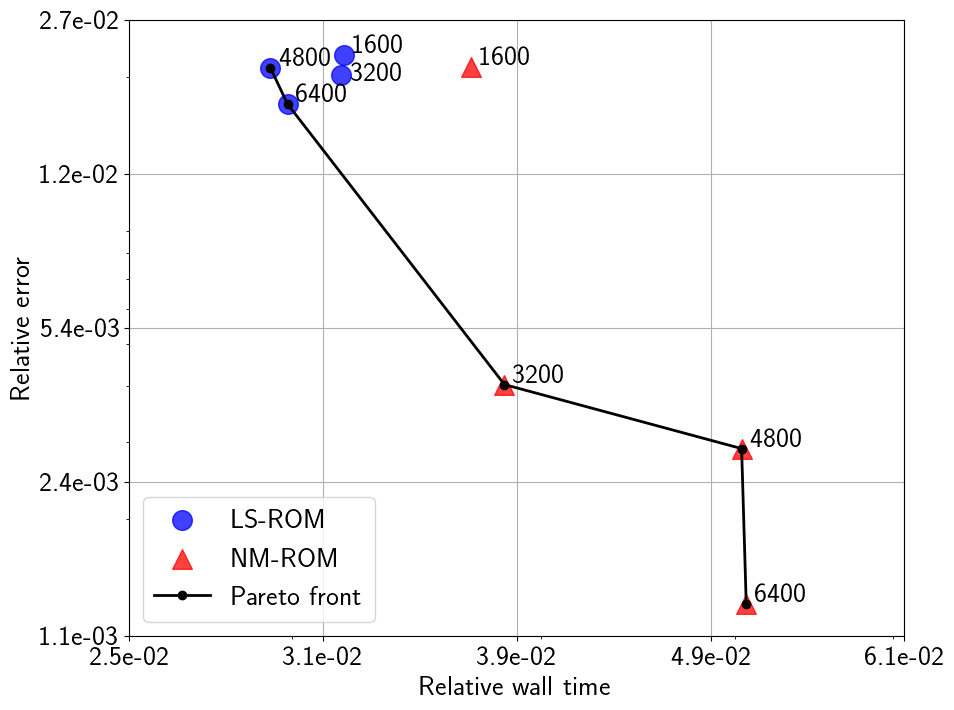}
    \includegraphics[width=0.49\textwidth]{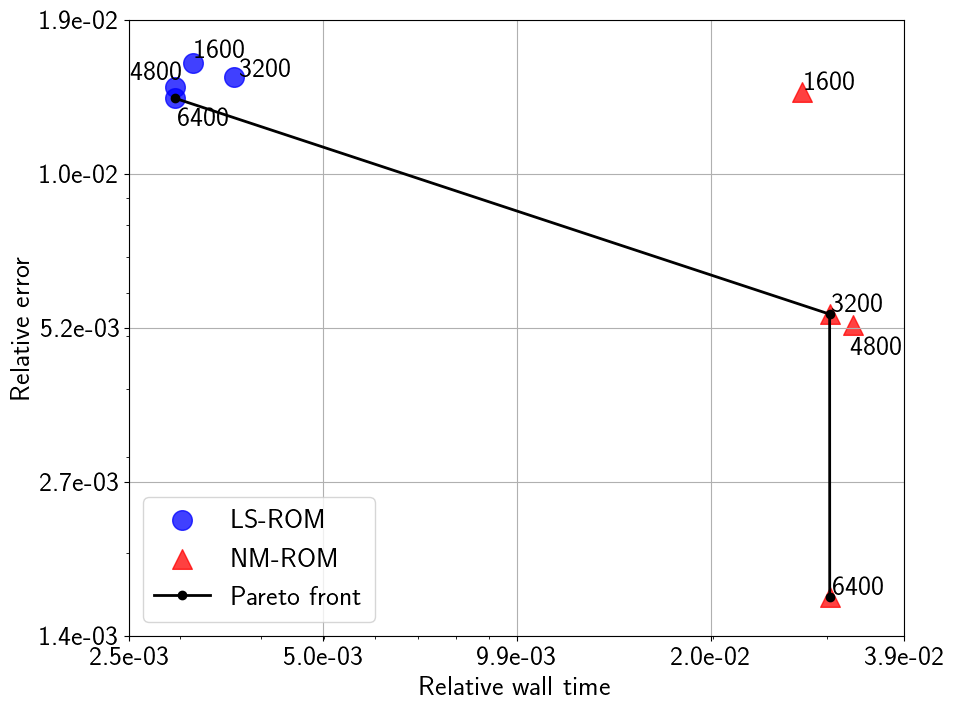}
    \caption{
        Left: Pareto front for LS-ROM and NM-ROM with $(n_i^\Omega, n_i^\Gamma)=(8,4)$ and varying number of training snapshots for WFPC formulation;
        Right: Pareto front for LS-ROM and NM-ROM with $(n_i^\Omega, n_i^\Gamma)=(8,4)$, $N_i^B=100$ HR nodes per subdomain, and varying number of training snapshots for WFPC formulation
    }
    \label{fig:pareto_nsnaps}
\end{figure}

%% file: conclusion.tex

\section{Conclusion}\label{sec:conclusion}
In this work, we detail the first application of NM-ROM with HR to a DD problem.
We extend the DD framework of \cite{CHoang_YChoi_KCarlberg_2021a}
and compute ROMs using the NM-ROM \cite{KLee_KTCarlberg_2020a} approach on each subdomain. 
Furthermore, we apply HR to NM-ROM on each subdomain, which informs the use of shallow, sparse autoencoders, as in \cite{YKim_YChoi_DWidemann_TZohdi_2022a}. 
We detail how to implement an inexact Lagrange-Newton SQP method to solve the constrained least-squares formulation of the ROM, 
where the inexactness comes from using a Gauss-Newton approximation of the residual terms, and from neglecting second-order decoder derivative terms coming from the compatibility constraints. 
We then provide a convergence result for the SQP solver used using standard theory of inexact Newton's method. 
We also provide {\it a priori} and {\it a posteriori} error bounds between the FOM and ROM solutions. 

From our numerical experiments on the 2D steady-state Burgers' equation, 
we showed that using the DD NM-ROM approach significantly decreases the number of required NN parameters per subdomain compared to the monolithic single-domain NM-ROM. 
We also showed that DD NM-ROM achieves an order of magnitude lower relative error than DD LS-ROM in nearly all cases tested. 
Furthermore, in the non-HR case, NM-ROM is faster than LS-ROM in some instances. 
We also saw that NM-ROM is more robust than LS-ROM with respect to changes in subdomain configuration. 
In some cases, the subdomain configuration increased the LS-ROM relative error by an order of magnitude. 
While LS-ROM with HR achieves much higher speedup than NM-ROM with HR, NM-ROM is still the clear winner in terms of ROM accuracy for a given ROM size. 
Moreover, HR allows NM-ROM to gain an extra $15$-$20\times$ speedup compared to the non-HR cases. 
While the speedup is not as drastic as for LS-ROM, these speedup gains for NM-ROM are the highest that have been achieved for NM-ROM to our knowledge. 
Our results indicate that NM-ROM should be the preferred approach for problems where ROM accuracy for a given ROM size is more important than speedup. 

Although NM-ROM performs better than LS-ROM in our experiments, LS-ROM still attains a relatively low relative error. 
This indicates that the model problem considered may still be too benign to expose the advantages that NM-ROM has over LS-ROM, particularly when applied to problems with slowly decaying Kolmogorov $n$-width. 
Therefore, in future work, we plan to apply DD NM-ROM to more complicated problems, including those with slowly decaying Kolmogorov $n$-width, as well as to time-dependent problems. 
Furthermore, the speedup of the DD NM-ROM is highly dependent on the SQP solver used. 
Thus, it will be important to investigate the use of other optimization algorithms for the solution of \eref{eq:dd_lspg_NLP} and examine their effects on computational speedup. 
Other directions for future research include a greedy sampling strategy for the parameter space $\cD$ when choosing which FOM snapshots to compute for NM-ROM training, 
implementing a ``bottom-up" training strategy that uses subdomain snapshots rather than full-domain snapshots for training, applying the DD NM-ROM framework to decomposable or component-based systems, and applying NM-ROM to other DD approaches such as the Schwarz method. 
Finally error estimates based on the first order necessary optimality conditions, as outlined in the last paragraph of Section~\ref{sec:error_analysis}
is also part of future research.